\documentclass{article}
\usepackage{graphicx} 
\usepackage[USenglish]{babel}
\usepackage{bbm}
\usepackage{amsmath}
\usepackage{amsfonts}
\usepackage{amssymb}
\usepackage{amsthm}
\usepackage{bussproofs}
\usepackage{enumerate}
\usepackage{enumitem}
\usepackage{mathtools}
\usepackage[hidelinks]{hyperref}
\usepackage{scrextend}
\usepackage{tikz}
\usepackage[autostyle=false, style=english]{csquotes}
\usepackage{commath}

\makeatletter
\newcommand*{\rn}[1]{%
  \expandafter\@rn\csname c@#1\endcsname%
}

\newcommand*{\@rn}[1]{%
  $\ifcase#1\or(i)\or(ii)\or(iii)\or(iv)\or(v)\or(vi)\or(vii)\or(viii)\or(ix)\or(x)%
    \else\@ctrerr\fi$%
}
\AddEnumerateCounter{\rn}{\@rn}{53.13}
\makeatother

  \newbox\gnBoxA
\newdimen\gnCornerHgt
\setbox\gnBoxA=\hbox{$\ulcorner$}
\global\gnCornerHgt=\ht\gnBoxA
\newdimen\gnArgHgt
\def\Godelnum #1{%
\setbox\gnBoxA=\hbox{$#1$}%
\gnArgHgt=\ht\gnBoxA%
\ifnum     \gnArgHgt<\gnCornerHgt \gnArgHgt=0pt%
\else \advance \gnArgHgt by -\gnCornerHgt%
\fi \raise\gnArgHgt\hbox{$\ulcorner$} \box\gnBoxA %
\raise\gnArgHgt\hbox{$\urcorner$}}

\makeatletter
\newcommand{\pushright}[1]{\ifmeasuring@#1\else\omit\hfill$\displaystyle#1$\fi\ignorespaces}
\newcommand{\pushleft}[1]{\ifmeasuring@#1\else\omit$\displaystyle#1$\hfill\fi\ignorespaces}
\makeatother

\newcommand{\PP}{\mathbb{P}}

\newcommand{\Q}{\dot{\mathbb{Q}}}
\newcommand{\1}{\mathbbm{1}}
\newcommand{\forces}{\Vdash}
\newcommand{\res}{\upharpoonright}

\newcommand{\dotieconcat}[2]{\text{\raisebox{.8ex}{$\smallfrown$}}}

\newcommand{\QQ}{\mathbb{Q}}
\newcommand{\BB}{\mathbb{B}}

\newcommand{\ran}{\mathrm{ran}}
\newcommand{\ZFC}{\mathrm{ZFC}}

\newcommand{\CH}{\mathrm{CH}}

\newcommand{\AD}{\mathrm{AD}}

\newtheoremstyle{nopoint}
  {}{}{\itshape}{}{\bfseries}{}{5pt}{}

\theoremstyle{plain}
\newtheorem{thm}{Theorem}[section]
\newtheorem{prop}[thm]{Proposition}
\newtheorem{lemm}[thm]{Lemma}

\newtheorem{cor}[thm]{Corollary}
\newtheorem{claim}[thm]{Claim}
\newtheorem{sub}[thm]{Subclaim}
\theoremstyle{definition}
\newtheorem{defn}[thm]{Definition}

\newtheorem{rem}[thm]{Remark}

\newtheorem*{que*}{Question}

\theoremstyle{nopoint}
\newtheorem*{lemm*}{Lemma}
\newtheorem*{thm*}{Theorem}
\newtheorem*{rem*}{Remark}
\usepackage{mathabx}
\usepackage{tikz-cd}
\usepackage{tikz}\usetikzlibrary{arrows}

\newcommand{\plus}{+}

\newcommand{\diacol}{\diamondsuit(\omega_1^{{<}\omega})}

\newcommand{\diacolp}{\diamondsuit^{\plus}(\omega_1^{{<}\omega})}

\newcommand{\diacolpgen}[1]{\diamondsuit^{\plus}_{#1}(\omega_1^{{<}\omega})}
\newcommand{\diacolip}{\diamondsuit_{I}^{\plus}(\omega_1^{{<}\omega})}
\newcommand{\diab}{\diamondsuit(\mathbb B)}

\newcommand{\diabgena}[1]{\diamondsuit_{#1}^{\plus}(\BB)}

\newcommand{\good}{\text{-slim}}

\newcommand{\goodit}{\text{\textit{-slim}}}
\newcommand{\fgood}{f\good}
\newcommand{\fgoodit}{f\goodit}

\newcommand{\Col}{\mathrm{Col}}
\newcommand{\Add}{\mathrm{Add}}

\newcommand{\Ult}{\mathrm{Ult}}

\newcommand{\NS}{\mathrm{NS}_{\omega_1}}
\newcommand{\NSf}{\mathrm{NS}_{f}}

\newcommand{\crit}{\mathrm{crit}}
\newcommand{\Pmax}{\mathbb P_{\mathrm{max}}}

\newcommand{\MA}{\mathrm{MA}}
\newcommand{\PFA}{\mathrm{PFA}}
\newcommand{\PFAf}{\PFA(f)}

\newcommand{\FA}{\mathrm{FA}}

\newcommand{\fMM}{\mathrm{MM}(f)}

\newcommand{\fMMpp}{\mathrm{MM}^{++}(f)}

\newcommand{\BP}{\mathbb{B}}
\newcommand{\MM}{\mathrm{MM}}
\newcommand{\MMpp}{\MM^{++}}

\newcommand{\MMTpp}{\MMpp(T)}
\newcommand{\BFA}{\mathrm{BFA}}

\newcommand{\BMM}{\mathrm{BMM}}
\newcommand{\BMMpp}{\BMM^{++}}

\newcommand{\DBFAPsi}{D\text{-}\BFA^{\Psi}_{\vec A}}
\newcommand{\BFAPsi}{\BFA^{\Psi}_{\vec A}}
\newcommand{\BFAPsigen}[1]{#1\text{-}\BFA^{\Psi}}

\newcommand{\Mmax}{\mathbb{F}_{\mathrm{max}}}
\newcommand{\Fmax}{\Mmax}

\newcommand{\lh}{\mathrm{lh}}

\newcommand{\Wstar}{(\ast)}
\newcommand{\stargen}[1]{#1\text{-}\Wstar}

\newcommand{\LofR}{L(\mathbb{R})}

\newcommand{\aspsch}{Asper\'o-Schindler}

\newcommand{\Ord}{\mathrm{Ord}}
\newcommand{\lang}{\mathcal L}
\newcommand{\cert}{\mathfrak{C}}
\newcommand{\ul}{\underline}
\newcommand{\ulc}{\ulcorner}
\newcommand{\urc}{\urcorner}
\newcommand{\ulx}{\ul{x}}
\newcommand{\dmu}{\dot\mu}
\newcommand{\dsigma}{\dot\sigma}

\newcommand{\dn}{\dot n}
\newcommand{\Htwo}{H_{\omega_2}}

\newcommand{\Hkappa}{H_{\kappa}}

\newcommand{\normalformformula}{\ulc\dot N_i\models\varphi(\ul{\gamma_1},\dots, \ul{\gamma_k}, \dot n_1,\dots, \dot n_l, \dot I, \dot a, \dot M_{j_1},\dots, \dot M_{j_m}, \dot \mu_{q_1, r_1}, \dots, \dot \mu_{q_s, r_s}, \dot{\vec M})\urc}
\newcommand{\normalformcert}{\cert=\langle\langle M_i, \mu_{i, j}, N_i, \sigma_{i, j}\mid i\leq j\leq\omega_1\rangle, \langle (k_n, \alpha_n)\mid n<\omega\rangle, \langle \lambda_\xi, X_\xi\mid\xi\in K\rangle\rangle}
\newcommand{\normalformcertprime}{\cert'=\langle\langle M_i', \mu_{i, j}', N_i', \sigma_{i, j}'\mid i\leq j\leq\omega_1\rangle, \langle (k_n', \alpha_n')\mid n<\omega\rangle, \langle \lambda_\rho', X_{\rho}'\mid\rho\in K'\rangle\rangle}
\newcommand{\normalfone}{\ulc\dot\mu_{i, \omega_1}(\dot n)=\ul{x}\urc}
\newcommand{\normalftwo}{\ulc\dmu_{\omega_1,\omega_1}(\ulx)=\ulx\urc}
\newcommand{\normalfthree}{\ulc\dsigma_{i, j}(\dn)=\dot m\urc}

\newcommand{\normalffive}{\ulc\ul{\xi}\mapsto\ul{\nu}\urc}
\newcommand{\normalfsix}{\ulc\ulx\in\dot X_\xi\urc}
\newcommand{\Lim}{\mathrm{Lim}}

\newcommand{\standarditeration}{\langle \PP_\alpha,\Q_\beta\mid\alpha\leq\gamma, \beta<\gamma\rangle}
\newcommand{\Smax}{\mathbb{S}_{\mathrm{max}}}
\newcommand{\SmaxT}{\Smax^{T}}

\newcommand{\Pomo}{\mathcal P(\omega_1)}
\newcommand{\Vmax}{\mathbb{V}_{\mathrm{max}}}

\newcommand{\Qmax}{\mathbb Q_{\mathrm{max}}}
\newcommand{\SRP}{\mathrm{SRP}}

\newcommand{\ADR}{\mathrm{AD}_{\mathbb R}}

\newcommand{\domnum}{\mathfrak{d}}

\newcommand{\Pdom}{P_{\domnum=\aleph_1}}

\newcommand{\MMppdom}{\MMpp(\domnum=\aleph_1)}

\newcommand{\bnum}{\mathfrak{b}}

\newcommand{\QMM}{\mathrm{QM}}
\newcommand{\QM}{\QMM}
\newcommand{\BQMM}{\mathrm{BQM}}
\newcommand{\BQM}{\BQMM}

\newcommand{\Qmaxm}{\Qmax^{-}}

\newcommand{\Qmaxast}{\Qmax^{\ast}}

\newcommand{\Ulamn}{\mathrm{Ulam}}
\newcommand{\boldDelta}{\boldsymbol\Delta}
\newcommand{\RCS}{\mathrm{RCS}}

\newcommand{\PPdia}{\PP^{\diamondsuit}}

\newcommand{\tc}{\mathrm{tc}}

\newcommand{\MRP}{\mathrm{MRP}}

\newcommand{\variant}{variation}
\newcommand{\Pmaxvariant}{\Pmax\text{-variation}}
\newcommand{\Pmaxvariants}{\Pmax\text{-variations}}

\numberwithin{thm}{section}

\newcounter{historycounter}

\theoremstyle{plain}

\newtheorem{blueprintthm}[thm]{First Blueprint Theorem}
\newtheorem{blueprinttwothm}[thm]{Second Blueprint Theorem}

\newtheorem{thmh}[historycounter]{Theorem}
\theoremstyle{definition}
\newtheorem*{defn*}{Definition}
\newtheorem{defnh}[historycounter]{Definition}
\newtheorem{conv}[thm]{Convention}

\newtheorem{queh}[historycounter]{Question}

\title{Forcing $``\mathrm{NS}_{\omega_1}$ is $\omega_1$-dense” From Large Cardinals}
\author{Andreas Lietz\footnote{Institut f\"ur Mathematische Logik und Grundlagenforschung, Universit\"at M\"unster, Einsteinstrasse 62, 48149 M\"unster, FRG. }\hspace{5pt}\footnote{Current address: Institut für Diskrete Mathematik und Geometrie, TU Wien, Wiedner Hauptstrasse 8-10/104, 1040 Wien, AT\\
This paper is part of the authors PhD thesis.}}
\date{October 2023}

\begin{document}

\maketitle

\begin{abstract}
    We answer a question of Woodin by showing that assuming an inaccessible cardinal $\kappa$ which is a limit of ${<}\kappa$-supercompact cardinals exists, there is a stationary set preserving forcing $\PP$ so that 
    $V^\PP\models``\NS\text{ is }\omega_1\text{-dense}"$. We also introduce a new forcing axiom $\QM$, show it is consistent assuming a supercompact limit of supercompact cardinals and prove that it implies $\stargen{\Qmax}$. Consequently, $\QM$ implies ``$\NS$ is $\omega_1$-dense".
\end{abstract}

\section{Introduction}\label{introductionsection}

\subsection{History of ``$\NS$ is $\omega_1$-dense"}
In 1930, Stanislaw Ulam published an influential paper \cite{ulamzurmass} dealing with a question of Stefan Banach generalizing the measure problem of Lebesgue. He proved the following theorem:
\begin{thmh}[Ulam] Suppose $\kappa$ is an uncountable cardinal and there is a $\sigma$-additive real-valued measure on $\kappa$ which
\begin{enumerate}[label=$(\roman*)$]
 \item measures all subsets of $\kappa$ and 
 \item vanishes on points.
 \end{enumerate}
Then there is a weakly inaccessible cardinal $\leq\kappa$.
\end{thmh}
Ulam noticed that he could strengthen his conclusion if he replaces \textit{real-valued} by \textit{0-1-valued}. In more modern terminology, his second result reads:
\begin{thmh}[Ulam] Suppose $\kappa$ is an uncountable cardinal and there is a nonprincipal $\sigma$-complete ultrafilter on $\kappa$. Then there is a (strongly) inaccessible cardinal $\leq\kappa$.
\end{thmh}

These theorems gave birth to what are now known as real-valued measurable cardinals and measurable cardinals respectively. In the interest of having all subsets of some cardinal $\kappa$ be measured in some sense, instead of increasing the size of $\kappa$, it is also possible to increase the number of allowed filters that measure. Henceforth Ulam considered the following question:

\begin{queh} Suppose $\kappa$ is an uncountable cardinal below the least inaccessible. What is the smallest possible size of a family $\mathcal F$ of $\sigma$-closed nonprincipal filters on $\kappa$ so that every subset of $\kappa$ is measured by some filter in $\mathcal F$?
\end{queh}
Let us call the cardinal in question the Ulam number of $\kappa$, $\Ulamn(\kappa)$. Ulam's second theorem above can be rephrased as ``$\Ulamn(\kappa)>1$". Indeed, Ulam proved in unpublished work that $\Ulamn(\kappa)\geq\omega$. At some point, Ulam proposed this question to Paul Erd\H{o}s, who, together with Leonidas Alaoglu, improved Ulam's result to ``$\Ulamn(\kappa)\geq\omega_1$" \cite{erdosremarks}. The problem, this time in the special case $\kappa=\omega_1$, was apparently revitalized by appearing  in the 1971 collection of unsolved problems in set theory popularized by Erd\H{o}s and Hajnal \cite{erdoshajnal}: Shortly after, Karel Prikry \cite{prikry72} produced a model in which $\Ulamn(\omega_1)=2^{\omega_1}=\omega_2$, and did the same again with a different method in \cite{prikry76}. \\
A critical step towards a model in which $\Ulamn(\omega_1)=\omega_1$ was taken by Alan D. Taylor: Building on earlier work of Baumgartner-Hajnal-Mat\'e \cite{bhm}, Taylor provided \cite{taylor79} an impressive amount of statements equivalent to a natural strengthening of $``\Ulamn(\omega_1)=\omega_1"$, here is a shortened list.

\begin{thmh}[Taylor] 
The following are equivalent:
\begin{enumerate}[label=\rn*]
\item There is a family of normal filters witnessing $\Ulamn(\omega_1)=\omega_1$.
\item There is a $\sigma$-closed uniform $\omega_1$-dense ideal on $\omega_1$.
\item\label{taylor3cond}There is a normal uniform $\omega_1$-dense ideal on $\omega_1$.
\end{enumerate}
\end{thmh}
The formulation \ref{taylor3cond} is much better suited for set-theoretical arguments. We also mention that Taylor proved that all the above statements fail under $\MA_{\omega_1}$.\\
Thus what remains of Ulam's original question was reduced to: Is the existence of a normal uniform  $\omega_1$-dense ideal on $\omega_1$ consistent with $\ZFC$? This was answered positively by W. Hugh Woodin in three different ways. The first was by forcing over a model of $\AD_{\mathbb R}$+$``\Theta\text{ is regular}"$, already in the fall of 1978. (unpublished). At that time, this theory was not yet known to be consistent relative to large cardinals. Naturally, somewhat later he did so from large cardinals:
\begin{thmh}[Woodin, unpublished\footnote{A proof can be found in Foreman's handbook article \cite{foremanhandbook}.}] Assume there is an almost-huge cardinal $\kappa$. Then there is a forcing extension in which there is a normal uniform $\omega_1$-dense ideal on $\omega_1=\kappa$.
\end{thmh}
 This finally resolved the question relative to large cardinals. But can the canonical normal uniform ideal, namely $\NS$, have this property? It is known that $\NS$ behaves a little different in this context.

\begin{thmh}[Shelah, \cite{shelahdense}] If $\NS$ is $\omega_1$-dense then $2^\omega=2^{\omega_1}$. In particular $\CH$ fails.
\end{thmh}
This is not true for other normal uniform ideals on $\omega_1$, for example $\CH$ holds in the model Woodin constructs from an almost huge cardinal. One can also ask about the exact consistency strength of the existence of such a normal uniform $\omega_1$-dense ideal on $\omega_1$. Both these questions were answered in subsequent work by Woodin, building on his $\Pmax$-technique.
\begin{thmh}[Woodin, {\cite[Corollary 6.150]{woodinbook}}] The following theories are equiconsistent:
\begin{enumerate}[label=\rn*]
\item $\ZFC+``\text{There are infinitely many Woodin cardinals.}"$
\item $\ZFC+``\NS\text{ is }\omega_1\text{-dense.}"$
\item $\ZFC+``\text{There is a normal uniform }\omega_1\text{-dense ideal on }\omega_1."$
\end{enumerate}
\end{thmh}

The direction $(iii)\Rightarrow (i)$ makes use of Woodin's core model induction technique, the argument is unpublished. We refer the interested reader to \cite{schindlersteel} where part of this is proven. Woodin's method for $(i)\Rightarrow(ii)$ is by forcing over $\LofR$, assuming $\AD$ there, with the $\Pmaxvariant$ $\Qmax$. This approach has one downside: It is a forcing construction over a canonical determinacy model. $\LofR$ can be replaced by larger determinacy models, but $\Qmax$ relies on a good understanding of the model in question. In practice, this is akin to an anti large cardinal assumption and leaves open questions along the lines of: Is ``$\NS$ is $\omega_1$-dense" consistent together with all natural large cardinals, e.g.~supercompact cardinals? Is it consistent with powerful combinatorial principles, for example $\SRP$?\\
Woodin's original motivation for these results was in fact the question of generic large cardinal properties of $\omega_1$: For example $\omega_1$ is not measurable by Ulam's theorem, but there can be a generic extension of $V$ with an elementary embedding $j:V\rightarrow M$ with transitive $M$ and critical point $\omega_1^V$. This leads to precipitous ideals on $\omega_1$.

\begin{defnh}
A uniform ideal $I$ on $\omega_1$ is precipitous if, whenever $G$ is generic for $(\Pomo/I)^+$ then $\Ult(V, U_G)$ is wellfounded\footnote{$U_G$ denotes the $V$-ultrafilter induced by $G$.}.
\end{defnh}

The existence of an $\omega_1$-dense ideal is a much stronger assumption than the existence of a precipitous ideal. There is a natural well-studied intermediate principle.

\begin{defnh}
A uniform ideal $I$ on $\omega_1$ is saturated if $(\Pomo/I)^+$ is $\omega_2$-c.c..
\end{defnh}

Here is a short history of similar result for these principles:
\begin{enumerate}[label=\rn*]
\item Mitchell forces a precipitous ideal on $\omega_1$ from a measurable in the mid 70s, see \cite{jmmp}.
\item Magidor forces ``$\NS$ is precipitous" from a measurable, published in \cite{jmmp}.
\item Kunen \cite{kunsat} forces a saturated ideal on $\omega_1$ from a huge cardinal, which he invented for this purpose.
\item\label{svwcond} Steel-Van Wesep \cite{svw} force ``$\NS$ is saturated" over a model of\footnote{Woodin \cite{woodinjustad} subsequently reduced the assumption to just $\AD$ .} $\AD+\mathrm{AC}_{\mathbb R}$.
\item\label{fmscond} Foreman-Magidor-Shelah \cite{fmsmmpp} force ``$\NS$ is saturated" from a supercompact with semiproper forcing. Later reduced to one Woodin cardinal by Shelah\footnote{The main ideas for the argument are in \cite[XVI]{shelahbook}, a write-up by Schindler can be found in \cite{nssat}.}.
\end{enumerate}

Woodin's results continue this line of research for $\omega_1$-dense ideals. But the analog of the step from \ref{svwcond} to \ref{fmscond} for $\omega_1$-dense ideals was missing. Accordingly, Woodin posed the following question:
\begin{queh}[Woodin, {\cite[Chapter 11 Question 18 b)]{woodinbookold}}]\label{woodinsquestion}
Assuming the existence of some large cardinal: Must there exist some semiproper partial order $\PP$ such that 
$$V^\PP\models``\NS\text{ is }\omega_1\text{-dense}"\ \text{?}$$
\end{queh}

We will answer this positively in this thesis.

\begin{thmh}
Assume there is an inaccessible cardinal $\kappa$ which is the limit of cardinals which are ${<}\kappa$-supercompact. Then there is a stationary set preserving forcing $\PP$ so that 
$$V^\PP\models``\NS\text{ is }\omega_1\text{-dense}".$$
If there is an additional supercompact cardinal below $\kappa$, we can find such $\PP$ that is semiproper.
\end{thmh}

On a different note, there has been significant interest recently into the possible $\boldDelta_1$-definability of $\NS$ (with parameters), in particular in the presence of forcing axioms. Note that $\NS$ is trivially $\Sigma_1(\omega_1)$-definable, but it is independent of $\ZFC$ whether $\NS$ is $\boldsymbol \Pi_1$-definable. Hoffelner-Larson-Schindler-Wu \cite{hlsw} show:

\begin{enumerate}[label=\rn*]
 \item If $\BMM$ holds and there is a Woodin cardinal then $\NS$ is not $\boldDelta_1$-definable.
 \item If $\Wstar$ holds then $\NS$ is not $\boldDelta_1$-definable.
 \item\label{MMppnotboldDeltaitem} Thus by \aspsch\ \cite{aspsch}, if $\MMpp$ holds, $\NS$ is not $\boldDelta_1$-definable.
 \item It is consistent relative to large cardinals that $\mathrm{BPFA}$ holds and $\NS$ is $\boldDelta_1$-definable.
 \end{enumerate} 
 
 There is also a forthcoming paper by Ralf Schindler and Xiuyuan Sun \cite{schindlersun} showing  that in \ref{MMppnotboldDeltaitem}, $\MMpp$ can be relaxed to $\MM$. \\
If $\NS$ is $\omega_1$-dense then $\NS$ is automatically $\boldDelta_1$-definable: If $\mathcal S$ is a set of $\omega_1$-many stationary sets witnessing the density, then $T\subseteq\omega_1$ is stationary iff 
$$\exists C\subseteq\omega_1\text{ a club, }\exists S\in\mathcal S\  C\cap S\subseteq T.$$
This was first observed by Friedman-Wu-Zdomskyy \cite{fwz}. In this context, two interesting points arise from our results here: First, we isolate for the first time a forcing axiom which \textit{implies} ``$\NS$ is $\boldDelta_1$-definable". Second, it is well known that many of the structural consequences of $\MM$ follow already from $\SRP$, for example ``$\NS$ is saturated", $2^\omega=\omega_2$, $\mathrm{SCH}$, etc. In contrast, in the result of Schindler-Sun, $\MM$ cannot be replaced by $\SRP$: If appropriate large cardinals are consistent, then so is $\SRP$ together with ``$\NS$ is $\boldDelta_1$-definable". 

\subsection*{Acknowledgements}
I want to thank my supervisor Ralf Schindler for his continued support that led to the results presented here. Most importantly, for relentlessly asking me to force ``$\NS$ is $\omega_1$-dense" from large cardinals from the beginning of my PhD on. It should be obvious to the reader that this work would not have been possible without Ralf's great supervision. \\
Moreover, I am grateful to David Asper\'o for a number of very helpful remarks and suggestions.\\
Funded by the Deutsche Forschungsgemeinschaft (DFG, German Research
Foundation) under Germany’s Excellence Strategy EXC 2044 390685587, Mathematics M\"unster:
Dynamics - Geometry - Structure.

\section{Notation}

First, we fix some notation. We will extensively deal with countable elementary substructures $X\prec H_\theta$ for large regular $\theta$. We will make frequent use of the following notation:

\begin{defn}
Suppose $X$ is any extensional set.
\begin{enumerate}[label=\rn*]
\item $M_X$ denotes the transitive isomorph of $X$\index{MX@$M_X$}.
\item  $\pi_X\colon M_X\rightarrow X$ denotes the inverse collapse\index{piX@$\pi_X$}.
\item\label{deltadefn} $\delta^X\coloneqq\omega_1\cap X$\index{deltaX@$\delta^X$}.
\end{enumerate}
\end{defn} 

In almost all cases, we will apply this definition to a countable elementary substructure $X\prec H_\theta$ for some uncountable cardinal $\theta$. In some cases, the $X$ we care about lives in a generic extension of $V$, even though it is a substructure of $H_\theta^V$. In that case, $\delta^X$ will always mean $X\cap \omega_1^V$.\medskip

We will also sometimes make use of the following convention in order to ``unclutter" arguments.

\begin{conv}
If $X\prec H_\theta$ is an elementary substructure and some object $a$ has been defined before and $a\in X$ then we denote $\pi_X^{-1}(a)$ by $\bar a$\index{abar@$\bar a$ (short for $\pi_X^{-1}(a)$)}. 
\end{conv}

We will make use of this notation only if it is unambiguous.

\begin{defn}
If $X, Y$ are sets then $X\sqsubseteq Y$\index{<squaresubset@$\sqsubseteq$} holds just in case 
\begin{enumerate}[label=\rn*]
\item $X\subseteq Y$ and
\item $\delta^X=\delta^Y$.
\end{enumerate}
\end{defn}

We use the following notions of clubs and stationarity on $[H_\theta]^\omega$:

\begin{defn}
Suppose $A$ is an uncountable set. 
\begin{enumerate}[label=\rn*]
\item $[A]^\omega$\index{[A]omega@$[A]^\omega$} is the set of countable subsets of $A$.
\item $\mathcal C\subseteq [A]^\omega$ is a club in $[A]^\omega$\index{[A]omega@$[A]^\omega$!club in} if
\begin{enumerate}[label=$\alph*)$]
\item for any $X\in[A]^\omega$ there is a $Y\in \mathcal C$ with $X\subseteq Y$ and
\item if $\langle Y_n\mid n<\omega\rangle$ is a $\subseteq$-increasing sequence of sets in $\mathcal C$ then $\bigcup_{n<\omega} Y_n\in \mathcal C$.
\end{enumerate}
\item $\mathcal S\subseteq[A]^\omega$ is stationary in $[A]^\omega$\index{[A]omega@$[A]^\omega$!stationary in} if $\mathcal S\cap\mathcal C\neq\emptyset$ for any club $\mathcal C$ in $[A]^\omega$.
\end{enumerate}
\end{defn}

Next, we explain our notation for forcing iterations.
\begin{defn} Suppose $\PP=\standarditeration$ is an iteration and $\beta\leq\gamma$. We consider elements of $\PP$ as functions of domain (or length) $\gamma$.
\begin{enumerate}[label=$(\roman*)$]
\item If $p\in\PP_\beta$ then $\lh(p)=\beta$.
\item If $G$ is $\PP$-generic then $G_\beta$ denotes the restriction of $G$ to $\PP_\beta$, i.e.
$$G_\beta=\{p\res\beta\mid p\in G\}.$$
Moreover, $\dot G_\beta$ is the canonical $\PP$-name for $G_\beta$.
\item If $G_\beta$ is $\PP_\beta$-generic then $\PP_{\beta,\gamma}$ denotes (by slight abuse of notation) the remainder of the iteration, that is 
$$\PP_{\beta,\gamma}=\{p\in\PP_\gamma\mid p\res\beta\in G_\beta\}.$$
$\dot\PP_{\beta,\gamma}$ denotes a name for $\PP_{\beta,\gamma}$ in $V$.
\item If $G$ is $\PP$-generic and $\alpha<\beta$ then $G_{\alpha,\beta}$ denotes the projection of $G$ onto $\PP_{\alpha,\beta}$.
\end{enumerate}
\end{defn}

There will be a number of instances were we need a structure to satisfy a sufficiently large fragment of $\ZFC$. For completeness, we make this precise.

\begin{defn}
\textit{Sufficiently much of }$\ZFC$\index{sufficiently much of ZFC@Sufficiently much of $\ZFC$} is the fragment $\ZFC^{-}+``\omega_1\text{ exists}"$. Here, $\ZFC^{-}$ is $\ZFC$ without the powerset axiom and with the collection scheme instead of the replacement scheme.
\end{defn}

\section{$\diacol$ and $\diacolp$}\label{diacolpnddiacolpsection}

We introduce the central combinatorial principle which is due to Woodin. The relevancy is motivated by the following observation: If $\NS$ is $\omega_1$-dense, then there is a dense embedding
$$\eta\colon \Col(\omega,\omega_1)\rightarrow \mathcal (P(\omega_1)/\NS)^+.$$
We aim to force a forcing axiom that implies this. As usual, the forcing achieving this is an iteration $\PP$ of some large cardinal length $\kappa$ which preserves $\omega_1$ and iterates forcings of size ${<}\kappa$ with countable support-style supports. $\PP$ will thus be $\kappa$-c.c.~and this means that some ``representation"
$$\eta_0\colon \Col(\omega,\omega_1)\rightarrow \NS^+$$
of $\eta$ exists already in an intermediate extension. By ``representation" we mean that in $V^\PP$, 
$$[\eta_0(p)]_{\NS}=\eta(p)$$
for all $p\in\Col(\omega,\omega_1)$\footnote{For $S\subseteq\omega_1$ and $I$ an ideal on $\omega_1$, $[S]_I$ denotes the equivalence class of $S$ induced by the equivalence relation $T\sim T'\Leftrightarrow T\triangle T'\in I$.}. With this in mind, one should isolate the relevant $\Pi_1$-properties which $\eta_0$ possesses in $V^\PP$. Consequently, $\eta_0$ satisfies these properties in the intermediate extension. It is hopefully easier to first force an object with this $\Pi_1$-fragment and we should subsequently only force with partial orders that preserve this property. This is exactly what we will do. The relevant combinatorial properties are $\diacol$ and $\diacolp$ and were already isolated by Woodin in his study of $\Qmax$ \cite[Section 6.2]{woodinbook}. We remark that the definition we use here is slightly stronger than Woodin's original principle in a technical way that turns out to be convenient for our purposes. Most results in this Section are essentially due to Woodin and proven in \cite[Section 6.2]{woodinbook}.

\begin{defn}
\begin{enumerate}[label=\rn*]
\item We say that $f$ \textit{guesses $\Col(\omega,\omega_1)$-filters}\index{f guesses B filters@$f$ guesses $\Col(\omega,\omega_1)$-filters} if $f$ is a function 
$$f\colon\omega_1\rightarrow H_{\omega_1}$$
and for all $\alpha<\omega_1$, $f(\alpha)$ is a $\Col(\omega,\omega_1)\cap\alpha$-filter\footnote{We consider the empty set to be a filter.}.
\item Suppose $\theta\geq\omega_2$ is regular and $X\prec H_\theta$ is an elementary substructure. We say $X$ is $\fgoodit$\index{Set!f slim@$\fgood$}\footnote{We use the adjective ``slim" for the following reason: An $\fgood$ $X\prec H_\theta$ cannot be too fat compared to its height below $\omega_1$, i.e.~$\delta^X$.  If $X\sqsubseteq Y\prec H_\theta$ and $Y$ is $\fgood$ then $X$ is $\fgood$ as well, but the converse can fail.} if
\begin{enumerate}[label=$(X.\roman*)$]
 \item $X$ is countable,
 \item $f, \Col(\omega,\omega_1)\in X$ and
 \item $f(\delta^X)$ is $\Col(\omega,\omega_1)\cap\delta^X$-generic over $M_X$.
 \end{enumerate} 
\end{enumerate}
\end{defn}

\begin{defn}\label{diadefn}
$\diacol$\index{diamondb@$\diacol$} states that there is a function $f$ so that
\begin{enumerate}[label=\rn*]
 \item $f$ guesses $\Col(\omega,\omega_1)$-filters and
 \item for any $b\in\Col(\omega,\omega_1)$ and regular $\theta\geq\omega_2$
 \begin{align*}
 \{X\prec H_\theta\mid X \text{ is }\fgood\wedge &b\in f(\delta^X)\}
 \end{align*}
 is stationary in $[H_{\theta}]^\omega$.
 \end{enumerate} 
$\diacolp$\index{diamondb0plus@$\diacolp$} is the strengthening of $\diacol$ where $(ii)$ is replaced by:
\begin{enumerate}
\item[$(ii)^\plus$] For any regular $\theta\geq\omega_2$
$$\{X\prec H_\theta\mid X \text{ is }\fgood\}$$
contains a club of $[H_\theta]^{\omega}$. Moreover, for any $b\in\Col(\omega,\omega_1)$ 
$$\{\alpha<\omega_1\mid b\in f(\alpha)\}$$
is stationary.
\end{enumerate}
We say that $f$ \textit{witnesses} $\diacol$, $\diacolp$ respectively. 
\end{defn}

We introduce some convenient shorthand notation.
\begin{defn}
If $f$ witnesses $\diacol$ and $b\in\Col(\omega,\omega_1)$ then
$$S^f_b\coloneqq \{\alpha<\omega_1\mid b\in f(\alpha)\}.$$
If $f$ is clear from context we will sometimes omit the superscript $f$.
\end{defn}

Note that if $f$ witnesses $\diacol$, then $S^f_b$ is stationary for all $b\in\Col(\omega,\omega_1)$. This is made explicit for $\diacolp$. This is exactly the technical strengthening over Woodin's original definition of $\diacol, \diacolp$.
\begin{defn}
If $f$ witnesses $\diacol$ and $\PP$ is a forcing, we say that $\PP$ \textit{preserves} $f$\index{Forcing!preserves f@preserves $f$} if whenever $G$ is $\PP$-generic then $f$ witnesses $\diacol$ in $V[G]$.
\end{defn}
We remark that if $f$ witnesses $\diacolp$ then ``$\PP$ preserves $f$" still only means that $f$ witnesses $\diacol$ in $V^\PP$.\medskip

Next, we define a variant of stationary sets related to a witness of $\diacol$. Suppose $\theta\geq\omega_2$ is regular. Then $S\subseteq\omega_1$ is stationary iff for any club $\mathcal C\subseteq[H_\theta]^\omega$, there is some $X\in\mathcal C$ with $\delta^X\in S$. $f$-stationarity results from restricting to $f$-slim $X\prec H_\theta$ only.

\begin{defn}
Suppose $f$ witnesses $\diacol$.
\begin{enumerate}[label=\rn*]
\item A subset $S\subseteq\omega_1$ is $f$\textit{-stationary}\index{f stationary@$f$-stationary!in omega one@in $\omega_1$} iff whenever $\theta\geq\omega_2$ is regular and $\mathcal C\subseteq[H_\theta]^\omega$ is club then there is some $\fgood$ $X\in\mathcal C$ with $\delta^X\in S$.
\item A forcing $\PP$ \textit{preserves} $f$\textit{-stationary sets}\index{Forcing!preserves fstationary sets@preserves $f$-stationary sets} iff any $f$-stationary set is still $f$-stationary in $V^\PP$.
\end{enumerate}
\end{defn}

Note that all $f$-stationary sets are stationary, but the converse might fail. $f$-stationary sets are the correct replacement of stationary set in our context. \\

We mention a few basic facts about $\diacol$ and $\diacolp$ which are all essentially due to Woodin \cite{woodinbook}, although he did not use the notion of $f$-stationary sets explicitly.

\begin{prop}\label{fstationaryequivformulationprop}
Suppose $f$ guesses $\Col(\omega,\omega_1)$-filters. The following are equivalent for any set $S\subseteq\omega_1$:
\begin{enumerate}[label=\rn*]
\item\label{fstationaryequivformulationcond1} $S$ is $f$-stationary.
\item\label{fstationaryequivformulationcond2} Whenever $\langle D_\alpha\mid\alpha<\omega_1\rangle$ is a sequence of dense subsets of $\Col(\omega,\omega_1)$, the set 
$$\{\alpha\in S\mid \forall\beta<\alpha\ f(\alpha)\cap D_\beta\neq\emptyset\}$$
is stationary.
\end{enumerate}
\end{prop}

\begin{prop}\label{diacolequivprop}
Suppose $f$ guesses $\Col(\omega,\omega_1)$-filters. The following are equivalent:
\begin{enumerate}[label=\rn*]
\item\label{diacolequivcond1} $f$ witnesses $\diacol$.
\item\label{diacolequivcond2} $S^f_b$ is $f$-stationary for all $b\in\Col(\omega,\omega_1)$.
\item\label{diacolequivcond3} For any $b\in\Col(\omega,\omega_1)$ and sequence $\langle D_\alpha\mid\alpha<\omega_1\rangle$ of dense subsets of $\Col(\omega,\omega_1)$, 
$$\{\alpha\in S^f_b\mid\forall\beta<\alpha\ f(\alpha)\cap D_\beta\neq\emptyset\}$$
is stationary.
\end{enumerate}
\end{prop}

\begin{proof}
The equivalence of \ref{diacolequivcond1} and \ref{diacolequivcond2} follows from the definitions. \ref{diacolequivcond2} and \ref{diacolequivcond3} are equivalent by the equivalent formulation of $f$-stationarity provided by Proposition \ref{fstationaryequivformulationprop}.
\end{proof}

We mention a handy corollary.

\begin{cor}\label{preservefcor}
Suppose $f$ witnesses $\diacol$. Any forcing preserving $f$-stationary sets preserves $f$.
\end{cor}

\begin{prop}\label{diacolpequivprop}
Suppose $f$ guesses $\Col(\omega,\omega_1)$-filters. The following are equivalent:
\begin{enumerate}[label=\rn*]
\item\label{diacolpequivcond1} $f$ witnesses $\diacolp$.
\item\label{diacolpequivcond2} For any $b\in\Col(\omega,\omega_1)$, $S^f_b$ is stationary and all stationary sets are $f$-stationary.
\item\label{diacolpequivcond3} If $D$ is dense in $\Col(\omega,\omega_1)$ then 
$$\{\alpha<\omega_1\mid\ f(\alpha)\cap D\neq\emptyset\}$$
contains a club and for all $b\in\Col(\omega,\omega_1)$, $S^f_b$ is stationary.
\item\label{diacolpequivcond4} All countable $X\prec H_\theta$ with $f\in X$ and $\theta\geq \omega_2$ regular are $\fgood$ and moreover for all $b\in \Col(\omega,\omega_1)$, $S^f_b$ is stationary.
\end{enumerate}
\end{prop}

We will now give a natural equivalent formulation of $\diacolp$.
Witnesses of $\diacolp$ are simply codes for regular embeddings\footnote{Regular embeddings, also known as complete embeddings, are embeddings between partial orders which preserve maximal antichains.} of $\Col(\omega,\omega_1)$ \mbox{into $\NS^+$}.

\begin{lemm}\label{diacomplemm}
The following are equivalent:
\begin{enumerate}[label=\rn*]
\item\label{diacompcond1} $\diacolp$.
\item\label{diacompcond2} There is a regular embedding $\eta\colon\Col(\omega,\omega_1)\rightarrow(\Pomo/\NS)^+$.
\end{enumerate}
\end{lemm}

The argument above suggests the following definition.
\begin{defn}
Suppose $f$ witnesses $\diacol$. We define 
$$\eta_f\colon \Col(\omega,\omega_1)\rightarrow (\Pomo/\NS)^+$$
by $b\mapsto[S^f_b]_{\NS}$ and call $\eta_f$ the \textit{embedding associated to} $f$\index{etaf@$\eta_f$ (the embedding associated to $f$)}.
\end{defn}

\begin{defn}
Suppose $f$ witnesses $\diacol$. $\NSf$\index{NSf@$\NSf$ (the ideal of $f$-nonstationary sets)} is the ideal of $f$-nonstationary sets, that is 
$$\NSf=\{N\subseteq\omega_1\mid N\text{ is not }f\text{-stationary}\}.$$
\end{defn}

\begin{lemm}\label{nsfnormalideallemm}
Suppose $f$ witnesses $\diacol$. $\NSf$ is a normal uniform ideal.
\end{lemm}

To each witness $f$ of $\diacol$, one can associate a version of semiproperness.

\begin{defn}
\begin{enumerate}[label=\rn*]
\item Let $\theta$ be a sufficiently large regular cardinal and $X\prec H_\theta$ $\fgood$ with $\PP\in X$. A condition $q\in\PP$ is $(X, \Col(\omega,\omega_1), f)$-semigeneric if $q$ is $(X, \Col(\omega,\omega_1)$-semigeneric and
$$q\forces``\check X[\dot G]\text{ is }\fgood"$$
\item $\PP$ is $f$-semiproper if for any sufficiently large regular $\theta$ and any $\fgood$ $X\prec H_\theta$ with $\PP\in X$ as well as all $p\in\PP\cap X$ there is $q\leq p$ that is $(X, \PP, f)$-semigeneric.
\end{enumerate}
\end{defn}

An $f$-semiproper forcing $\PP$ need not preserve stationary sets, however it will preserve $f$-stationary sets as $f$-stationary sets and hence $f$ will still witness $\diacol$ in $V^\PP$.

However, just as for semiproperness, $f$-semiproper forcings can be iterated.

\begin{thm}[Lietz, \cite{lietzIteration}]\label{fsemiproperiterationtheoremotherpaper}
Suppose $f$ witnesses $\diacol$. Any nice iteration of $f$-semiproper forcings is $f$-semiproper.
\end{thm}

We refer to \cite{miyamotosemiproper} for the definition of nice iterations. For all intents and purposes, nice iterations can be replaced by RCS iterations here.
\section{A Forcing Axiom That Implies ``$\NS$ Is $\omega_1$-Dense"}\label{qmmsection}

We formulate a forcing axiom that implies $\stargen{\Qmax}$. We go on and show that it can be forced from a supercompact limit of supercompact cardinals.

\subsection{$\mathrm{Q}$-Maximum}

\begin{defn}
$\mathrm{Q}$-Maximum\index{Q Maximum@$\mathrm{Q}$-Maximum}, denoted $\QMM$, holds if there is a witness $f$ of $\diacol$ and $\FA(\Gamma)$ holds where 
$$\Gamma=\{\PP\mid\PP\text{ preserves }f\}=\{\PP\mid \forall p\in\Col(\omega,\omega_1)\ S^f_p\in (\NSf^+)^{V^\PP}\}.$$
\end{defn}

 We remark that the consistency of $\QMM$ is a subtle matter, for example any ``$++$"-version of $\QM$ would be inconsistent. It is however relevant to our purposes.

\begin{lemm}\label{NSdensefromQMlemm}
If $f$ witnesses $\QM$ then $\eta_f$ is a dense embedding. In particular, $\NS$ is $\omega_1$-dense.
\end{lemm}
\begin{proof}
Suppose $S\subseteq\omega_1$ is so that 
$$S^f_p\nsubseteq S\mod\NS$$
for all $p\in\Col(\omega,\omega_1)$. Let $\PP$ be the canonical forcing that shoots a club through $T\coloneqq \omega_1- S$. That is $p\in\PP$ iff $p\subseteq T$ is closed and bounded and $p\leq q$ iff $q$ is an initial segment of $p$. 
\begin{claim}
$\PP$ preserves $f$.
\end{claim}
\begin{proof}
Let $b\in\Col(\omega,\omega_1)$, we have to show that $S^f_b$ is $f$-stationary in $V^\PP$. Let $p\in\PP$, $\dot C$ a $\PP$-name for a club and $\langle \dot D_i\mid i<\omega_1\rangle$ a sequence of $\PP$-names for dense subsets of $\Col(\omega,\omega_1)$. We will find $q\leq p$ with 
\begin{equation*}\label{NSdensefromQMlemmtag}\tag{$q$}
q\forces\exists\xi\in\dot C\cap S^{\check f}_{\check b}\ \forall i<\xi\ \check f(\xi)\cap \dot D_i\neq\emptyset.
\end{equation*}
Let $\theta$ be large and regular. Note that $\fMM$ holds and hence $f$ witnesses $\diacolp$. As $T\cap S^f_b$ is stationary, $T\cap S^f_b$ is $f$-stationary and we can find some $X\prec H_\theta$ with
\begin{enumerate}[label=$(X.\roman*)$]
\item $X$ is $\fgood$,
\item $\PP, p, \dot C,\langle \dot D_i\mid i<\omega_1\rangle\in X$ and
\item $\delta^X\in T\cap S^f_b$.
\end{enumerate}

Now find a decreasing sequence $\langle p_n\mid n<\omega\rangle$ with 
\begin{enumerate}[label=$(\vec p.\roman*)$]
\item $p_0=p$,
\item $\forall n<\omega\ p_n\in\PP\cap X$ and
\item for all $D\in M_X[f(\delta^X)]$ dense in $\pi_X^{_1}(\PP)$, there is $n<\omega$ with $p_n\in \pi_X[D]$.
\end{enumerate}
Set $q=\bigcup_{n<\omega}p_n\cup\{\delta^X\}$ and note that $q\in\PP$ as $\delta^X\in T$. It is clear that $q$ is $(X, \PP, f)$-semigeneric so that if $G$ is $\PP$-generic with $q\in G$ then 
$$\forall i<\delta^X=\delta^{X[G]}\ f(\delta^X)\cap \dot D_i^G\neq\emptyset$$
as well as $\delta^{X}\in \dot C^G\cap S^f_b$. Thus $q$ indeed satisfies (\ref{NSdensefromQMlemmtag}).
\end{proof}
Thus $\FA(\{\PP\})$ holds. This implies that if $G$ is $\PP$-generic then 
$$(\Htwo;\in)^V\prec_{\Sigma_1}(\Htwo;\in)^{V[G]}$$
and as $T$ contains a club in $V[G]$, this must already be true in $V$. This means $S$ is nonstationary which is what we had to show.
\end{proof}

We will prove eventually that $\QM$ can be forced from large cardinals.

\begin{thm}
Suppose there is a supercompact limit of supercompact cardinals. Then $\QM$ holds in a  forcing extension by stationary set preserving forcing.
\end{thm}

\subsection{$\mathrm{Q}$-iterations}
Our strategy to force $\QM$, or ``$\NS$ is $\omega_1$-dense" for that matter has to make use of an iteration theorem that allows us to iterate essentially arbitrary $f$-preserving forcings for a witness $f$ of $\diacol$ so that $f$ is preserved. We have proven in \cite{lietzIteration} a more general version of the following theorem.

\begin{thm}[Lietz, \cite{lietzIteration}]\label{iterationtheoremotherpaper}
    Suppose $f$ witnesses $\diacol$ and $\PP=\standarditeration$ is a nice iteration of $f$-preserving forcings. Suppose that 
    \begin{enumerate}[label=$(\PP.\roman*)$]
        \item\label{itm:srprestriction} $\forces_{\PP_{\alpha+2}}\SRP$ for all $\alpha+2\leq\gamma$ and
        \item\label{itm:harshrestriction} $\forces_{\PP_\alpha}``\Q_{\alpha}\text{ preserves }f\text{-stationary sets from } \bigcup_{\beta<\alpha}V[\dot G_\beta]$.
    \end{enumerate}
    Then $\PP$ preserves $f$.
\end{thm}
The immediate problem is that \ref{itm:harshrestriction} puts an undesired additional requirement on the forcings we want to iterate. Luckily, there is a small trick to still get away with this: Note that an $f$-preserving forcing must preserve the $f$-stationarity of the set $S^f_b$ for $b\in\Col(\omega,\omega_1)$. Suppose that at all successor steps, we arrange that any $f$-stationary set from the previous extension contains some $S^f_b$ modulo a non-stationary set. Now at a limit step, suddenly every $f$-preserving forcing will satisfy requirement \ref{itm:harshrestriction}.  As \ref{itm:srprestriction} does not ask anything of us at limit steps either, we are free to use any $f$-preserving forcing we desire at limit steps.
\begin{defn}
Suppose $f$ witnesses $\diacol$. We say that a forcing $\PP$ \textit{freezes} \index{Forcing!freezes NS along f@freezes $\NS$ along $f$}$\NS$ along $f$ if for any $\PP$-generic $G$ we have
\begin{enumerate}[label=\rn*]
 \item $f$ witnesses $\diacol$ in $V[G]$ and
 \item for any $S\in\mathcal P(\omega_1)\cap V$, we either have $S\in \NS^{V[G]}$ or there is $p\in\Col(\omega,\omega_1)$ with $S^f_p\subseteq S\mod\NS^{V[G]}$.
 \end{enumerate} 
\end{defn}

 We hope to have motivated the following definition.

\begin{defn}
Suppose $f$ witnesses $\diacol$. A $Q$\textit{-iteration}\index{Q iteration@$\mathrm{Q}$-iteration} (w.r.t. $f$) is a nice iteration $\PP=\langle \PP_\alpha,\Q_\beta\mid\alpha\leq\gamma, \beta<\gamma\rangle$ which satisfies
\begin{enumerate}[label=\rn*]
\item $\forces_{\PP_{\alpha}}``\Q_\alpha\text{ preserves }f"$,
\item\label{Qiterationddaggercond} $\forces_{\PP_{\alpha+1}}(\ddagger)$ and
\item if $\alpha+1<\gamma$ then $\forces_{\PP_{\alpha+1}}``\Q_{\alpha+1}\text{ freezes }\NS\text{ along }f"$
\end{enumerate}
for all $\alpha<\gamma$.
\end{defn}

As an immediate consequence of Theorem \ref{iterationtheoremotherpaper}, we get the following ``iteration theorem".

\begin{thm}\label{Qiterationthm}
Suppose $f$ witnesses $\diacol$. All $Q$-iterations (w.r.t. $f$) preserve $f$.
\end{thm}

Provided we find enough forcings which allow us to continue a $Q$-iteration up to a supercompact cardinal, we are able to force $\QM$. To be precise, we will prove the following two lemmas.

\begin{lemm}\label{freezinglemm}
Suppose $f$ witnesses $\diacol$, there is a Woodin cardinal and $V$ is closed under $X\mapsto M_1^\sharp(X)$. Then there is a $f$-preserving forcing which freezes $\NS$ along $f$.
\end{lemm}

\begin{lemm}
    Suppose $f$ witnesses $\diacol$ and there is a supercompact cardinal. Then there is an $f$-preserving forcing $\PP$ with $V^\PP\models\SRP$.
\end{lemm}

We can show this right away.

\begin{proof}
    The same construction which forces $\SRP$ via semiproper forcing from a supercompact cardinal can be used. A small change in the proof gives that any forcing for an instance of $\SRP$ is not only semiproper, but also $f$-semiproper. Now use Theorem \ref{fsemiproperiterationtheoremotherpaper} instead of Shelah's iteration theorem for semiproper forcings and do a nice iteration instead of a RCS iteration.
\end{proof}

We will eventually prove Lemma \ref{freezinglemm} in the next section. The basic idea is to use a version of the \aspsch\ $\Wstar$-forcing with $\Pmax$ replaced by $\Qmax$. However, we will run into a number of problems we need to solve first.

\section{Blueprints for Instances of ``$\MMpp\Rightarrow\Wstar$"}\label{mmimpliesstarsection}

We modify the $\Wstar$-forcing method of \aspsch\ in a way that allows us to prove a variety of instances of $\MMpp\Rightarrow\Wstar$, though our main interest lies in Lemma \ref{freezinglemm}.

\begin{defn}
Let $\PP\in \LofR$ be a forcing. $\stargen{\PP}$\index{(*)@$\Wstar$!P@$\PP$-} asserts that $\AD$ holds in $\LofR$ and there is a filter $g\subseteq\PP$ with 
\begin{enumerate}[label=\rn*]
\item $g$ is $\PP$-generic over $\LofR$ and
\item $\mathcal P(\omega_1)\subseteq \LofR[g]$.
\end{enumerate}
\end{defn}
$\Wstar$\index{(*)@$\Wstar$} is $\stargen{\Pmax}$. $\Pmax$ is the most prominent of a number of similar forcing notions defined and analyzed by Woodin in \cite{woodinbook}. A central notion to all of them is that of a generically iterable structure.

\begin{defn}
Suppose the following holds:
\begin{enumerate}[label=$(M.\roman*)$]
\item $(M;\in, I)$ is a countable transitive model of (sufficiently much of) $\ZFC$ where $I$ is allowed as a class parameter in the schemes.
\item $(M;\in, I)\models\text{``}I$ is a normal uniform ideal on $\omega_1$".
\item $a_0,\dots, a_n\in M$.
\end{enumerate}
In this case, we call $(M, I, a_0,\dots, a_n)$ a \textit{potentially iterable structure}\index{Potentially iterable structure}.
A \textit{generic iteration} \index{Potentially iterable structure!generic iteration of} of $(M, I, a_0,\dots, a_n)$ is a sequence 
$$\langle (M_\alpha, I_\alpha, a_{0,\alpha}, \dots, a_{n, \alpha}), \mu_{\alpha,\beta}\mid\alpha\leq\beta\leq\gamma\rangle$$
with 
\begin{itemize}
\item $(M_0, I_0)=(M, I)$,
\item $a_{i,\alpha}=\mu_{0,\alpha}(a_i)$ for $i\leq n$,
\item $\mu_{\alpha, \alpha+1}\colon (M_\alpha;\in, I_\alpha)\rightarrow (M_{\alpha+1};\in, I_{\alpha+1})$ is a generic ultrapower of $M_\alpha$ w.r.t $I_\alpha$ and
\item if $\alpha\in\Lim$ then 
$$\langle (M_\alpha;\in, I_\alpha), \mu_{\beta,\alpha}\mid\beta<\alpha\rangle=\varinjlim\langle (M_\beta;\in, I_\beta), M_{\beta,\xi}\mid\beta\leq\xi<\alpha\rangle$$
\end{itemize}
for all $\alpha\leq\gamma$. $(M, I, a_0,\dots, a_n)$ is a \index{Potentially iterable structure!generically iterable} generically iterable structure if all (countable) generic iterations of $(M, I, a_0,\dots, a_n)$ produce wellfounded models. Note that this only depends on $(M, I)$ and that we do not require $I\in M$. 
\end{defn}

\begin{rem}
A generic iteration $\langle (M_\alpha, I_\alpha, a_{0,\alpha},\dots, a_{n,\alpha}), \mu_{\alpha,\beta}\mid\alpha\leq\beta\leq\gamma\rangle$ can be read off from the final map $\mu_{0,\gamma}\colon M_0\rightarrow M_\gamma$, so we will frequently identify one with the other. We also reserve the right to call generic iterations simply iterations.
\end{rem}

\begin{defn}
$\Pmax$-conditions\index{Pmax@$\Pmax$} are generically iterable structures $(M, I, a)$ with $a\in\Pomo^M$ and $M\models \omega_1^{L[a]}=\omega_1$. $\Pmax$ is ordered by $q=(N,J, b)<_{\Pmax} p$ iff there is a generic iteration 
$$\mu\colon p\rightarrow p^\ast=(M^\ast, I^\ast, a^\ast)$$
of length $\omega_1^q+1$ in $q$ so that 
\begin{enumerate}[label=$(<_{\Pmax}.\roman*)$]
\item $I^\ast= J\cap M^\ast$ and
\item $a^\ast=b$.
\end{enumerate}
\end{defn}

There are a number of ways this definition can be varied, leading to different partial orders. We will work with such variants in a general context.

\subsection{$\Pmaxvariants$ and the $\Vmax$-multiverse view}

\begin{defn}
A $\Pmaxvariant$\index{Pmax@$\Pmax$!Variation@-\variant} is a nonempty projective preorder \hbox{$(\Vmax,\leq_{\Vmax})$} with the following properties:

\begin{enumerate}[label=$(\Vmax.\roman*)$]
\item Conditions in $\Vmax$ are generically iterable structures $(M, I, a_0,\dots, a_n)$ for some fixed $n=n^{\Vmax}$\footnote{Of course, not all structures of this form are necessarily conditions.}.
\item\label{Vmaxvarphi} There is a first order formula $\varphi^{\Vmax}$ in the language\footnote{When dealing with $\Pmaxvariant$s, we stick to the convention that capitalized symbols are unary predicates symbols which are lower case are constants.} $\{\in, \dot I,\dot a_0,\dots ,\dot a_n\}$ so that $q=(N, J, b_0,\dots b_n)<_{\Vmax}(M, I, a_0,\dots a_n)$ iff there is a generic iteration 
$$j\colon p\rightarrow p^\ast=(M^\ast, I^\ast, a_0^\ast, \dots, a_n^\ast)$$
in $N$ of length $\omega_1^N+1$ with
$$(N;\in, J, b_0,\dots, b_n)\models \varphi^{\Vmax}(p^\ast).$$
\item\label{Vmaxordertransitem} If $\mu\colon p\rightarrow p^\ast$ witnesses $q<_{\Vmax} p$ and $\sigma\colon q\rightarrow q^\ast$ witnesses $r<_{\Vmax} q$ then $\sigma(\mu)\colon p\rightarrow\sigma(p^\ast)$ witnesses $r<_{\Vmax} p$. 
\item\label{Vmaxiteriesareconds} Suppose $(M, I)$ is generically iterable, $j\colon (M, I)\rightarrow (M^\ast, I^\ast)$ is a generic iteration of $(M, I)$ of countable length and $a_0,\dots a_n\in M$. Then 
$$(M, I, a_0,\dots, a_n)\in\Vmax\Leftrightarrow(M^\ast, I^\ast, j(a_0),\dots, j(a_n))\in\Vmax.$$
\item\label{Vmaxnomin} $\Vmax$ has no minimal conditions.
\end{enumerate}
\end{defn}
We always consider $\Pmaxvariants$ as a class defined by a projective formula, rather then the set itself. So if we mention $\Vmax$ in, e.g.~a forcing extension of $V$, then we mean the evaluation of the projective formula in that model\footnote{In practice this extension will be projectively absolute so it does not matter which projective formula we choose. Also all the variations we consider will have a $\Pi^1_2$-definition.}.
\begin{rem}
Typically, $\varphi^{\Vmax}$ dictates e.g.~one or more of the following:
\begin{itemize}
\item $a_0^\ast=b_0,\dots, a_n^\ast=b_n$.
\item $I^\ast=J\cap M^\ast$.
\item Some first order property is absolute between $M^\ast$ and $N$.
\end{itemize}
\end{rem}

We want to relate forcing axioms to star axioms of the form $\stargen{\Vmax}$ for $\Pmaxvariant$s $\Vmax$. To explain this relationship heuristically we present the $\Vmax$-Multiverse View: \\
Suppose $\Vmax$ is a $\Pmaxvariant$ (with $n^{\Vmax}=0$ for convenience) and
\begin{itemize}
\item $V=(V_\kappa)^{\mathcal V}$ for some large cardinal $\kappa$ in some larger model $\mathcal V$ and
\item there are a proper class of Woodin cardinals both in $V$ and $\mathcal V$.
\end{itemize}
We will take the point of view of $\mathcal V^{\Col(\omega,\kappa)}$. Note that our assumptions imply generic projective absoluteness (and more) in $\mathcal V$, in particular $\Vmax$ is a $\Pmaxvariant$ also in $\mathcal V^{\Col(\omega,\kappa)}$ and $\Vmax^{W}=\Vmax\cap W$ for any generic extension of $V$. Pick some $\vec A=(A_0,\dots, A_{n^{\Vmax}})\in\Htwo^{V}$. Let $\mathcal M(V)$ denote the closure of $V$ under generic extensions and grounds containing $\vec A$. Points $W\in\mathcal M(V)$ may be considered as $\Vmax$-conditions if 
$$(W, \NS^W, A_0,\dots, A_{n^{\Vmax}})\in\Vmax.$$
In this case we identify $W$ with this condition. In practice, this can only reasonably hold if $\omega_1^W=\omega_1^V$ so we make this an explicit condition. The $\Vmax$-multiverse of $V$ (w.r.t. $\vec A$) is
$$\mathcal M_{\Vmax}(V)=\{W\in\mathcal M(V)\mid W\in\Vmax\wedge \omega_1^W=\omega_1^V\}.$$
If we $\vec A$ picked with sufficient care then $\mathcal M_{\Vmax}(V)$ should be nonempty.
If $W[G]$ is a generic extension of $W$, both in $\mathcal M_{\Vmax}(V)$, then it is a good extension if
$$W[G]\leq_{\Vmax} W.$$
Here, $p\leq_{\Vmax} q$ means $p\forces_{\Vmax}\check q\in\dot G$. The existence of a proper class of Woodin cardinals in $V$ should guarantee that $\mathcal M_{\Vmax}(V)$ reversely ordered by good extensions is ``as rich as" $\Vmax$.\\
In this sense, iterated forcing along good extensions corresponds to building descending sequences in $\Vmax$.  In practice, $\Pmaxvariant$s are $\sigma$-closed. From this point of view, $\sigma$-closure of $\Vmax$ becomes roughly equivalent to a forcing iteration theorem: If 
$$\langle W[G_\alpha]\mid\alpha<\gamma\rangle$$
is a chain of good extensions $W[G_\alpha]\subseteq W[G_\beta]$ of points 
$$W[G_\alpha],W[G_\beta]\in\mathcal M_{\Vmax}(V),\ \alpha\leq\beta<\gamma\in V$$ then this constitutes a countable decreasing chain\footnote{Note that the size of $\gamma$ in $V$ does not matter here.} in $\Vmax$ in $\mathcal V^{\Col(\omega,\kappa)}$. $\sigma$-closure of $\Vmax$ suggests that there should be a further point 
$$W[G_\gamma]\in \mathcal M_{\Vmax}(V)$$
below all $W[G_\alpha]$, $\alpha<\gamma$. Thus the ``forcing iteration along \hbox{$\langle W[G_\alpha]\mid\alpha<\gamma\rangle$"} preserves $\omega_1$ and enough structure to be able to be extended to a $\Vmax$-condition below all $W[G_\alpha]$ without collapsing $\omega_1$.\\
We should be able to find points satisfying $\stargen{\Vmax}$ by constructing ``closure points" $W\in\mathcal M_{\Vmax}(V)$ of sufficiently generic $\leq_{\Vmax}$-decreasing sequences 
$$\langle W_\alpha\mid\alpha<\gamma\rangle$$
in $\mathcal M_{\Vmax}(V)$. To make that precise, we want:
\begin{equation*}\label{multiverseeq}\tag{$\bigstar$}
\text{If }D\in\LofR^W\text{ is dense open in }\Vmax^W\text{ then }W_\alpha\in D^\ast\text{ for some }\alpha<\gamma.
\end{equation*}
Here, $D^\ast$ is the reinterpretation of the universally Baire $D$ in $\mathcal V^{\Col(\omega,\kappa)}$.
The degree of closure of $W\in\mathcal M_{\Vmax}(V)$ under this procedure is measured by
$$g^{W}=\{p\in \Vmax\mid W<_{\Vmax} p\}$$
which should be a filter if $W$ is ``sufficiently closed". $g^{W}$ can be defined in $W$ via 
$$g^{W}=\{p\in\Vmax\mid\exists \mu\colon p\rightarrow p^\ast\text{ of length }\omega_1+1\text{ with } \varphi^{\Vmax}(p^\ast)\}^{W}$$
if $\Vmax$ has unique iterations.

\begin{defn}
$\Vmax$ has unique iterations\index{Pmax@$\Pmax$!Variation@-\variant!has unique Iterations} if whenever $q<_{\Vmax} p$ then there is a unique generic iteration of $p$ witnessing this.
\end{defn}

Under reasonable assumptions, (\ref{multiverseeq}) implies that $g^W$ is generic over $\LofR^W$. Finally, an additional property\footnote{Often, simply $(\neg\CH)^W$ is enough. Woodin \cite{woodinnote} (see also \cite{ralfnotchnote}) has shown that if $\AD^{\LofR}$ holds, there is a filter $g\subseteq\Pmax$ generic over $\LofR$ and $\CH$ fails then $g$ witnesses $\Wstar$.} like $W\models``\NS\text{ is saturated}"$ should imply $\Pomo^W\subseteq\LofR^W[g^W]$.\\
Taking a step back, forcing a forcing axiom related to good extensions via iterated forcing looks like it should produce such sequences $\langle W_\alpha\mid\alpha<\gamma\rangle$ with (\ref{multiverseeq}) and $\NS$ saturated in $W$, so $\stargen{\Vmax}$ should follow from such a forcing axiom.\\
On the other hand, $W$ looks like an endpoint of an iteration liberally incorporating forcings leading to good extensions: For $\alpha<\gamma$, if $D\in\LofR^{W_\alpha}$ is dense open in $\Vmax^{W_\alpha}$ then $D^\ast$ is dense open in the full $\Vmax$. $D^\ast$ can also be considered as a dense subset of $\mathcal M_{\Vmax}(V)$. As $D^\ast\cap \Vmax^W\in\LofR^W$, by (\ref{multiverseeq}), there will be some later $\alpha\leq\beta<\gamma$ with $W_\beta\in D^\ast$. Thus one might expect a forcing axiom to hold at $W$. This suggest that $\Vmax$ should in fact be \textit{equivalent} to a forcing axiom related to good extensions. The consistency of this forcing axiom should follow from the iteration theorem suggested by the $\sigma$-closure of $\Vmax$. \\
If we look at the case $\Vmax=\Pmax$ and let $A$ be some subset of $\omega_1$ so that $\omega_1^{L[A]}=\omega_1^V$ then stationary set preserving extensions are exactly the generic extensions intermediate to a good extension. The $\Pmax$-Multiverse View is roughly correct in the sense that:
\begin{itemize}
\item (Woodin) $\Pmax$ is $\sigma$-closed assuming $\AD^{\LofR}$.
\item (Shelah) Semiproper forcings can be iterated and the class of stationary set preserving forcings and semiproper forcings coincide under $\MM$.
\item(\aspsch) If there is a proper class of Woodin cardinals then
$$\Wstar\Leftrightarrow (\mathcal{P}(\mathbb R)\cap \LofR)\text{-}\BMMpp.$$
\end{itemize}

The rest of this section distills this heuristic into rigorous mathematics that relates more $\Pmaxvariant$s to forcing axioms. We will assume (two-step) generic absoluteness in this section, though this is not fully necessary. Note that in this case, if $\Vmax$ is a $\Pmaxvariant$ then we have
$$V^\PP\models`` \Vmax\text{ is a }\Pmaxvariant"$$
in any generic extension $V^\PP$, where $\Vmax$ is  to be understood as defined by a projective formula. Usually, $\Pmaxvariant$s are $\Pi^1_2$. \\
We will from now on work with some fixed $\Pmaxvariant$ $\Vmax$ and assume $n^\Vmax=0$ to ease notation.

\begin{defn}
We say that a structure $\mathcal H$ is \textit{almost a }$\Vmax$\textit{-condition}\index{almost a Vmax condition@Almost a $\Vmax$-condition} if
$$V^{\Col(\omega,\mathcal H)}\models\widecheck{\mathcal H}\in\Vmax.$$
For $A\in\Htwo$, $\mathcal H_A$ denotes the structure:
$$\mathcal H_A\coloneqq (\Htwo, \NS, A)$$
\end{defn}

 Suppose that for some fixed $A\in\Htwo$ we have that $\mathcal H\coloneqq \mathcal H_A$ is almost a $\Vmax$-condition. We may define 
$$g_{A}=\{p\in\Vmax\mid V^{\Col(\omega,2^{\omega_1})}\models \mathcal H <_{\Vmax} p\}.$$
Our goal is to show that $g_A$ witnesses $\stargen{\Vmax}$ under favorable circumstances. At the very least, it should be a filter.
\begin{prop}\label{generalfilterprop}
Suppose $g_A$ meets all projective dense $D\subseteq\Vmax$. Then $g_A$ is a filter.
\end{prop}
\begin{proof}
It is easy to see that if $q<_{\Vmax}p$ and $q\in g_A$ then $p\in g_A$. So assume $p, q\in g_A$ and we have to find some $r\in g_A$ with $r\leq_{\Vmax}p, q$. Consider 
$$D=\{r\in \Vmax\mid r\leq_{\Vmax} p,q \vee r\perp p\vee r\perp q\}$$
and note that $D$ is a projective dense subset of $\Vmax$, so by assumption we can find some  $r\in D\cap g_A$. Now in $V^{\Col(\omega,2^{\omega_1})}$ we have $r, p, q\leq_{\Vmax}\mathcal H$ and thus $r$ is compatible with both $p$ and $q$. By generic absoluteness, this is true in $V$ as well so that $r\leq_{\Vmax}p, q$ as $r\in D$.
\end{proof}

Even assuming that $g_A$ is a fully generic over $\LofR$, we still have to arrange $\Pomo\subseteq\LofR[g_A]$. 

\begin{defn}
Suppose that
\begin{enumerate}[label=$(\roman*)$]
\item $g\subseteq\Vmax$ is a filter,
\item $p\in g$ and
\item $\langle p_\alpha,\mu_{\alpha,\beta}\mid\alpha\leq\beta\leq\gamma\rangle$ is a generic iteration of $p_0=p$.
\end{enumerate}
Then we say that $\langle p_\alpha,\mu_{\alpha,\beta}\mid\alpha\leq\beta\leq\gamma\rangle$ is guided\index{Potentially iterable structure!generic iteration of!guided by g@guided by $g$} by $g$ if $p_\alpha\in g$ for all countable $\alpha\leq\gamma$.
\end{defn}

\begin{lemm}\label{guidediterationslemm}
Suppose $\Vmax$ has unique iterations and $g\subseteq\Vmax$ is a filter meeting all projective dense $D\subseteq\Vmax$. For any $p\in g$ and any $\gamma\leq\omega_1$, there is a unique iteration 
$$\langle p_\alpha,\mu_{\alpha,\beta}\mid\alpha\leq\beta\leq\gamma\rangle$$
of $p_0=p$ of length $\gamma+1$ guided by $g$.
\end{lemm}

\begin{proof}
First, we prove existence for all $\gamma<\omega_1$.

\begin{claim}
There is $q\in g$ with $\omega_1^q>\gamma$.
\end{claim}
\begin{proof}
Let $D=\{q\in\Vmax\mid\omega_1^q>\gamma\}$. Clearly, $D$ is projective and we will show that $D$ is dense. Let $q\in\Vmax$ and using \ref{Vmaxnomin}, find $r<_{\Vmax} q$ as witnessed by
$$\sigma\colon q\rightarrow q^\ast.$$
Now let 
$$\nu\colon r\rightarrow r^\ast$$
be any generic iteration of $r$ of length $\gamma+2$, consequently $\omega_1^{r^\ast}>\gamma$. We have $r^\ast\in\Vmax$ by \ref{Vmaxiteriesareconds}. Note that the iteration $\nu\circ\sigma$ witnesses $r^{\ast}<_{\Vmax}q$. Again applying \ref{Vmaxnomin}, there is $s<_{\Vmax}r^{\ast}$ and thus $s<_{\Vmax} q$ and $s\in D$.\\
Thus $g\cap D\neq\emptyset$.
\end{proof}
As $g$ is a filter, we can find $q<_{\Vmax} p$ with $\omega_1^q>\gamma$. Thus if $\mu\colon p\rightarrow p^\ast$ witnesses this then $\mu$ is an iteration
$$\langle p_{\alpha, \beta},\mu_{\alpha, \beta}\mid \alpha\leq \beta\leq \omega_1^{q}\rangle$$
of length $\omega_1^q+1>\gamma+1$ by \ref{Vmaxvarphi}. 
\begin{claim}
$\langle p_{\alpha, \beta},\mu_{\alpha, \beta}\mid \alpha\leq \beta\leq\gamma\rangle$ is guided by $g$. 
\end{claim}
\begin{proof}
Let $\alpha\leq\gamma$. Then $\mu_{\alpha,\omega_1^q}$ is an iteration of length $\omega_1^q+1$ in $q$ and $q\models\varphi^{\Vmax}(p_{\omega_1^q})$, thus $q<_{\Vmax} p_\alpha$ and $p_\alpha\in g$.
\end{proof}
Next we prove uniqueness. By proceeding by induction on $\gamma\leq\omega_1$, it is in fact enough to verify the case $\gamma=1$. Suppose that $\mu_i\colon p\rightarrow p^\ast_i$ is a generic ultrapower of $p$ with $p^\ast_i\in g$ for $i<2$. As $g$ is a filter and by \ref{Vmaxnomin}, there is $q\in g$ with $q<_{\Vmax} p^\ast_i$ as witnessed by some
$$\mu^\ast_i\colon p^\ast_i\rightarrow p^{\ast\ast}_i$$
for $i<2$ as well as $q<_{\Vmax}p$ as witnessed by
$$\mu\colon p\rightarrow p^{\ast\ast}.$$
Let $i<2$. We have that $p, p^{\ast}_i$ are countable in $q$. As 
$$``p^{\ast}_i\text{ is a generic ultrapower of }p"$$
is a true $\Sigma^1_1(p, p^{\ast}_i)$-statement, it is true in $q$ as well. Thus there is a generic ultrapower
$$\mu'_i\colon p\rightarrow p^{\ast}_i$$
in $q$. Both $\mu, \mu^{\ast}_i\circ\mu'_i$ witness $q<_{\Vmax} p$ and as $\Vmax$ has unique iterations, $\mu=\mu^{\ast}_i\circ\mu'_i$. It follows that $p^\ast_0=p^\ast_1$.
\begin{claim}
$\mu_0^\ast=\mu_1^\ast$.
\end{claim}
\begin{proof}
 Assume this fails, then 
 $$``\text{There are distinct generic ultrapower maps }p\rightarrow p_0^\ast"$$
 is another true $\Sigma^1_1(p, p_0^\ast)$-statement which accordingly must hold in $q$. Thus there is a generic ultrapower map $\mu_0^{\prime\prime}\colon p\rightarrow p^\ast_0$ in $q$ different from $\mu_0'$. But then both $\mu_0^\ast\circ\mu_0'$ and $\mu_0^\ast\circ\mu_0^{\prime\prime}$ witness $q<_{\Vmax} p$, which contradicts that $\Vmax$ has unique iterations.
\end{proof}
Finally, existence of a generic iteration of $p$ of length $\omega_1+1$ guided by $g$ follows from existence and uniqueness of generic iterations of $p$ guided by $g$ of any countable length.
\end{proof}

This suggests the following definition:
\begin{defn}
Suppose $\Vmax$ is a $\Pmaxvariant$ with unique iterations and $g\subseteq\Vmax$ is a filter. For $p\in g$, the $g$\textit{-iteration}\index{g iteration@$g$-iteration} of $p$ is the unique generic iteration of $p$ of length $\omega_1+1$ that is guided by $g$ (if it exists).
\end{defn}

\begin{cor}\label{getVmaxstarcor}
Suppose that
\begin{enumerate}[label=$(\roman*)$]
\item $\AD$ holds in $\LofR$,
\item $\Vmax$ has unique iterations,
\item $\mathcal H_A$ is almost a $\Vmax$-condition,
\item $g_A\cap D\neq\emptyset$ for all dense $D\subseteq\Vmax$, $D\in\LofR$ and
\item\label{getVmaxstarpomocond} $\Pomo=\bigcup\{\Pomo\cap p^\ast\mid p\in g_A\wedge \mu\colon p\rightarrow p^\ast\text{ is guided by }g_A\}$.
\end{enumerate}
Then $\stargen{\Vmax}$ holds and $g_A$ witnesses this.
\end{cor}
\begin{proof}
$g_A$ is a filter by Proposition \ref{generalfilterprop} and thus $\LofR$-generic by assumption. To see that $\Pomo\subseteq\LofR[g_A]$, notice that for any $p\in g_A$, $\LofR$ knows of all countable generic iterations of $p$. Hence, $\LofR[g_A]$ can piece together the $g_A$-iteration of $p$ from the countable iterations of $p$ that are guided by $g_A$. $\Pomo\subseteq\LofR[g_A]$ now follows immediately from \ref{getVmaxstarpomocond}.
\end{proof}

The biggest obstacle by far is to get into a situation where $g_A\cap D\neq\emptyset$ for all dense $D\subseteq\Vmax$, $D\in\LofR$. The main idea is:

\begin{lemm}\label{meetDlemm}
Suppose that all of the following hold:
\begin{enumerate}[label=$(\roman*)$]
\item $D\subseteq\Vmax$ is dense.
\item $\mathcal H_A$ is almost a $\Vmax$-condition.
\item $\PP$ is a forcing and $D$ is $\vert \PP\vert$-universally Baire.
\item\label{meetDgenitcon} In $V^\PP$ there is $q\in D^\ast$ and an iteration $\sigma\colon q\rightarrow q^\ast$ with  
$$(\Htwo;\in,\NS, A)^{V^\PP}\models \varphi^{\Vmax}(q^\ast).$$
\item $\Gamma$ is a set of formulas in the language $\{\in, \dot I, \dot a, \dot D\}$ so that
\begin{enumerate}[label=$(\Gamma.\roman*)$]
\item $\varphi^{\Vmax}\in\Gamma$,
\item $\Sigma_0\subseteq\Gamma$, where $\Sigma_0$ is computed in the language $\{\in,\dot D\}$ and
\item $\Gamma$ is closed under $\exists$ and  $\wedge$.
\end{enumerate}
\item\label{meetDFAcon} $(H_{\omega_2};\in, \NS, A, D)^V\prec_\Gamma(\Htwo;\in,\NS, A, D^\ast)^{V^\PP}$.
\end{enumerate}
Then $g_A\cap D\neq\emptyset$.\\
If additionally
\begin{enumerate}[resume,label=$(\roman*)$]
\item\label{meetDHtwocon} $\Htwo^V\subseteq q^\ast$
\end{enumerate}
then $\Pomo=\bigcup\{\Pomo\cap p^\ast\mid p\in g_A\wedge \mu\colon p\rightarrow p^\ast\text{ is guided by }g_{A}\}$.
\end{lemm}
\begin{proof}
Observe that $(\Htwo;\in)\prec_{\Sigma_1}(\Htwo;\in)^{V^\PP}$ implies that $\PP$ preserves $\omega_1$. The statement
$$\exists q\in \dot D\ \exists \sigma\colon q\rightarrow q^\ast\text{ an iteration of length }\omega_1+1\text{ and }\varphi^{\Vmax}(q^\ast)$$
is in $\Gamma$ and thus is true in 
$$(\Htwo;\in,\NS, A, D)^V$$
as witnessed by some $p\in D$ and iteration $\mu\colon p\rightarrow p^\ast$. It follows that $\mu$ witnesses $\mathcal H_A<_{\Vmax} q$ in $V^{\Col(\omega,2^{\omega_1})}$ so that $p\in D\cap g_A$.\\
Now assume \ref{meetDHtwocon}, it is our duty to show
$$\Pomo=\bigcup\{\Pomo\cap p^\ast\mid p\in g_A\wedge \mu\colon p\rightarrow p^\ast\text{ is guided by }g_{A}\}.$$
Let $X\subseteq\omega_1$. As above, 
$$\exists q\in\Vmax\ \exists\sigma\colon q\rightarrow q^\ast\text{ an iteration of length }\omega_1+1\text{ and }\varphi^{\Vmax}(q^\ast)\wedge X\in q^\ast$$
reflects down to $V$. The iteration witnessing this in $V$ is guided by $g_{A}$ by the same argument that showed $p\in g_A$ above.
\end{proof}

Condition \ref{meetDFAcon} is a typical consequence of a (bounded) forcing axiom. It is left to construct forcings $\PP$ with property \ref{meetDgenitcon} to which hopefully a broad range of forcing axioms may apply. 

\subsection{\aspsch\ $\Wstar$-forcing}

We describe the results of \aspsch \cite{aspsch}. Their results carry over to any $\Pmaxvariant$ $\Vmax$ though they were originally proven in the case of $\Vmax=\Pmax$.  Suppose that 
\begin{enumerate}[label=$(\roman*)$]
\item $\NS$ is saturated,
\item $A\in \Htwo$ is so that $\mathcal H=(\Htwo,\NS, A)$ is almost a $\Vmax$-condition and
\item\label{univbaireassumption} $D\subseteq\Vmax$ is a $2^{\omega_1}$-universally Baire dense subset of $\Vmax$ whose reinterpretation is still dense in extensions by forcings of size $\leq 2^{\omega_1}$, as witnessed by trees $T, S$ with $D=p[T]$.
\end{enumerate}

\aspsch\ construct a partial order $\PP=\PP(\Vmax, A, D)$ so that in $V^\PP$ the following picture

\begin{tikzpicture}
\def\x{2.1};
\def\y{0.5};

\node (pT) at (0*\x,  0*\y){$p[T]$};
\node (N) at (0*\x, -2*\y) {$q_0=(N, I, b)$};
\node (Nast) at (2*\x, -2*\y) {$q_{\omega_1}=(N^*, I^*, b)$};
\node (M0) at (-2*\x, -4*\y) {$p_0$};
\node (MN) at (0*\x, -4*\y) {$p_{\omega_1^N}$};
\node (M3) at (2*\x, -4*\y) {$p_{\omega_1}$};
\node (H) at (2*\x, -6*\y) {$((H_{\omega_2})^V, \mathrm{NS}_{\omega_1}^V, A)=\mathcal H$};
\node (Vmax) at (-2*\x, -6*\y) {$\Vmax$};

\path (N)--(pT) node[midway, rotate=90]{$\in$};

\draw[->] (N)--(Nast) node[midway, above]{$\sigma_{0,\omega_1}$};

\path (MN)--(N) node[midway, rotate=90]{$\in$};
\path (M3)--(Nast) node[midway, rotate=90]{$\in$};

\draw[->] (M0)--(MN) node[midway, above]{$\mu_{0, \omega_1^N}$};
\draw[->] (MN)--(M3) node[midway, above]{$\mu_{\omega_1^N, \omega_1}$};

\path (H)--(M3) node[midway, rotate=90]{$=$};

\path (Vmax)--(M0) node[midway, rotate=270]{$\in$};

\end{tikzpicture}

exists so that
\begin{enumerate}[label=$(\PP.\roman*)$]
\item $\mu_{0, \omega_1}, \sigma_{0,\omega_1}$ are generic iterations of $p_0$, $q_0$ respectively,
\item\label{smalleraspschcond} $\mu_{0,\omega_1^N}$ witnesses $q_0<_{\Vmax} p_0$,
\item $\mu_{0,\omega_1}=\sigma_{0,\omega_1}(\mu_{0,\omega_1^N})$ and
\item\label{idealsmatchaspschcond} the generic iteration $\sigma_{0, \omega_1}\colon q_0\rightarrow q_{\omega_1}$ is \textit{correct}, i.e.~$I^\ast=\NS^{V^\PP}\cap N^\ast$.
\end{enumerate} 
If $\varphi^{\Vmax}((M, J, a))$ implies $J=\dot I\cap M$ then $\NS^{p_{\omega_1^N}}=I\cap p_{\omega_1^N}$. This gets transported upwards along $\sigma_{0,\omega_1}$ and shows $\NS^V=I^\ast\cap \Htwo^V$. Together with \ref{idealsmatchaspschcond}, this yields $\NS^V=\NS^{V^\PP}\cap V$, i.e.~$\PP$ preserves stationary sets. If $\MMpp$ holds in $V$ then 
$$(\Htwo;\in,\NS, A, D)^V\prec_{\Sigma_1}(\Htwo;\in,\NS, A, D^\ast)^{V^\PP}$$
and it follows from Lemma \ref{meetDlemm} that $g_A\cap D\neq\emptyset$ (note that $\varphi^{\Pmax}((M, I, a))``=I=\dot I\cap M\wedge a=\dot a"$). This is how \aspsch\ prove $\MMpp\Rightarrow\Wstar$.\\
An important observation is the following: To invoke a forcing axiom in the case of $\PP$ or variants thereof, typically $\PP$ needs to preserve certain structure, like stationary sets in the example above. This preservation is proven in two steps:
\begin{enumerate}[label=$(\roman*)$]
 \item Preservation between $q_{\omega_1}$ and $V^\PP$. This is governed by the iteration $\sigma_{0,\omega_1}$ having certain properties in $V^\PP$, e.g.~correctness. 
 \item Preservation between $p_{\omega_1}$ and $q_{\omega_1}$. This is governed by the nature of $\Vmax$, specifically the formula $\varphi^{\Vmax}$. 
 \end{enumerate}  
 
We will modify the construction of $\PP$ and get a forcing $\PPdia$ which strengthens \ref{idealsmatchaspschcond} so that $\PPdia$ can have a variety of preservation properties depending on the $\Pmaxvariant$ $\Vmax$ in question, for example
\begin{itemize}
\item preserving stationary sets as well as all Suslin trees  or
\item preserving a witness $f$ of $\diacol$ ($\rightsquigarrow\QM\Rightarrow\stargen{\Qmax}$).
\end{itemize}

\subsection{$\diamondsuit$-iterations}
We introduce the concept that is roughly the equivalent of $\diamondsuit$-forcing in the world of generic iterations.
\begin{defn}
Suppose $(N, I)$ is generically iterable. A generic iteration 
$$\langle (N_i, I_i), \sigma_{i, j}\mid i\leq j\leq\omega_1\rangle$$
of $(N, I)=(N_0, I_0)$ is a $\diamondsuit$\textit{-iteration}\index{diamond iteration@$\diamondsuit$-iteration} if for any 
\begin{enumerate}[label=$(\roman*)$]
\item sequence $\langle D_i\mid i<\omega_1\rangle$ of dense subsets of $((\Pomo/I_{\omega_1})^+)^{N_{\omega_1}}$ and
\item $S\in \Pomo^{N_{\omega_1}}- I_{\omega_1}$
\end{enumerate}
the set
$$\{\xi\in S\mid \forall i<\xi\ g_\xi\cap \sigma_{\xi,\omega_1}^{-1}[D_i]\neq\emptyset\}$$
is stationary. Here, $g_\xi$ is the generic ultrafilter applied to $N_\xi$ for $\xi<\omega_1$.
\end{defn}

If $(N, I)$ is generically iterable and $\diamondsuit$ holds then there is a $\diamondsuit$-iteration of $(N, I)$. But this is not generally the case. Paul Larson noted that if $(M, I)$ is generically iterable and 
$$\langle M_\alpha, \mu_{\alpha,\beta}\mid\alpha\leq\beta\leq\omega_1\rangle$$
is a \textit{generic} generic iteration of $(M, I)=(M_0, I_0)$ of length $\omega_1$ then this is a $\diamondsuit$-iteration. By this we mean that this iteration has been constructed generically by forcing with countable approximations ordered by endextension.
\begin{lemm}\label{diamondsuititerationdiablemm}
Suppose 
$$\langle (N_i, I_i), \sigma_{i, j}, g_i\mid i\leq j\leq\omega_1\rangle$$
is a $\diamondsuit$-iteration. If 
$$N_{\omega_1}\models ``f\text{ witnesses }\diabgena{I_{\omega_1}}"$$
then $I_{\omega_1}=\NSf\cap N_{\omega_1}$. In particular, $f$ witnesses $\diab$.
\end{lemm}

\begin{proof}
Let $S\in \Pomo^{N_{\omega_1}}- I_{\omega_1}$, we have to show that $S$ is $f$-stationary. Let $\langle D_i'\mid i<\omega_1\rangle$ be a sequence of dense subsets of $\BP$. As $f$ witnesses $\diabgena{I_{\omega_1}}$ in $N_{\omega_1}$, we have 
$$N_{\omega_1}\models``\eta_f\colon \BP\rightarrow(\Pomo/I_{\omega_1})^+\text{ is a complete embedding}"$$
and notice that $\eta_f$ is a complete embedding in $V$ as well. Thus $D_i=\eta_f[D_i']$ is dense for $i<\omega_1$. As $\sigma_{0,\omega_1}\colon N_0\rightarrow N_{\omega_1}$ is a $\diamondsuit$-iteration, 
$$T\coloneqq \{\xi\in S\mid \forall i<\xi\ g_\xi\cap \sigma_{\xi,\omega_1}^{-1}[D_i]\neq\emptyset\}$$
is stationary. Thus if $C\subseteq\omega_1$ is club, we can find $\xi\in C\cap T$ with 
$\omega_1^{N_\xi}=\xi$ and $f\in\ran(\sigma_{\xi,\omega_1})$. It follows that 
$$f(\xi)=\eta_{\sigma_{\xi,\omega_1}^{-1}(f)}^{-1}[g_\xi]$$
so that $f(\xi)\cap D_i'\neq\emptyset$ for all $i<\xi$.
\end{proof}

\subsection{$\diamondsuit\text{-}\Wstar$-forcing}
\begin{thm}\label{forcingexistencelemm}
Suppose that 
\begin{enumerate}[label=$(\roman*)$]
\item\label{diastarforcinggenabscond} generic projective absoluteness holds for generic extensions by forcings of size $2^{\omega_1}$,
\item $\Vmax$ is a $\Pmaxvariant$,
\item $\NS$ is saturated and $\Pomo^\sharp$ exists,
\item $(\Htwo,\NS, A_0,\dots, A_{n^{\Vmax}})$ is almost a $\Vmax$-condition and
\item\label{diastarforcingdensecond} $D\subseteq\Vmax$ is $2^{\omega_1}$-universally Baire and dense in $\Vmax$ in any generic extension by a forcing of size $2^{\omega_1}$, as witnessed by trees $T, S$ with $p[T]=D$.
\end{enumerate}
Then there is a forcing $\PPdia$ so that in $V^{\PPdia}$ the following picture

\begin{tikzpicture}
\def\x{2.1};
\def\y{0.5};

\node (pT) at (0*\x,  0*\y){$p[T]$};
\node (N) at (0*\x, -2*\y) {$q_0$};
\node (Nast) at (2*\x, -2*\y) {$q_{\omega_1}$};
\node (M0) at (-2*\x, -4*\y) {$p_0$};
\node (MN) at (0*\x, -4*\y) {$p_{\omega_1^N}$};
\node (M3) at (2*\x, -4*\y) {$p_{\omega_1}$};
\node (H) at (2*\x, -6*\y) {$((H_{\omega_2})^V, \mathrm{NS}_{\omega_1}^V, A_0,\dots, A_{n^{\Vmax}})=\mathcal H$};
\node (Vmax) at (-2*\x, -6*\y) {$\Vmax$};

\path (N)--(pT) node[midway, rotate=90]{$\in$};

\draw[->] (N)--(Nast) node[midway, above]{$\sigma_{0,\omega_1}$};

\path (MN)--(N) node[midway, rotate=90]{$\in$};
\path (M3)--(Nast) node[midway, rotate=90]{$\in$};

\draw[->] (M0)--(MN) node[midway, above]{$\mu_{0, \omega_1^N}$};
\draw[->] (MN)--(M3) node[midway, above]{$\mu_{\omega_1^N, \omega_1}$};

\path (H)--(M3) node[midway, rotate=90]{$=$};

\path (Vmax)--(M0) node[midway, rotate=270]{$\in$};

\end{tikzpicture}

exists so that
\begin{enumerate}[label=$(\PPdia.\roman*)$]
\item $\mu_{0, \omega_1}, \sigma_{0,\omega_1}$ are generic iterations of $p_0$, $q_0$ respectively,
\item\label{smallerppdiacond} $\mu_{0,\omega_1^N}$ witnesses $q_0<_{\Vmax} p_0$,
\item $\mu_{0,\omega_1}=\sigma_{0,\omega_1}(\mu_{0,\omega_1^N})$ and
\item\label{idealsmatchppdiacond} the generic iteration $\sigma_{0, \omega_1}\colon q_0\rightarrow q_{\omega_1}$ is a $\diamondsuit$-iteration.
\end{enumerate} 
\end{thm}

For the remainder of this section, $\omega_1$ will always denote $\omega_1^V$.\\
So suppose \ref{diastarforcinggenabscond}-\ref{diastarforcingdensecond} holds. We will assume $n^{\Vmax}=0$ for notational purposes. For the most part, we will follow the construction of $\PP$ in \cite{aspsch} but will put additional constraints on the certificates. The idea that guides us here is: 

\begin{quote}
In order for $\sigma_{0,\omega_1}\colon q\rightarrow q^\ast$ to be a $\diamondsuit$-iteration, the forcing $\PPdia$ will have to anticipate dense subsets of the forcing $(I^+)^{N_{\omega_1}}$ so that they have been ``hit before". This should be captured by the map $K\rightarrow C$. Formulating this correctly produces a strengthened version of the ``genericity condition" put onto semantic certificates.
\end{quote}

A reader who can compile the above paragraph without syntax error can probably safely skip most the definition of $\PP$ and go straight to \ref{newgenericitycondition}.\\

We try to keep our notation here consistent with the notation in the paper \cite{aspsch}. For this reason, we will identify a condition $p=(M, I, a)\in\Vmax$ with its first coordinate $M$. Additionally, by even more abuse of notation:
\begin{conv}
If $(N, J, b)$ is (almost) a condition in $\Vmax$, then
\begin{itemize}
\item $I^N$ denotes $J$,
\item $(I^+)^N$ denotes $\Pomo^M-J$ and
\item $a^N$ denotes $b$.
\end{itemize}
\end{conv}

We will additionally assume both $2^{\omega_1}=\omega_2$ and $\diamondsuit_{\omega_3}$ to hold. Otherwise, first force with $\Add(\omega_2,1)\ast\Add(((2^{\omega_1})^+)^V, 1)$ and note that \ref{diastarforcinggenabscond} and \ref{diastarforcingdensecond} still hold for forcing with $\Col(\omega,\omega_2)$, which is all we need. Moreover, observe that this preserves ``$\NS$ is saturated".\\
We will denote $\omega_3$ by $\kappa$ and pick a $\diamondsuit_\kappa$-sequence $\langle \bar A_\lambda\mid\lambda<\kappa\rangle$.

We may find $T_0\subseteq T$ of size $\omega_2$ so that 
$$V^{\Col(\omega,\omega_2)}\models\exists q\in p[T_0]\ q<_{\Vmax} \mathcal H.$$
Here we use that $\mathcal H$ is almost a $\Vmax$-condition as well as \ref{Vmaxnomin}. Note that $p[T_0]\subseteq p[T]$ in any outer model. Without loss of generality, we may assume that $T_0$ is a tree on $\omega\times\omega_2$. \\
 Fix a bijection 
$$c\colon \kappa\rightarrow H_\kappa.$$
For $\lambda<\kappa$ let 
$$Q_\lambda\coloneqq c[\lambda] \text{ and }A_\lambda\coloneqq c\left[\bar A_\lambda\right].$$
There is then a club $C\subseteq\kappa$ with 
\begin{enumerate}[label=\rn*]
\item $T_0, p \in Q_\lambda$ and $\omega_2+1\subseteq Q_\lambda$,
\item $Q_\lambda\cap\Ord=\lambda$ and
\item $(Q_\lambda;\in)\prec (H_\kappa;\in)$
\end{enumerate}
for all $\lambda\in C$. We now have 
\begin{align*}
 &\text{For all }P, B\subseteq H_\kappa\text{ the set }\\
(\diamondsuit)\hspace{10pt}&\hspace{10pt}\{\lambda\in C\mid (Q_\lambda;\in, P\cap Q_\lambda, A_\lambda)\prec (H_\kappa;\in, P, B)\}\\
&\text{is stationary.}
\end{align*}
We will also define $Q_\kappa$ as $H_\kappa$. The forcing $\PP$ will add some 
$$(N_0, I_0, a_0)\in D^\ast$$
together with a generic iteration 
$$\langle N_i, \sigma_{i, j}\mid i\leq j \leq\omega_1\rangle$$
by Henkin-style finite approximations. By abuse of notation, we let $N_i=(N_i; I_i, a_i)$. For readability we will also write 
$$N_{\omega_1}=(N_{\omega_1}, I^\ast, a^\ast).$$
$\PPdia$ will be the last element of an increasing sequence $\langle \PPdia_\lambda\mid \lambda\in C\cup\{\kappa\}\rangle$ of forcings which we define inductively. We will have:
\begin{enumerate}[label=\rn*]
\item $\PPdia_\lambda\subseteq Q_\lambda$,
\item conditions in $\PPdia_\lambda$ will be finite sets of formulae in a first order language $\mathcal L_\lambda$ and
\item the order on $\PPdia_\lambda$ is reverse inclusion.
\end{enumerate}

Suppose now that $\lambda\in C\cup\{\kappa\}$ and $\PPdia_\nu$ is defined for all $\nu\in C\cap\lambda$. 

We will make use of the same convention as \aspsch.

\begin{conv}
$x\subseteq\omega$ is a real code for $N_0=(N, I_0, a_0)$ if there is a surjection $f\colon \omega\rightarrow N$ so that $x$ is the monotone enumeration of G\"odel numbers of all expressions of the form 
$$\ulcorner \dot N\models \varphi(\dot n_1,\dots , \dot n_l, \dot  I, \dot a)\urcorner$$
where $\varphi$ is a first order formula of the language associated to $(N_0, I_0, a_0)$(see below) and 
$$N\models\varphi(f(n_1),\dots, f(n_l), I_0, a_0)$$
holds.
\end{conv}

We will have conditions in $\PPdia_\lambda$ be certified in a concrete sense by objects $\mathfrak{C}$ which exist in generic extensions of $V$ that satisfies projective absoluteness w.r.t. $V$. They are of the form 
$$\mathfrak{C}=\langle\langle M_i, \mu_{i, j}, N_i, \sigma_{i, j}\mid i\leq j\leq\omega_1\rangle, \langle (k_n, \alpha_n)\mid n<\omega\rangle, \langle \lambda_\xi, X_\xi\mid\xi\in K\rangle\rangle$$

where

\begin{enumerate}[label=$(\cert.\arabic*)$]
\item\label{semcertcondsfirst} $M_0$, $N_0\in\Vmax$,
\item\label{semcertcondssecond} $x=\langle k_n\mid n<\omega\rangle$ is a real code for $N_0=(N_0;\in, I, a_0)$ and \linebreak $\langle (k_n,\alpha_n)\mid n<\omega\rangle$ is a branch through $T_0$,
\item\label{semcertcondsthird} $\langle M_i, \mu_{i, j}\mid i\leq j\leq\omega_1^{N_0}\rangle\in N_0$ is a generic iteration of $M_0$ witnessing $N_0<_{\Vmax}M_0$,
\item\label{semcertcondsfourth} $\langle N_i, \sigma_{i, j}\mid i\leq j \leq\omega_1\rangle$ is a generic iteration of $N_0$,
\item $\langle M_i, \mu_{i, j}\mid i\leq j\leq\omega_1\rangle=\sigma_{0,\omega_1}(\langle M_i, \mu_{i, j}\mid i\leq j\leq\omega_1^{N_0}\rangle)$ and 
$$M_{\omega_1}=((H_{\omega_2})^V;\in, (\NS)^V, A),$$
\item\label{semcertcondslast} $K\subseteq\omega_1$ and for all $\xi\in K$
\begin{enumerate}[label=$(\cert.6.\alph*)$]
\item $\lambda_\xi\in\lambda\cap C$, and if $\gamma<\xi$ is in $K$ then $\lambda_\gamma<\lambda_\xi$ and $X_{\gamma}\cup\{\lambda_\gamma\}\subseteq X_\xi$,
\item\label{deltacondition} $X_\xi\prec(Q_{\lambda_\xi};\in, \PPdia_{\lambda_\xi}, A_{\lambda_\xi})$ and $\delta^{X_\xi}=\xi$.
\end{enumerate}
\end{enumerate}

If $\mathfrak{C}$ has these properties, we call $\mathfrak{C}$ a \textit{potential certificate}\index{Potential certificate}.\medskip

Next up, we will define a certain first order language $\mathcal L$. $\lang$ will have the following distinguished constants
\begin{itemize}
\item \ul{$x$} for any $x\in H_\kappa$,
\item $\dot n$ for any $n<\omega$,
\item $\dot M_i$ for $i<\omega_1$,
\item $\dot\mu_{i, j}$ for $i\leq j\leq\omega_1$,
\item $\dot{\vec M}$,
\item $\dot N_i$ for $i<\omega_1$,
\item $\dot \sigma_{i, j}$ for $i\leq j<\omega_1$,
\item $\dot I$, $\dot a$ and
\item $\dot X_\xi$ for $\xi<\omega_1$.
\end{itemize}

The constants $\dot n$ will eventually produce ``Henkin-style" term models for the $N_i$. Formulas in the language $\mathcal L$ are of the form 
$$\ulc\dot N_i\models\varphi(\ul{\gamma_1},\dots, \ul{\gamma_k}, \dot n_1,\dots, \dot n_l, \dot I, \dot a, \dot M_{j_1},\dots, \dot M_{j_m}, \dot \mu_{q_1, r_1}, \dots, \dot \mu_{q_s, r_s}, \dot{\vec M})\urc$$
where 
\begin{itemize}
\item $i<\omega_1$,
\item $\gamma_1,\dots \gamma_k<\omega_1$,
\item $n_1, \dots, n_l<\omega$,
\item $j_1,\dots, j_m<\omega_1$,
\item $q_t\leq r_t<\omega_1$ for $t\in\{1,\dots, s\}$
\end{itemize}
and $\varphi$ is a first order $\in$-formula. Moreover we allow as formulas
\begin{itemize}
\item $\ulc\dot\mu_{i, \omega_1}(\dot n)=\ul{x}\urc$ for $i<\omega_1, n<\omega$ and $x\in H_{\omega_2}$,
\item $\ulc\dmu_{\omega_1,\omega_1}(\ulx)=\ulx\urc$ for $x\in \Htwo$,
\item $\ulc\dsigma_{i, j}(\dn)=\dot m\urc$ for $i\leq j<\omega_1$ and $n, m<\omega$,
\item $\ulc(\vec{\ul{k}}, \vec{\ul{\alpha}})\in\ul{T}\urc$ for $\vec k\in\omega^{{<}\omega}$ and $\vec\alpha\in\omega_2^{{<}\omega}$,
\item $\ulc\ul{\xi}\mapsto\ul{\nu}\urc$ for $\xi<\omega_1$ and $\nu<\kappa$ and
\item $\ulc\ulx\in\dot X_\xi\urc$ for $\xi<\omega_1$ and $x\in\Hkappa$.
\end{itemize}

$\lang^\lambda$ is the set of $\lang$-formulae $\varphi$ so that if $\ulx$ appears in $\varphi$ for some $x\in\Hkappa$ then $x\in Q_\lambda$. We assume formulae in $\lang^\lambda$ to be coded in a reasonably way (ultimately uniform in $\lambda$) so that $\lang^\lambda=\lang\cap Q_\lambda$. We will not make this precise.\medskip

A potential certificate 
$$\normalformcert$$
is ($\lambda$-)\textit{precertified}\index{Potential certificate!lambda precertified@$\lambda$-precertified} by $\Sigma\subseteq\lang^\lambda$ if there are surjections $e_i\colon \omega\rightarrow N_i$ for $i<\omega_1$ so that
\begin{enumerate}[label=($\Sigma$.\arabic*)]
\item\label{syncertcondsfirst} $\normalformformula\in\Sigma$ iff
\begin{enumerate}
\item $i<\omega_1$,
\item $\gamma_1,\dots, \gamma_k\leq\omega_1^{N_i}$,
\item $n_1,\dots, n_l <\omega$,
\item $j_1,\dots, j_m\leq\omega_1^{N_i}$,
\item $q_t\leq r_t\leq\omega_1^{N_i}$ for $t\in\{1,\dots, s\}$
\end{enumerate}
and 

\begin{align*}
N_i\models\varphi(&\gamma_1,\dots, \gamma_k, e_i(n_1),\dots, e_i(n_l), I^{N_i}, a^{N_i},\\ &M_{j_1}, \dots, M_{j_m}, \mu_{q_1, r_1},\dots, \mu_{q_s, r_s}, \vec{M})
\end{align*}

where $\vec M=\langle M_j, \mu_{j, j'}\mid j\leq j'\leq\omega_1^{N_i}\rangle$,
\item $\normalfone\in\Sigma$ iff $i<\omega_1$, $n<\omega$ and $\mu_{i, \omega_1}(e_i(n))=x$,
\item $\normalftwo\in\Sigma$ for all $x\in\Htwo$,
\item $\normalfthree\in\Sigma$ iff $i\leq j<\omega_1$ and $\sigma_{i, j}(e_i(n))=e_j(m)$,
\item $\ulc(\vec{\ul{l}}, \vec{\ul{\beta}})\in\ul{T}\urc\in\Sigma$ iff for some $n<\omega$, $\lh(\vec l)=n=\lh(\vec\beta)$ and for all $m<n$ $l_m=k_m$, $\beta_m=\alpha_m$,
\item $\normalffive\in\Sigma$ iff $\xi\in K$ and $\nu=\lambda_\xi$ and
\item $\normalfsix\in\Sigma$ iff $\xi\in K$ and $x\in X_\xi$.
\end{enumerate}

Note that $\cert$ can be ``read off" from $\Sigma$ in a unique way via a Henkin-style construction. For $i<\omega_1$ and $n, m<\omega$, let 
$$n\sim_i m\Leftrightarrow \ulc N_i\models \dot n =\dot m\urc\in\Sigma$$
and denote the equivalence class of $n$ modulo $\sim_i$ by $[n]_i^{\Sigma}$. We will usually drop the superscript $\Sigma$ if it is clear from context. Also let
$$n\tilde{\in}_i m\Leftrightarrow\ulc N_i\models \dot n \in \dot m\urc\in\Sigma.$$
Then $(N_i, \in)\cong (\omega, \tilde{\in}_i)/\sim_i$. We call the latter model the term model producing $N_i$.  See Lemma 3.7 in \cite{aspsch} for more details. For $x\in N_i$ we say $x$ is represented by $n$ if $x$ gets mapped to $[n]_i$ by the unique isomorphism of $N_i$ to the term model. The term model for $N_{\omega_1}$ is then the direct limit along the term models producing the $N_i$, $i<\omega_1$ and elements can then be represented by pairs $(i, n)$, $i<\omega_1, n<\omega$ in the natural way. 

To define certificates, we make use of the following concept:

\begin{defn}
For $\bar\lambda\in C\cap\lambda$,
$$Z\subseteq\PPdia_{\bar\lambda}\times\omega_1\times\omega$$
is a $\bar\lambda$-code\index{lambda code@$\lambda$-code!for a dense subset@for a dense subset of $(I^+)^{\dot N_{\omega_1}}$} for a dense subset of $(I^+)^{\dot N_{\omega_1}}$ given that
\begin{enumerate}[label=\rn*]
\item if $(p, i, n)\in Z$ then
$$\ulc \dot N_i\models``\dot n\in \dot I_i^+"\urc\in p,$$
\item for any $(q, j, m)\in\PP_{\bar\lambda}\times\omega_1\times\omega$ with
$$\ulc \dot N_j\models``\dot m\in \dot I_j^+"\urc\in q$$
there is $(p, i, n)\in Z$ with 
\begin{enumerate}[label=$(\alph*)$]
\item $p\leq q$, $j\leq i$ and
\item $\ulc\dot N_i\models``\dot n\subseteq \dot k\mod \dot I_i"\urc, \ulc\dot\sigma_{j, i}(\dot m)=\dot k\urc\in p$ for some $k<\omega$,
\end{enumerate} 
\item and if $(p, i, n)\in Z$ as well as $q\leq p$ then $(q, i, n)\in Z$.
\end{enumerate}
Suppose that 
$$\normalformcert$$
is $(\lambda-)$precertified by $\Sigma\subseteq\lang^\lambda$ as witnessed by $(e_i)_{i<\omega_1}$. For $Z_0\subseteq Z$ we define the evaluation of $Z_0$ by $\Sigma$ as\index{Z0 Sigma@$Z_0^\Sigma$ (the evaluation of $Z_0$ by $\Sigma$)}
$$Z_0^{\Sigma}\coloneqq\{S\in N_{\omega_1}\mid\exists p\in [\Sigma]^{{<}\omega}\exists i<\omega_1\exists n<\omega\ ((p, i, n)\in Z_0\wedge S=\sigma_{i,\omega_1}(e_i(n)))\}.$$
\end{defn}

A potential certificate $\cert$ is ($\lambda$-)\textit{certified}\index{Potential certificate!lambda certified@$\lambda$-certified} by a collection $\Sigma\subseteq\lang^\lambda$ if $\cert$ is $(\lambda$-)precertified by $\Sigma$ and additionally
\begin{enumerate}[resume*]
\item\label{newgenericitycondition} whenever $\xi\in K$ and $Z$ is a $\lambda_\xi$-code for a dense subset of $(I^+)^{\dot N_{\omega_1}}$ definable over 
$$(Q_{\lambda_\xi};\in, \PPdia_{\lambda_\xi}, A_{\lambda_\xi})$$
from parameters in $X_\xi$, then there is $S\in (Z\cap X_\xi)^{\Sigma}$ with $\xi\in S$.
\end{enumerate}
\begin{defn}
In the case that \ref{newgenericitycondition} is satisfied, we call $\cert$ a semantic certificate\index{Semantic certificate}, and $\Sigma$ a syntactic certificate\index{Syntactic certificate}, relative to 
$$\Vmax, A, \Htwo, T_0, \langle A_\nu\mid \nu\in C\cap\lambda\rangle\text{ and }\langle\PPdia_{\nu}\mid\nu\in C\cap\lambda\rangle.$$
\end{defn}

\begin{rem}
The genericity condition in \cite{aspsch} that is replaced here with \ref{newgenericitycondition} (adapted to our context) is:
\begin{enumerate}[label=$(\Sigma.\arabic*)^{\mathrm{AS}}$, start=8]
\item\label{aspschgenericitycond} If $\xi\in K$ and $E\subseteq\PPdia_{\lambda_\xi}$ is dense and definable over 
$$(Q_{\lambda_\xi};\in, \PPdia_{\lambda_\xi}, A_{\lambda_\xi})$$
from parameters in $X_\xi$ then 
$$[\Sigma]^{<\omega}\cap E\cap X_\xi\neq\emptyset.$$
\end{enumerate}
Condition \ref{newgenericitycondition} is stronger than \ref{aspschgenericitycond}: From any such $E$, 
$$Z=\{(p, i, n)\in \PPdia_{\bar\lambda}\times\omega_1\times\omega\mid \exists q\in E\ p\leq q\wedge \ulc \dot N_i\models``\dot n\in \dot I_i^+"\urc\in p\}$$
is a $\lambda_\xi$-code for a dense subset of $(I^+)^{\dot N_{\omega_1}}$ definable over the same structure from the same parameters. If $(Z\cap X_\xi)^{\Sigma}\neq \emptyset$, it follows that 
$$[\Sigma]^{<\omega}\cap E\cap X_\xi\neq\emptyset.$$
\end{rem}
Suppose $\Sigma$ is a certificate that certifies 
$$\normalformcert,$$
$\xi\in K$ and $Z$ is a $\lambda_\xi$-code for a dense subset of $(I^+)^{\dot N_{\omega_1}}$ definable over 
$$(Q_{\lambda_\xi};\in, \PPdia_{\lambda_\xi}, A_{\lambda_\xi}).$$
$Z$ is supposed to represent a dense subset of $(I^+)^{N_{\omega_1}}$ (w.r.t. inclusion $\mod I^{N_{\omega_1}}$) in $V^{\PPdia_\lambda}$. $\Sigma$ may not be ``generic over $V$", so it may not be the case that $Z^\Sigma$ is dense in $(I^+)^{N_{\omega_1}}$. Nonetheless, already \ref{aspschgenericitycond} implies that 
$$D=\sigma_{\xi,\omega_1}^{-1}[(Z\cap X_\xi)^{\Sigma}]\subseteq (I^+)^{N_\xi}$$
is dense. $D$ may not be in $N_\xi$, so it is not guaranteed that $D$ is hit by the ultrapower $\sigma_{\xi,\xi+1}\colon N_\xi\rightarrow N_{\xi+1}$ just from genericity over $N_\xi$ alone, however \ref{newgenericitycondition} makes sure that this happens (observe that $\omega_1^{N_\xi}=\xi$). So in essence, the idea of \ref{newgenericitycondition} is that any dense subset of $(I^+)^{N_{\omega_1}}$ that exists in the final $V^{\PPdia_\kappa}$ has been ``hit" before at some point along the iteration of $N_0$ to $N_{\omega_1}$.

\begin{rem}
Note that for any syntactic certificate, there is a unique semantic certificate it corresponds to. Given a semantic certificate, its corresponding syntactic certificate is unique modulo the choice of the maps $(e_i)_{i<\omega}$.
\end{rem}

A finite set $p$ of $\lang^\lambda$-formulas is \textit{certified} by $\Sigma$ iff $\Sigma$ is a syntactic certificate and $p\subseteq\Sigma$. If $\cert$ is a semantic certificate then we also say $p$ is certified by $\cert$ in case there is a syntactic certificate $\lambda$ certifying both $\cert$ and $p$. 

\begin{defn}
Conditions $p\in\PPdia_\lambda$\index{Pdiamondalambda@$\PPdia_\lambda$} are finite sets of $\lang^\lambda$ formulae so that 
$$V^{\Col(\omega, \omega_2)}\models``\exists\Sigma\subseteq\lang^\lambda\ \Sigma\text{ certifies }p".$$
\end{defn}

This completes the construction of $\PPdia_\lambda$.
\begin{prop}\label{certinsamllcolprop}
Let $p\in[\lang^\lambda]^{<\omega}$. If $p$ is certified in some outer model, then $p$ is certified in $V^{\Col(\omega,\omega_2)}$.
\end{prop}
\begin{proof}
Let $g$ be $\Col(\omega,\omega_2)$-generic. If there is some outer model in which $p$ is certified, then by Shoenfield absoluteness we can find in $V[g]$ a set of $\lang^\lambda$-formulas $\Sigma$ with $p\in[\Sigma]^{<\omega}$ such that if 
$$\normalformcert$$
is the corresponding semantic interpretation then 
\begin{enumerate}[label=$(\roman*)$]
\item $\Sigma$ satisfies \ref{syncertcondsfirst}-\ref{newgenericitycondition},
\item $\cert$ satisfies \ref{semcertcondssecond} as well as \ref{semcertcondsfourth}-\ref{semcertcondslast} and
\item $\cert$ satisfies \ref{semcertcondsthird} in the sense that $\mu_{0, \omega_1^{N_0}}\in N_0$ and $N_0\models \varphi^{\Vmax}(M_{\omega_1^{N_0}})$,
\end{enumerate}
as this can be expressed by a $\mathbf{\Sigma}^1_2$-formula. It remains to show that \ref{semcertcondsfirst} holds true as well, i.e.~$M_0, N_0\in\Vmax$. For $N_0$ this follows as $N_0\in p[T_0]$ and by assumption \ref{diastarforcingdensecond}, $p[T_0]\subseteq\Vmax$ in $V[g]$. To see that $M_0\in\Vmax$, note that $\mathcal H\in\Vmax$ as $\mathcal H$ is almost a $\Vmax$-condition in $V$. By \ref{Vmaxiteriesareconds}, it is enough to see that $M_0$ is generically iterable. This follows from (the proof of) Theorem 3.16 in \cite{woodinbook}, here we use $\Pomo^\sharp$ exists in $V$.
\end{proof}
We let $\PPdia=\PPdia_\kappa$. As in \aspsch, we conclude that there is a club $D\subseteq C$ so that for all $\lambda\in D$
$$\PPdia_\lambda=\PPdia\cap Q_\lambda$$ 
and hence we get 

\begin{align*}
 &\text{for all } B\subseteq H_\kappa\text{ the set }\\
(\diamondsuit(\PPdia))\hspace{10pt}&\hspace{10pt}\{\lambda\in C\mid (Q_\lambda;\in, \PPdia_\lambda, A_\lambda)\prec (H_\kappa;\in, \PP, B)\}\\
&\text{is stationary.}
\end{align*}

\begin{lemm}
$\emptyset\in \PPdia_{\min(C)}$.
\end{lemm}
The argument is essentially the same as the proof of Lemma 3.6 in \cite{aspsch} modulo some details that arise from replacing $\Pmax$ by a general $\Pmaxvariant$.
\begin{proof}
Let $g$ be generic for $\Col(\omega, \omega_2)$. Note that $\mathcal H\in\Vmax$ as $\mathcal H$ is almost a $\Vmax$-condition in $V$. By choice of $T_0$, we can find $N_0=(N_0, I_0, a_0)\in D^\ast$ with $N_0<_{\Vmax}\mathcal H$. Let $\langle (k_n,\alpha_n)\mid n<\omega\rangle$ witness $N_0\in p[T]$. Let us denote $M_0=\mathcal H$ and let
$$\mu_{0, \omega_1^{N_0}}\colon M_0\rightarrow M_{\omega_1^{N_0}}$$
witness $N_0<_{\Vmax} M_0$. Now let 
$$\sigma_{0, \kappa}\colon N_0\rightarrow N_\kappa$$
be a generic iteration of $N_0$ of length $\kappa+1=\omega_1^{V[g]}+1$ as well as 
$$\mu_{0,\kappa}\coloneqq \sigma_{0,\kappa}(\mu_{0,\omega_1^{N_0}})\colon M_0\rightarrow M_\kappa$$
the stretch of $\mu_{0,\omega_1^{N_0}}$ by $\sigma_{0, \kappa}$. Note that this is a generic iteration of $M_0$ of length $\kappa+1$. 
\begin{claim}\label{extendclaim}
The generic iteration 
$$\langle M_\alpha, \mu_{\alpha,\beta}\mid\alpha\leq\beta\leq\kappa\rangle$$ 
can be extended to a generic iteration of $M_0^+\coloneqq (V, \NS^V)$ of length $\kappa+1$. That is, there is a generic iteration 
$$\langle M_\alpha^+,\mu^+_{\alpha,\beta}\mid\alpha\leq\beta\leq\kappa\rangle$$
of $M_0^+$ so that for all $\alpha\leq\beta\leq\kappa$
\begin{enumerate}[label=$(+.\roman*)$]
\item\label{pluscond1} $M_\alpha=\left(\Htwo\right)^{M_\alpha^+}$ and
\item\label{pluscond2} $\mu_{\alpha,\beta}=\mu_{\alpha,\beta}^+\res M_\alpha$.
\end{enumerate}
\end{claim}
\begin{proof}
The iteration $\langle M_\alpha^+,\mu^+_{\alpha,\beta}\mid\alpha\leq\beta\leq\kappa\rangle$ arises by applying the same generic ultrafilter $g_\alpha$ which generates $\mu_{\alpha,\alpha+1}\colon M_\alpha\rightarrow M_{\alpha+1}$ to $M_\alpha^+$. By induction on $\alpha$, as $M_\alpha=\left(\Htwo\right)^{M_\alpha^+}$, $g_\alpha$ measures all subsets of $\omega_1^{M_\alpha^+}$ in $M_\alpha^+$. It is a generic ultrafilter as 
$$M_\alpha^+\models``\NS\text{ is saturated}"$$
by elementarity of $\mu_{0,\alpha}^+$, and hence all maximal antichains in $(\NS^+)^{M_\alpha^+}$ are already in $M_\alpha$, hence are met by $g_\alpha$. Now let 
$$\mu_{\alpha,\alpha+1}^+\colon M_\alpha^+\rightarrow M_{\alpha+1}^+\coloneqq \Ult(M_\alpha^+, g_\alpha)$$
be the ultrapower. Any $x\in (\Htwo)^{M_{\alpha+1}^+}$ is represented by some function 
$f\colon \omega_1^{M_\alpha^+}\rightarrow \left(\Htwo\right)^{M_\alpha^+}$ which is an element of $\left(\Htwo\right)^{M_\alpha^+}=M_\alpha$. It follows that $\mu_{\alpha,\alpha+1}=\mu_{\alpha,\alpha+1}^+\res M_\alpha$. It is easy to see that the properties \ref{pluscond1},\ref{pluscond2} are stable under taking direct limits.
\end{proof}
The point is that 
$$\langle\langle M_i,\mu_{i, j}, N_i,\sigma_{i, j}\mid i\leq j\leq\omega_1\rangle,\langle (k_n,\alpha_n)\mid n<\omega\rangle, \emptyset\rangle$$
is a semantic certificate for $\emptyset$ in $M^+\coloneqq M_\kappa^+$ with respect to 
$$\mu^+(\Vmax), \mu^+(A), \left(\Htwo\right)^{M^+}, \mu^+(T_0), \mu^+(\langle A_\nu\mid \nu\in C\cap\lambda\rangle), \mu^+(\langle \PPdia_\nu\mid\nu\in C\cap\lambda\rangle)$$
for $\lambda=\min(C)$ and $\mu^+=\mu_{0,\kappa}^+$. By Proposition \ref{certinsamllcolprop}, 
$$M^+\models\emptyset\in\mu^+(\PPdia_{\min (C)})$$
so that $\emptyset\in \PPdia_{\min(C)}$ in $V$ by elementarity of $\mu^+$.
\end{proof}

\begin{lemm}\label{certificatelemm}
Suppose $\lambda\in C\cup\{\kappa\}$ and $g\subseteq\PPdia_\lambda$ is a filter with 
\begin{enumerate}[label=$(\roman*)$]
\item $g\cap E\neq\emptyset$
whenever $E\subseteq\PPdia_\lambda$ is dense and definable over 
$$(Q_\lambda;\in, \PPdia_\lambda, A_\lambda),$$
\item $g$ is an element of a generic extension of $V$ by a forcing of size $\leq 2^{\omega_2}$.
\end{enumerate}
Then $\bigcup g$ is a semantic certificate.
\end{lemm}

\begin{proof}
Read off the canonical candidate 
$$\normalformcert$$
from $g$. The proof of Lemma 3.7 in \cite{aspsch} shows that $\bigcup g$ $\lambda$-precertifies $\cert$. Note that the argument from Proposition \ref{certinsamllcolprop} gives that $M_0, N_0\in\Vmax$ and \ref{semcertcondsthird} follows from \ref{syncertcondsfirst} and  \ref{Vmaxvarphi}. It remains to check \ref{newgenericitycondition}. So suppose $\xi\in K$ and $Z$ is a $\lambda_\xi$-code for a dense subset of $(I^+)^{\dot N_{\omega_1}}$ definable over 
$$\mathcal Q_\lambda\coloneqq (Q_{\lambda_\xi};\in, \PPdia_{\lambda_\xi}, A_{\lambda_\xi})$$
from a parameter $x\in X_\xi$. Then there is $p\in g$ with 
$$\ulc \ul{\xi}\mapsto\ul{\lambda_\xi}\urc, \ulc \ul{x}\in\dot X_\xi\urc\in p.$$
Let $\Sigma^\prime$ be a syntactic certificate certifying $p$ (in some extension of $V$ by $\Col(\omega,\omega_2)$) and 
$$\normalformcertprime$$
the corresponding semantic certificate. We have $\xi\in K$ and $\lambda_\xi'=\lambda_\xi$ as well as $x\in X_\xi'$. Thus $Z$ is definable over $\mathcal Q_\lambda$ from parameters in $X_\xi'$. As $\Sigma'$  satisfies \ref{newgenericitycondition}, there is $S\in (Z\cap X^\prime_\xi)^{\Sigma'}$ with $\xi\in S$. We may now find $(q, i, n)\in Z\cap X_\xi^\prime$ so that 
$$S=\sigma_{i,\omega_1}([n]_i^{\Sigma'}).$$
Note that $i<\xi$ as $\delta^{X_\xi^\prime}=\xi$. Let $\sigma_{i,\xi+1}([n]_i^{\Sigma'}]=[m]_{\xi+1}^{\Sigma'}$. It follows that 
$$\ulc \dot N_{\xi+1}\models``\ul{\xi}\in \dot m"\urc, \ulc\dot\sigma_{i, \xi+1}(\dot n)=\dot m\urc\in \Sigma'.$$
This is a density argument that shows: There are $s\geq r\in g$, $j<\xi$, $l<\omega$ so that 
\begin{enumerate}[label=$(\roman*)$]
\item $(s, j, l)\in Z$,
\item $\ulc \ul{s}\in \dot X_\xi\urc\in r$ and
\item $\ulc\dot N_{\xi+1}\models``\ul{\xi}\in\dot k"\urc, \ulc\dot\sigma_{j, \xi+1}(\dot l)=\dot k\urc\in r$ for some $k<\omega$.
\end{enumerate}
It follows that for $S=\sigma_{j,\omega_1}([l]_j^{\bigcup g})$, we have $S\in (Z\cap X_\xi)^{\bigcup g}$ and $\xi\in S$.
\end{proof}

\begin{lemm}\label{strongpreservationlemm}
Suppose $g$ is generic for $\PPdia$ and 
$$\normalformcert$$
is the resulting semantic certificate. Then in $V[g]$, 
$$\langle N_i, \sigma_{i, j}\mid i\leq j\leq\omega_1\rangle$$
is a $\diamondsuit$-iteration.
\end{lemm}

\begin{proof}
Let $\dot S, \dot C$ be $\PPdia$-names with 
$$p\forces``\dot C\subseteq\omega_1\text{ is club and }\dot S\in(\dot I^+)^{\dot N_{\omega_1}}"$$
for some $p\in\PPdia$. Further suppose $\langle \dot D_\alpha\mid\alpha<\omega_1\rangle$ is a sequence of $\PPdia$-names for dense subsets of $(I^+)^{\dot N_{\omega_1}}$. We may suppose that
\begin{align*}
p\forces\dot S=\dot\sigma_{i_0,\omega_1}([\check n]_{\check i_0}^{\bigcup \dot G})
\end{align*}
for some $i_0<\omega_1$ and $n<\omega$ where $\dot\sigma_{i_0,\omega_1}$ is a name for $\sigma_{i_0, \omega_1}$ which arises in the semantic certificate corresponding to the generic filter. It is our duty to find $\xi<\omega_1$ and $q\leq p$ with 
\begin{equation*}\label{fstationarypresqeq}\tag{$\spadesuit$}
q\forces\check\xi\in\dot S\cap\dot C\wedge\forall \alpha<\check\xi\ \dot g_\xi\cap\dot\sigma_{\xi,\omega_1}^{-1}[\dot D_\alpha]\neq\emptyset
\end{equation*}
where $\dot g_\xi$ is a name for the generic ultrafilter applied to $\dot N_\xi$ along the iteration to $\dot N_{\omega_1}$. We will replace the $\dot D_\alpha$ with codes for them: For $\alpha<\omega_1$, let $Z_\alpha$ be defined by $(q, j, m)\in Z_\alpha$ iff
\begin{enumerate}[label=$(Z.\roman*)$]
\item $(q, j, m)\in\PPdia\times\omega_1\times\omega$,
\item $\ulc\dot N_j\models``\dot m\in\dot I_j"\urc\in q$ and
\item $q\forces\dot\sigma_{j, \omega_1}\left([m]_j^{\bigcup \dot G}\right)\in\dot D_\alpha$.
\end{enumerate}
Further, for $\alpha<\omega_1$, we let 
$$E_\alpha=\{q\leq p \mid\exists\beta\ \alpha\leq\beta\wedge q\forces\check\beta\in\dot C\}$$
and 
$$E=\{(q, \alpha)\in\PPdia\times\omega_1\mid q\forces\check\alpha\in\dot C\}.$$
Finally we define 
$$\tau=\left(\bigoplus_{\alpha<\omega_1} Z_\alpha\right)\oplus \left(\bigoplus_{\alpha<\omega_1} E_\alpha\right)\oplus E.$$
We may now find $\lambda\in C$ so that $p\in \PPdia_\lambda$ and 
$$(Q_\lambda;\in, \PPdia_\lambda, A_\lambda)\prec (\Hkappa;\in, \PPdia, \tau).$$
 Here, $\oplus$ denotes some canonical way of coding at most $\omega_1$-many subsets of $\Hkappa$ into a subset of $H_\kappa$. Let $h$ be $\Col(\omega, \omega_2)$-generic over $V$.
 
\begin{claim}
In $V[h]$, there are filters $g, G$ that satisfy the following properties \ref{forcingexistencegGcond1}-\ref{forcingexistencegGcond3}:
\begin{enumerate}[label=\rn*]
\item\label{forcingexistencegGcond1} $g$ meets every dense subset of $\PPdia_\lambda$ that is definable (with parameters) in 
$$(Q_\lambda;\in, \PPdia_\lambda, A_\lambda).$$
\end{enumerate}
Let 
$$\normalformcert$$ 
denote the semantic certificate corresponding to $g$.
\begin{enumerate}[label=\rn*, resume]
\item $G$ is $(I^+)^{N_{\omega_1}}$-generic over $N_{\omega_1}$ with $\dot S^g=[n]_{i_0}^{\bigcup g}\in G$.
\item\label{forcingexistencegGcond3} $G$ meets $Z^{\bigcup g}$ whenever $Z$ is a $\lambda$-code for a dense subset of $(\dot I^+)^{\dot N_{\omega_1}}$ definable (with parameters) over 
$$(Q_\lambda;\in, \PPdia_\lambda, A_\lambda).$$
\end{enumerate}
\end{claim}

\begin{proof}
Let $g'\subseteq\PPdia_\lambda$ be generic over $V$ and let
$$\normalformcertprime$$
be the semantic certificate corresponding to $\bigcup g'$. Let further $G'$ be $(I^+)^{N_{\omega_1}'}$-generic over $V[g']$ (so in particular over $N_{\omega_1}'$) with $\dot S^{g'}=[n]_{i_0}^{\bigcup g'}\in G'$. It is clear that $g', G'$ satisfy \ref{forcingexistencegGcond1}-\ref{forcingexistencegGcond3} above. The existence of such filters is $\Sigma^1_1$ in a real code for $(Q_\lambda;\in, \PPdia_\lambda, A_\lambda)$ so that there are $g, G\in V[h]$ with \ref{forcingexistencegGcond1}-\ref{forcingexistencegGcond3} by Shoenfield-absoluteness.
\end{proof}

We now work in $V[h]$. Let $G, g$ be the filters given by the claim above and let 
 $$\normalformcert$$ be the semantic certificate that comes from $g$. Let
$$\sigma_{\omega_1,\omega_1+1}\colon N_{\omega_1}\rightarrow N_{\omega_1+1}=\Ult(N_{\omega_1}, G)$$
be the generic ultrapower.  We can further extend the generic iteration 
$$\langle N_i, \sigma_{i, j}\mid i\leq j\leq \omega_1+1\rangle$$
to one of length $\kappa+1$, say
$$\langle N_i, \sigma_{i, j}\mid i\leq j\leq \kappa\rangle.$$
Further, set 
$$\vec M=\langle M_i, \mu_{i, j}\mid i\leq j\leq\kappa\rangle\coloneqq \sigma_{\omega_1, \kappa}(\langle M_i, \mu_{i, j}\mid i\leq j\leq\omega_1\rangle).$$
As $\cert$ is certified, $M_{\omega_1}=\mathcal H$ and as in Claim \ref{extendclaim}, we can extend the tail of $\vec M$ that is an iteration of $M_{\omega_1}$ to a generic iteration of $M_{\omega_1}^+\coloneqq (V, \NS^V, A)$, say
$$\langle M_i^+, \mu_{i, j}^+\mid \omega_1\leq i \leq j \leq \kappa\rangle$$
and have all $M_i^+$, $i\in[\omega_1, \kappa]$, wellfounded. Let us write 
$$\mu^+\coloneqq \mu_{\omega_1, \kappa}^+\colon V\rightarrow M_{\omega_1}^+=:M^+.$$
Work in $M^+$. We will now use 
$$\langle M_i, \mu_{i, j}, N_i, \sigma_{i, j}\mid i\leq j \leq\kappa\rangle$$
as part of a certificate. Set 
$$q\coloneqq \mu^+(p)\cup\{\ulc\ul{\omega_1}\mapsto\mu^+(\lambda)\urc, \ulc \dot\sigma_{i_0, \omega_1+1}(\dot n)=\dot m\urc, \ulc \dot N_{\omega_1+1}\models``\ul{\omega_1}\in\dot m"\urc\}$$
where $\dot m$ represents $\sigma_{\omega_1,\omega_1+1}(S)$ in the term model for $N_{\omega_1+1}$.
\begin{claim}\label{isaconditionclaim}
$q\in\mu^+(\PPdia)$.
\end{claim}

\begin{proof}
Set
$$\cert^\ast=\langle\langle M_i, \mu_{i, j}, N_i, \sigma_{i, j}\mid i\leq j \leq\kappa\rangle, \langle (k_n, \mu^+(\alpha_n))\mid n<\omega\rangle, \langle \lambda_\xi^\ast, X_\xi^\ast\mid \xi\in K^\ast\rangle\rangle$$
where 
\begin{itemize}
\item $K^\ast=K\cup\{\omega_1\}$,
\item for $\xi\in K$, $\lambda_\xi^\ast=\mu^+(\lambda_\xi)$ and $X_\xi^\ast=\mu^+[X_\xi]$ and
\item $\lambda_{\omega_1}=\mu^+(\lambda)$, $X_{\omega_1}^\ast=\mu^+[Q_\lambda]$.
\end{itemize}
We show that $\cert^\ast$ is a semantic certificate for $q$ in $M^+$. Note that we have to show that $\cert^\ast$ is a certificate relative to 
$$\mu^+(\Vmax), \mu^+(A), \mu^+(\Htwo)=(\Htwo)^{M^+}, \mu^+(T_0), \mu^+(\langle A_\nu\mid \nu\in C\rangle), \mu^+(\langle \PP_\nu\mid \nu\in C\rangle).$$
Observe that we can find a corresponding set of formulae $\Sigma^+$ that corresponds to $\cert^\ast$ with $\mu^+[\bigcup g]\subseteq\Sigma^+$ which we aim to prove to be a syntactic certificate.\\
We have $M_\kappa=\left(\Htwo\right)^{M^+}$. Notice also that 
$$\langle(k_n, \mu^+(\alpha_n))\mid n<\omega\rangle\in [\mu^+(T_0)]$$
and that $(k_n)_{n<\omega}$ is still a real code for $N_0$. Next, we prove \ref{newgenericitycondition}. First assume $\xi\in K$. Then 
$$X_\xi^\ast=\mu^+[X_\xi]\prec (\mu^+(Q_{\lambda_\xi});\in, \mu^+(\PPdia_{\lambda_\xi}), \mu^+(A_{\lambda_\xi}))$$
and $\delta^{X_\xi^\ast}=\delta^{X_\xi}=\xi$ as $\crit(\mu)=\omega_1>\xi$. As $\mu^+[X_\xi
]=X_\xi^\ast$, \ref{newgenericitycondition} holds for $\xi$ in $\cert^\ast$, since it holds for $\xi$ in $\cert$.
\\
Finally, let us consider the case $\xi=\omega_1$. We have 
$$X_{\omega_1}^\ast=\mu^+[Q_\lambda]\prec (\mu^+(Q_\lambda);\in, \mu^+(\PPdia_\lambda), \mu^+(A_\lambda))$$
and $\delta^{X_{\omega_1}^\ast}=\omega_1$ as $\mu^+$ has critical point $\omega_1$. Clearly $X_{\omega_1}^\ast$ collapses to $Q_\lambda$. So if $x\in X_{\omega_1}^\ast$ and
\begin{align*}
M^+\models``&\hat Z\text{ is a }\mu^+(\lambda)\text{-code for a dense subset of } (\dot I^+)^{N_\kappa}\text{ definable over} \\
&\hspace{60pt}(\mu^+(Q_\lambda);\in, \mu^+(\PPdia_\lambda), \mu^+(A_\lambda))\\
&\text{with parameter }x"
\end{align*}
for some $x\in X_{\omega_1}^\ast$, then by elementarity, the same definition defines a $\lambda$-code $Z$ for a dense subset of $(\dot I^+)^{\dot N_{\omega_1}}$ over 
$$(Q_\lambda; \in, \PPdia_\lambda, A_\lambda)$$
with parameter $(\mu^+)^{-1}(x)$ and we have  $\mu^+(Z)= \hat Z$.
Our properties of $g, G$ imply that there is $R\in G\cap Z^{\bigcup g}$. It is not difficult to see 
$$(\hat Z\cap X_{\omega_1}^\ast)^{\Sigma^+}=\sigma_{\omega_1, \kappa}[Z^{\bigcup g}]$$
and hence $\omega_1\in\sigma_{\omega_1, \kappa}(R)\in (\hat Z\cap X_{\omega_1}^\ast)^{\Sigma^+}$. This shows \ref{newgenericitycondition} at $\omega_1$.\\
We conclude that indeed, $\cert^\ast$ is a semantic certificate for $q$ which exists in some outer model of $M^+$. This gives $q\in \mu^+(\PPdia)$ by Proposition \ref{certinsamllcolprop}.
\end{proof}

Thus we have 
\begin{align*}
M^+\models &``\exists\xi<\mu^+(\omega_1)\\ &\left(\mu^+(p)\cup\{\ulc\ul{\xi}\mapsto\ul{\mu^+(\lambda)}\urc, \ulc \dot\sigma_{i_0, \xi+1}(\dot n)=\dot m\urc, \ulc \dot N_{\xi+1}\models``\ul{\xi}\in\dot m"\urc\}\in\mu^+(\PPdia)\right)".
\end{align*}
By elementarity of $\mu^+$, we conclude 
$$V\models``\exists \xi<\omega_1\ \left(p\cup\{\ulc\ul{\xi}\mapsto\ul{\lambda}\urc, \ulc \dot\sigma_{i_0, \xi+1}(\dot n)=\dot m\urc, \ulc \dot N_{\xi+1}\models``\ul{\xi}\in\dot m"\urc\}\in\PPdia\right)".$$
Let $\xi$ witness this and set 
$$q=p\cup\{\ulc\ul{\xi}\mapsto\ul{\lambda}\urc, \ulc \dot\sigma_{i_0, \xi+1}(\dot n)=\dot m\urc, \ulc \dot N_{\xi+1}\models``\ul{\xi}\in\dot m"\urc\}.$$
We will show that $q,\xi$ witness (\ref{fstationarypresqeq}). From this point on, we work in $V$ again and forget about $h, g, \cert$, etc.

\begin{claim}\label{fstationarypresclaim1}
$q\forces\check\xi\in\dot C\cap\dot S$.
\end{claim}

\begin{proof}
As in Claim 3.17 in \cite{aspsch}, exploit the components of $\tau$ made up from $E$ as well as $E_\alpha$, $\alpha<\omega_1$.
\end{proof}

\begin{claim}\label{fstationarypresclaim2}
$q\forces\forall \alpha<\check\xi\ \dot g_\xi\cap \dot\sigma_{\xi,\omega_1}^{-1}[\dot D_\alpha]\neq\emptyset$.
\end{claim}
\begin{proof}

Let $g$ be $\PPdia$-generic with $q\in g$ and let 
$$\normalformcert$$
be the resulting semantic certificate. We have $\xi\in K$ and $\lambda_\xi=\lambda$ as $q\in g$. Fix some $\alpha<\xi$. Clearly, 
$$\bar{Z}_\alpha=Z_\alpha\cap Q_\lambda$$
is a $\lambda$-code for a dense subset of $(\dot I^+)^{\dot N_{\omega_1}}$ which is definable over 
$$(Q_\lambda;\in, \PPdia_\lambda, A_\lambda)$$
from a parameter in $X_\xi$, namely $\alpha$. Recall that $\delta^{X_\xi}=\xi$. Using \ref{newgenericitycondition}, we find that there is 
$$R\in (\bar Z_\alpha\cap X_\xi)^{\bigcup g}$$
with $\xi \in R$. Note that there are $r\in g$, $j<\xi=\delta^{X_\xi}$ as well as $k<\omega$ with
\begin{enumerate}[label=\rn*]
\item $(r, j, k)\in \bar{Z}_\alpha\subseteq Z_\alpha$ and
\item $R=\sigma_{j,\omega_1}([k]_j^{\bigcup g})$.
\end{enumerate}
By definition of $Z_\alpha$, and as $r\in g$, $R\in D_\alpha$ and since $\xi\in R$, $R\in g_\xi$, where $g_\xi$ is the generic ultrafilter generating $\sigma_{\xi,\xi+1}\colon N_\xi\rightarrow N_{\xi+1}$.
\end{proof}
(\ref{fstationarypresqeq}) follows from Claim \ref{fstationarypresclaim1} together with Claim \ref{fstationarypresclaim2}. 
\end{proof}

This completes the proof of Theorem \ref{forcingexistencelemm}. We denote the forcing $\PPdia$ constructed above in the instance of a $\Pmaxvariant$ $\Vmax$, the set $A\in\Htwo$ and appropriate dense $D\subseteq\Vmax$ by $\PPdia(\Vmax, A, D)$ (and forget that $\PPdia$ also depends on the choice of $T, T_0,$ etc.).

\subsection{The first blueprint}
We will formulate a general theorem that will allow us to prove a variety of instances of $\MMpp\Rightarrow\Wstar$. In order to formulate the relevant forcing axioms, we use that in practice $\varphi^{\Vmax}$ has a specific form.

\begin{defn}
A $\Pmaxvariant$ $\Vmax$ is \textit{typical}\index{Pmax@$\Pmax$!Variation@-\variant!typical} if $\varphi^{\Vmax}$ can be chosen to be the form 
\begin{align*}
\varphi^{\Vmax}(x)=``&\exists M, I, a_0,\dots, a_n\ x=(M, I, a_0,\dots, a_n)\\ &\wedge\forall y\in M \bigwedge_{\psi\in\Psi} \left[\psi(y)\leftrightarrow (M;\in, I, a_0,\dots, a_n)\models\psi(y)\right]"
\end{align*}
for $n=n^{\Vmax}$ and a finite set $\Psi$ of formulae $\psi(y)$ in the language \hbox{$\{\in,\dot I,\dot a_0,\dots, \dot a_n\}$}. Moreover, $\Psi$ contains the formulae $\psi(x)=``x\in\dot I"$ and $\psi_i(x)=``x=\dot a_i"$ for all $i\leq n^{\Vmax}$. We say that $\Psi$ witnesses the typicality of $\Vmax$. \\
This means that $q<_{\Vmax} p$ iff there is a generic iteration $\mu\colon p\rightarrow p^\ast$ of $p$ in $q$ of length $\omega_1^q+1$ so that the formulae in $\Psi$ are absolute between $q, p^\ast$.
\end{defn}

\begin{rem}
For example, $\Pmax$ is (or can be construed as) a typical $\Pmaxvariant$. We have that typicality of $\Pmax$ is witnessed by $\{\psi_0^{\Pmax},\psi_1^{\Pmax}\}$ where
\begin{itemize}
\item $\psi_0^{\Pmax}(y)=``y\in\dot I"$  and
\item $\psi_1^{\Pmax}(y)=``y=\dot a_0"$.
\end{itemize}
All $\Pmaxvariants$ we will encounter, except for $\Qmaxm$, are typical $\Pmaxvariants$.
\end{rem}

Next, we formulate the relevant bounded and unbounded forcing axioms as general as possible.

\begin{defn}
Suppose $\psi(x)$ is a formula in the language $\{\in,\dot I,\dot a_0,\dots,\dot a_n\}$ and $\vec A=(A_0,\dots, A_n)\in \Htwo$. 
\begin{enumerate}[label=$(\roman*)$]
\item We define $R^\psi_{\vec A}$\index{RpsiA@$R^{\psi}_{\vec A}$} via
$$R^{\psi}_{\vec A}\coloneqq\{x\in\Htwo\mid (\Htwo;\in,\NS, A_0,\dots, A_n)\models\psi(x)\}.$$
\item For $x\in\Htwo$, we say that $C\subseteq\omega_1$\index{Code!for x@for $x$} is a code for $x$ if: Let $l\colon \omega_1\rightarrow\omega_1\times\omega_1$ denote G\"odels pairing function and $E=l[C]$. Then $(\omega_1\times\omega_1, E)$ is wellfounded and $(\tc(\{x\}),\in)$ is the transitive isomorph\footnote{$\tc$ denotes transitive closure.}.
\item $C\subseteq\omega_1$ is a code for an element\index{Code!for an element of RpsiA@for an element of $R^{\psi}_{\vec A}$} of $R^\psi_{\vec A}$ if $C$ is a code for some $x\in R^\psi_{\vec A}$.
\end{enumerate}
\end{defn}

\begin{defn}
Suppose that
\begin{itemize}
\item $\Gamma$ is a class of forcings,
\item $\vec A=(A_0,\dots, A_n)\in\Htwo$ and
\item $\Psi$ is a set of formulae $\psi(x)$ in the language $\{\dot I,\dot a_0,\dots, \dot a_n\}$. 
\end{itemize}
\begin{enumerate}[label=$(\roman*)$]
\item $\DBFAPsi(\Gamma)$\index{BFAPsiGamma@$\BFAPsi(\Gamma)$} states that $D\subseteq\mathbb R$ is $\infty$-universally Baire and whenever $\PP\in\Gamma$ and $g$ is $\PP$-generic then
$$\left(\Htwo;\in, D, R_{\vec A}^{\psi}\mid\psi\in\Psi\right)^V\prec_{\Sigma_1}\left(\Htwo;\in, D^\ast, R_{\vec A}^{\psi}\mid\psi\in\Psi\right)^{V[g]}.$$
For $\Delta\subseteq\mathcal P(\mathbb R)$, $\BFAPsigen{\Delta}_{\vec A}(\Gamma)$ means $\DBFAPsi(\Gamma)$ for all $D\in\Delta$.
\item $\FA^{\Psi}_{\vec A}(\Gamma)$\index{FAPsiA@$\FA^{\Psi}_{\vec A}$} states that whenever $\PP\in\Gamma$ and 
\begin{enumerate}[label=$(\mathrm{FA}.\roman*)$]
\item $\mathcal D$ is a set of at most $\omega_1$-many dense subsets of $\PP$,
\item $\mathcal N_\psi$ is a set of at most $\omega_1$-many $\PP$-names for codes of elements of $(R^{\psi}_{\vec A})^{V^\PP}$ for $\psi\in\Psi$
\end{enumerate}
then there is a filter $g\subseteq\PP$ so that
\begin{enumerate}[label=$(g.\roman*)$]
\item $g\cap D\neq\emptyset$ for all $D\in\mathcal D$ and
\item $\dot S^g=\{\alpha<\omega_1\mid\exists p\in g\ p\forces\check\alpha\in\dot S\}$ is a code for an element of $R^\psi_{\vec A}$ for all $\dot S\in \mathcal N_\psi$, $\psi\in\Psi$.
\end{enumerate}

\end{enumerate}
\end{defn}

We note that the methods of Bagaria in \cite{bagariabfa} readily yield the following.
\begin{lemm}\label{generalbagaria}
Suppose that 
\begin{enumerate}[label=$(\roman*)$]
\item $\Gamma$ is a class of forcings,
\item $\vec A=(A_0,\dots, A_n)\in\Htwo$ and
\item $\Psi$ is a set of formulae $\psi(x)$ in the language $\{\dot I,\dot a_0,\dots, \dot a_n\}$. 
\end{enumerate}
If $\FA^{\Psi}_{\vec A}(\Gamma)$ holds then so does $\BFAPsigen{\mathrm{uB}}_{\vec A}(\Gamma)$.
\end{lemm}

\begin{defn}
Let $\Psi$ be a set of formulae in the language $\{\dot I,\dot a_0,\dots, \dot a_n\}$ for some $n$. For $\vec A=(A_0,\dots, A_n)$, we say that a forcing $\PP$ is $(\Psi,\vec A)$\textit{-preserving}\index{Psi A preserving@$(\Psi, \vec A)$-preserving} iff 
$$R^\psi_{\vec A}=\left(R^\psi_{\vec A}\right)^{V^\PP}\cap V$$
for all $\psi\in\Psi$. $\Gamma^\Psi_{\vec A}$\index{GammaPsiA@$\Gamma^{\Psi}_{\vec A}$} denotes the class of $(\Psi, \vec A)$-preserving forcings.
\end{defn}

\begin{defn}
A $\Pmaxvariant$ $\Vmax$ \textit{accepts} $\diamondsuit$\textit{-iterations}\index{Pmax@$\Pmax$!Variation@-\variant!accepts diamond iterations@accepts $\diamondsuit$-Iterations} if
\begin{align*}
``&\text{If }p\in\Vmax \text{ and }p\rightarrow p^\ast=(M, I, a_0,\dots, a_{n^{\Vmax}})\\&\text{ is a }\diamondsuit\text{-iteration }\text{then }\mathcal H_{(a_0,\dots, a_{n^{\Vmax}})}\models\varphi^{\Vmax}(p^\ast)" 
\end{align*}
is provable in $\mathrm{ZFC}^{-}+``\omega_1\text{ exists}"$ (that is, from sufficiently much of $\ZFC$).
\end{defn}

\begin{blueprintthm}\label{blueprintthm}
Suppose that\index{Blueprint Theorem!First}
\begin{enumerate}[label=$(\roman*)$]
\item $\Vmax$ is a typical $\Pmaxvariant$ with typicality witnessed by $\Psi$,
\item $\Vmax$ has unique iterations and accepts $\diamondsuit$-iterations,
\item $\vec A\in\Htwo$ and $\mathcal H_{\vec A}$ is almost a $\Vmax$-condition,
\item $\SRP$ holds and
\item\label{blueprintfacond} $\FA^{\Psi}_{\vec A}(\Gamma^\Psi_{\vec A})$ holds.
\end{enumerate}
Then $\stargen{\Vmax}$ holds as witnessed by $g_{\vec A}$.
\end{blueprintthm}

\begin{proof}
Let us assume $n^{\Vmax}=0$, so $\vec A=A$. $\SRP$ entails ``$\NS$ is saturated" as well as $\forall\kappa\geq\omega_2\neg\Box_\kappa$. Results of Steel \cite{steelpfa} show that the latter implies that $V$ is closed under \hbox{$X\mapsto M_\omega^\sharp(X)$}. As a consequence
\begin{itemize}
\item $\AD^{\LofR}$,
\item all sets of reals in $\LofR$ are $\infty$-universally Baire and 
\item $(\LofR^V;\in, D)\equiv(\LofR^{V[G]};\in, D^\ast)$ for all sets $D\subseteq\mathbb R$ in $\LofR$ and any generic extension $V[G]$ of $V$.
\end{itemize}
Thus generic projective absoluteness holds in $V$ and if $D\in\LofR$ is a dense subset of $\Vmax$, then $D^\ast$ is a dense subset of $\Vmax$ in any generic extension. Thus $\PPdia(\Vmax, A,D)$ exists for any such $D$.
\begin{claim}
For any dense $D\subseteq\Vmax$, $D\in\LofR$, $\PPdia(\Vmax, A, D)$ is $(\Psi, A)$-preserving.
\end{claim}
\begin{proof}
Let $g$ be $\PPdia(\Vmax, A, D)$-generic. By Theorem \ref{forcingexistencelemm}, in $V[g]$ we have

\begin{tikzpicture}
\def\x{2.1};
\def\y{0.5};

\node (pT) at (0*\x,  0*\y){$D^\ast$};
\node (N) at (0*\x, -2*\y) {$q_0$};
\node (Nast) at (2*\x, -2*\y) {$q_{\omega_1}=(N^*, I^*, b^\ast)$};
\node (M0) at (-2*\x, -4*\y) {$p_0$};
\node (MN) at (0*\x, -4*\y) {$p_{\omega_1^N}$};
\node (M3) at (2*\x, -4*\y) {$p_{\omega_1}$};
\node (H) at (2*\x, -6*\y) {$((H_{\omega_2})^V, \mathrm{NS}_{\omega_1}^V, A)=\mathcal H_{A}$};
\node (Vmax) at (-2*\x, -6*\y) {$\Vmax$};

\path (N)--(pT) node[midway, rotate=90]{$\in$};

\draw[->] (N)--(Nast) node[midway, above]{$\sigma_{0,\omega_1}$};

\path (MN)--(N) node[midway, rotate=90]{$\in$};
\path (M3)--(Nast) node[midway, rotate=90]{$\in$};

\draw[->] (M0)--(MN) node[midway, above]{$\mu_{0, \omega_1^N}$};
\draw[->] (MN)--(M3) node[midway, above]{$\mu_{\omega_1^N, \omega_1}$};

\path (H)--(M3) node[midway, rotate=90]{$=$};

\path (Vmax)--(M0) node[midway, rotate=270]{$\in$};

\end{tikzpicture}

where
\begin{enumerate}[label=$(\PPdia.\roman*)$]
\item $\mu_{0, \omega_1}, \sigma_{0,\omega_1}$ are generic iterations of $p_0$, $q_0$ respectively,
\item $\mu_{0,\omega_1^N}$ witnesses $q_0<_{\Vmax} p_0$,
\item $\mu_{0,\omega_1}=\sigma_{0,\omega_1}(\mu_{0,\omega_1^N})$ and
\item the generic iteration $\sigma_{0, \omega_1}\colon q_0\rightarrow q_{\omega_1}$ is a $\diamondsuit$-iteration.
\end{enumerate}
Note that 
$$(N^\ast;\in, I^\ast, b^\ast)\models\varphi^{\Vmax}(\mathcal H_A).$$
As $\Vmax$ is typical, we must have $b^\ast=A$. As $\Vmax$ accepts $\diamondsuit$-iterations, 
$$(\Htwo;\in,\NS, A)^{V[g]}\models \varphi^{\Vmax}(q_{\omega_1})$$
and finally it follows from typicality that
$$(\Htwo;\in, \NS, A)^{V[g]}\models\varphi^{\Vmax}(\mathcal H_A).$$
As $\Psi$ witnesses the typicality of $\Vmax$, it follows that $\PPdia(\Vmax, A, D)$ is $(\Psi, A)$-preserving.
\end{proof}

It follows from Theorem \ref{forcingexistencelemm}, Lemma \ref{generalbagaria} and Lemma \ref{meetDlemm} that
\begin{itemize}
\item $g_{\vec A}\cap D\neq\emptyset$ for all dense $D\subseteq\Vmax$, $D\in\LofR$ and 
\item $\Pomo=\bigcup\{\Pomo\cap p^\ast\mid p\in g_{\vec A}
\wedge\mu\colon p\rightarrow p^\ast\text{ is guided by }g_{\vec A}\}$.
\end{itemize}
By Corollary \ref{getVmaxstarcor}, $g_{\vec A}$ witnesses $\stargen{\Vmax}$.
\end{proof}

\begin{rem}
If additionally there are a proper class of Woodin cardinals, then $g_{\vec A}$ meets all $\infty$-universally Baire dense subsets of $\Vmax$.
\end{rem}

\subsection{The second blueprint}

From the right perspective, $\stargen{\Vmax}$ is a forcing axiom. As noted before, \aspsch\ show that if there is a proper class of Woodin cardinals, then $\Wstar$ is equivalent to $(\mathcal P(\mathbb R)\cap \LofR)$-$\BMMpp$. Some additional assumption like large cardinals is necessary as $\BMM$ implies closure of $V$ under sharps while $\Wstar$ holds in the $\Pmax$-extension of $\LofR$. We try to generalize this result roughly to all natural $\Pmaxvariants$ for which the $\PPdia$-method can prove them from some forcing axiom. We will have to restrict to better behaved $\Pmaxvariants$.
\begin{defn}
Let $\Vmax$ be a $\Pmaxvariant$ with unique iterations and $g$ be $\Vmax$-generic over $\LofR$.
\begin{enumerate}[label=$(\roman*)$]
 \item We say that $g$ \textit{produces} $(A_0,\dots, A_{n^{\Vmax}})$\index{g produces A@$g$ produces $(A_0,\dots, A_{n})$} if there is $p\in g$ so that if 
 $$\mu\colon p\rightarrow p^\ast=(M, I, a_0,\dots, a_{n^{\Vmax}})$$
 is the $g$-iteration of $p$ then $a_i=A_i$ for all $i\leq n^{\Vmax}$.
 \item If $\Vmax$ is typical, we set
 $$\mathcal H_g\coloneqq (\Htwo, \NS, A_0,\dots, A_{n^{\Vmax}})^{\LofR[g]}$$
 \index{Hg@$\mathcal H_g$}where $(A_0,\dots, A_{n^{\Vmax}})$ is the unique sequence produced by $g$.
\end{enumerate}
\end{defn}

\begin{defn}
A $\Pmaxvariant$ $\Vmax$ with unique iterations is \textit{self-assembling}\index{Pmax@$\Pmax$!Variation@-\variant!self-assembling} if: Whenever $g$ is $\Vmax$-generic over $\LofR$ then
\begin{enumerate}[label=$(\roman*)$]
\item $\mathcal H_g$ is almost a $\Vmax$-condition and
\item $(\Htwo)^{\LofR[g]}=\bigcup\{p^\ast\mid p\in g, \mu\colon p\rightarrow p^\ast\text{ guided by }g\}$.
\end{enumerate}
\end{defn}
All $\Pmaxvariant$ we will work with are self-assembling (assuming $\AD$ in $\LofR$). For example, $\Pmax$ is self-assembling. The relevance of this property for us is partly explained by the following result.

\begin{lemm}\label{selfassemblingfilterlemm}
Suppose $\Vmax$ is a self-assembling $\Pmaxvariant$ with unique iterations and typicality of $\Vmax$ is witnessed by a set $\Psi$ of $(\Sigma_1\cup\Pi_1)$-formulae. If $\stargen{\Vmax}$ holds as witnessed by $g$ then
\begin{enumerate}[label=$(\roman*)$]
    \item\label{selfassemblecond1} $\mathcal H_{\vec A}$ is almost a $\Vmax$-condition and
    \item\label{selfassemblecond2} $g=g_{\vec A}$
\end{enumerate}
where $g$ produces $\vec A$.
\end{lemm}
\begin{proof}
As $\Vmax$ is self-assembling, $\mathcal H_g$ is almost a $\Vmax$-condition. Moreover, $\Pomo\subseteq\LofR[g]$ as $g$ witnesses $\stargen{\Vmax}$. It follows that $\mathcal H_g=\mathcal H_{\vec A}$ and thus \ref{selfassemblecond1} holds.\\
Let us now prove \ref{selfassemblecond2}, note that it suffices to show $g\subseteq g_{\vec A}$.
\begin{claim}\label{matchupclaim}
If $q\in g$ and 
$$\sigma\colon q\rightarrow q^\ast=(M^\ast, I^\ast, a_0^\ast,\dots, a_{n^{\Vmax}}^\ast)$$
is the $g$-iteration of $q$ then $I^\ast=\NS\cap M^\ast$ and $a_i^\ast=A_i$ for $i\leq n^{\Vmax}$.
\end{claim}
\begin{proof}
$a_i^\ast=A_i$ for $i\leq n^{\Vmax}$ follows easily from typicality, we show $I^\ast=\NS\cap M^\ast$. It is clear that $I^\ast\subseteq\NS$ since if $S\in I^\ast$, then a tail of the iteration points of the iteration $\sigma\colon q\rightarrow q^\ast$ is missing from $S$. On the other hand, suppose $S\in \Pomo^{M^\ast}-I^\ast$. We may assume $S=\mu(\bar S)$ for some $\bar S\in q$. If $C\subseteq\omega_1$ is club then as $\Vmax$ is self-assembling, there is $r\in g$, such that if $\nu\colon r\rightarrow r^\ast$ is the $g$-iteration of $r$, then $C\in\ran(\nu)$, say $C=\nu(\bar C)$. Note that we may assume $r<_{\Vmax} q$, say this is witnessed by
$$\bar \sigma\colon q\rightarrow \bar q=(\bar M,\bar I,\bar a).$$
Write $r=(N, J, b)$. As $\Vmax$ is typical, $\bar I=J\cap \bar M$ and hence $\bar\sigma(\bar S)\cap \bar C\neq\emptyset$ which gives
$$\nu\circ\bar\sigma(\bar S)\cap C\neq\emptyset.$$
Clearly, $\nu(\bar\sigma)$ is an iteration of $q$ of length $\omega_1+1$ guided by $g$. Thus, by Lemma \ref{guidediterationslemm}, $\nu(\bar\sigma)=\sigma$. $S\cap C\neq\emptyset$ follows.
\end{proof}
Let $p\in g$ and let $\mu\colon p\rightarrow p^\ast$ be the $g$-iteration of $p$.
\begin{claim}
$\mathcal H_{\vec A}\models\varphi^{\Vmax}(p^\ast)$.
\end{claim}
\begin{proof}
Let $\psi\in\Psi$ and assume $\psi$ is $\Sigma_1$, so write $\psi(x)=\exists y\ \theta(x, y)$ where $\theta$ is $\Sigma_0$. So suppose for some $x\in p$ and $y\in \Htwo$ we have
$$\mathcal H_{\vec A}\models\exists y\ \theta(x, y).$$
As $\Vmax$ is self-assembling, we can find $q\in g$ with
\begin{enumerate}[label=$(q.\roman*)$]
    \item $q<_{\Vmax} p$ as witnessed by $\bar\mu\colon p\rightarrow \bar p$ and
    \item $\mathcal H_{\vec A}\models\theta(x, \sigma(y))$ for some $y\in q$ 
\end{enumerate}
where $\sigma\colon q\rightarrow q^\ast$ is the $g$-iteration of $q$. By Claim \ref{matchupclaim},
$$q^\ast\prec_{\Sigma_0}\mathcal H_{\vec A}$$
and as $\sigma(\bar\mu)=\mu$ by Lemma \ref{guidediterationslemm} as well as elementarity of $\sigma$ we find 
$$q\models\theta(\bar\mu(x), y).$$
Finally, $q\models(\varphi^{\Vmax}(\bar p))$ so that 
$$\bar p\models\exists z\ \theta(\bar\mu(x), z)$$
and hence $p\models\exists z\ \theta(x, z)$ by elemntarity of $\bar\mu$.\\
The ``dual argument" works if $\psi$ is $\Pi_1$ instead.
\end{proof}
Now if $G$ is $\Col(\omega,2^{\omega_1})$-generic then the above shows that $\mu\colon p\rightarrow p^\ast$ witnesses $\mathcal H_{\vec A}<_{\Vmax} p$ in $V[G]$. Thus $p\in g_{\vec A}$.
\end{proof}

Theorem \ref{blueprintthm} gives a hint how the forcing axiom equivalent to $\stargen{\Vmax}$ should look like. However, $\Gamma^\Psi_{\vec A}$ is not the right class of forcings, for example one can construe two $\Pmaxvariants$ which are the same as forcings, but for which the resulting classes $\Gamma^{\Psi}_{\vec A}$ are fundamentally different for reasonable $\vec A$. Instead, we should look at the class of forcings which roughly lie on the way to the good extensions highlighted in the $\Vmax$-Multiverse View.

\begin{defn}
Suppose that 
\begin{enumerate}[label=$(\roman*)$]
\item $\Vmax$ is a typical $\Pmaxvariant$,
\item typicality of $\Vmax$ is witnessed by $\Psi$ and
\item $\vec A=(A_0,\dots, A_{n^{\Vmax}})\in \Htwo$.
\end{enumerate}
The class $\Gamma^{\Vmax}_{\vec A}(\Psi)$\index{GammaVmaxAPsi@$\Gamma^{\Vmax}_{\vec A}(\Psi)$} consists of all $(\Psi, \vec A)$-preserving forcings $\PP$ so that if $g$ is $\PP$-generic, then there is a forcing $\QQ\in V[g]$ with 
$$V[g]\models``\QQ\text{ is }(\Psi,\vec A)\text{-preserving}"$$
and if further $h$ is $\QQ$-generic over $V[g]$, then in $V[g][h]$ both
\begin{enumerate}[label=$(h.\roman*)$]
\item $\mathcal H_{\vec A}$ is almost a $\Vmax$-condition and
\item $\NS$ is saturated.
\end{enumerate}
\end{defn}

It just so happens that, maybe by accident, for the $\Pmaxvariants$ we will look at explicitly, if there is a proper class of Woodin cardinals then one can choose $\Psi$ so that $\Gamma^{\Psi}_{\vec A}=\Gamma^{\Vmax}_{\vec A}(\Psi)$ in case that $\Gamma^{\Vmax}_{\vec A}\neq\emptyset$.

\begin{defn}
Suppose that $(M, I)$ is a potentially iterable structure and $Y\subseteq\mathbb R$. We say that $(M, I)$ is (generically) $Y$\textit{-iterable}\index{Potentially iterable structure!Y iterable@$Y$-iterable} if for $X\coloneqq Y\cap M$ we have
\begin{enumerate}[label=$(\roman*)$]
    \item $(M;\in, I, X)$ is a model of (sufficiently much of) $\ZFC$ where $Y$ is allowed as a class parameter in the schemes and
    \item whenever $\langle (M_\alpha, I_\alpha, X_\alpha),\mu_{\alpha,\beta}\mid\alpha\leq\beta\leq\gamma\rangle$ is a generic iteration of $(M_0, I_0, X_0)=(M, I, X)$, i.e.
    \begin{enumerate}[label=$(\mu.\roman*)$]
        \item $(M_{\alpha+1};\in , I_{\alpha+1}, X_{\alpha+1})$ is an ultrapower of $(M_\alpha;\in , I_\alpha, X_\alpha)$ by a $M_\alpha$-generic ultrafilter w.r.t. $I_\alpha$ for $\alpha<\gamma$,
        \item if $\alpha\leq\gamma$ is a limit then 
        $$\langle(M_\alpha, I_\alpha, X_\alpha),\mu_{\xi,\alpha}\mid\xi<\alpha\rangle=\varinjlim(\langle(M_\beta,I_\beta, X_\beta),\mu_{\beta,\xi}\mid\beta\leq\xi<\alpha\rangle)$$
    \end{enumerate}
    then $X_\gamma=Y\cap M_\gamma$.
\end{enumerate}
\end{defn}

\begin{prop}[Folklore]\label{Xiterableprop}
Suppose that $\NS$ is saturated and $X\subseteq\mathbb R$ is $\infty$-universally Baire.
Then in any forcing extension $V[G]$ in which $\Htwo^V$ is countable, $(\Htwo, \NS, X)^V$ is $X^\ast$-iterable. 
\end{prop}
\begin{proof}

Let $\PP$ be some forcing which collapses $2^{\omega_1}$ to $\omega$. Let $T, S\in V$ witness that $X$ is $\vert\PP\vert$-universally Baire with $p[T]=X, p[S]=\mathbb R-X$. Let $G$ be $\PP$-generic over $V$. Let 
$$\langle (M_\alpha, I_\alpha, X_\alpha),\mu_{\alpha,\beta}\mid\alpha\leq\beta\leq\gamma\rangle$$
be any generic iteration of $(M_0, I_0, X_0)=(H_\kappa, \NS, X)^V$. Then as in Claim \ref{extendclaim}, this iteration can be lifted to a generic iteration 
$$\langle(M_\alpha^+, I_\alpha, X_\alpha),\mu_{\alpha,\beta}^+\mid\alpha\leq\beta\leq\gamma\rangle$$
of $(M_0^+, I_0, X_0)=(V, \NS^V, X)$. In particular, $M_\gamma$ is wellfounded as $M^+_\gamma$ is wellfounded. Let $\mu^+=\mu^+_{0,\gamma}$, $M^+=M^+_\gamma$.
\begin{claim}
In $V[G]$, $p[\mu^+(T)]=X^\ast$.
\end{claim}
\begin{proof}
Work in $V[G]$. We have $X^\ast=p[T]$ and this implies $X^\ast\subseteq p[\mu^+(T)]$, likewise $\mathbb R- X^\ast\subseteq p[\mu^+(S)]$. In $M^+$, $\mu^+(T), \mu^+(S)$ project to complements and an absoluteness of wellfoundedness argument shows that this must be true in $V[G]$ as well, so that we indeed have $X^\ast=p[\mu^+(T)]$.
\end{proof}
We conclude
$$X_\gamma=\mu^+(X)=\mu^+(p[T])=p[\mu^+(T)]\cap M^+=X^\ast\cap M^+=X^\ast\cap M_\gamma$$
which is what we had to show.
\end{proof}

\begin{lemm}\label{fafromstarlemm}
Suppose that 
\begin{enumerate}[label=$(\roman*)$]
\item $\Vmax$ is a typical self-assembling $\Pmaxvariant$ with unique iterations,
\item typicality of $\Vmax$ is witnessed by a set of $(\Sigma_1\cup\Pi_1)$-formulae $\Psi$,
\item there is a proper class of Woodin cardinals,
\item $\stargen{\Vmax}$ holds as witnessed by $g$ and
\item $g$ produces $\vec A$.
\end{enumerate}
Then $\BFAPsigen{(\mathcal P(\mathbb R)\cap\LofR)}_{\vec A}(\Gamma^{\Vmax}_{\vec A}(\Psi))$ holds true.
\end{lemm}

\begin{proof}
We will assume $n^{\Vmax}=0$. Let $g$ witness $\stargen{\Vmax}$. Let $p\in g$ and $\mu\colon p\rightarrow p^\ast=(M, I, A)$ the generic iteration of $p$ guided by $g$. We will show that 
$$\BFAPsigen{(\mathcal P(\mathbb R)\cap\LofR)}_{A}(\Gamma^{\Vmax}_{A}(\Psi))$$
holds. By Lemma \ref{selfassemblingfilterlemm}, $\mathcal H_g=\mathcal H_{A}$ is almost a $\Vmax$-condition. Now let $\PP\in\Gamma^{\Vmax}_A(\Psi)$ and $X\in\mathcal P(\mathbb R)\cap\LofR$. Let $G$ be $\PP$-generic. We have to show that 
$$(\Htwo;\in, X, R^\psi_A\mid \psi\in\Psi)^V\prec_{\Sigma_1}(\Htwo;\in,X^\ast, R^\psi_A\mid\psi\in\Psi)^{V[G]}.$$
So let $v\in\Htwo^V$, and $\theta$ a $\Sigma_0$-formula such that
$$(\Htwo;\in,X^\ast, R^\psi_A\mid\psi\in\Psi)^{V[G]}\models \exists u\ \theta(u, v).$$
As $\Vmax$ is self-assembling, we may assume without loss of generality that $v=\mu(\bar v)$ for some $\bar v\in p$. Let $V[G][H]$ be a further generic extension by $(\Psi, A)$-preserving forcing so that in $V[G][H]$
\begin{enumerate}[label=$(H.\roman*)$]
\item $\mathcal H_A^{V[G][H]}$ is almost a $\Vmax$-condition and
\item $\NS$ is saturated.
\end{enumerate}
Note that 
$$(H_{\omega_2};\in, X^\ast, R^\psi_A\mid \psi\in\Psi)^{V[G]}\prec_{\Sigma_0}(\Htwo;\in, X^{\ast\ast}, R^\psi_A\mid \psi\in\Psi)^{V[G][H]}$$
as the extension is $(\Psi, A)$-preserving. Here, $X^{\ast\ast}$ denotes the reevaluation of $X^\ast$ in $V[G][H]$. Accordingly,
$$(\Htwo;\in, X^{\ast\ast}, R^\psi_A\mid \psi\in\Psi)^{V[G][H]}\models\exists u\ \theta(u, v).$$
Let $g$ be $\Col(\omega,2^{\omega_1})^{V[G][H]}$-generic over $V[G][H]$ and $X^{\ast\ast\ast}$ the reevaluation of $X^{\ast\ast}$ in $V[G][H][g]$. Then in $V[G][H][g]$, 
$$(\Htwo, \NS, X^{\ast\ast})^{V[G][H]}$$
is $X^{\ast\ast\ast}$-iterable by Proposition \ref{Xiterableprop}.
\begin{claim}\label{strongerclaim}
$\mathcal H_A^{V[G][H]}<_{\Vmax} q$ for all $q\in g$.
\end{claim}
\begin{proof}
Let $q\in g$ and $\sigma\colon q\rightarrow q^\ast$ the $g$-iteration of $q$. It follows from the proof of Lemma \ref{selfassemblingfilterlemm} that 
$$(\Htwo;\in,\NS, A)^V\models\varphi^{\Vmax}(q^\ast)$$
and since the extension $V\subseteq V[G][H]$ is $(\Psi, A)$-preserving, 
$$(\Htwo;\in,\NS, A)^{V[G][H]}\models\varphi^{\Vmax}(q^\ast)$$
follows.
\end{proof}

Let $q\in g$, $q<_{\Vmax} p$ as witnessed by $\bar \mu\colon p\rightarrow \bar p$. $\mathcal H^{V[G][H]}_A$ witnesses in $V[G][H][g]$ that there is $r=(M, I, a)<_{\Vmax} q$, as witnessed by $\sigma\colon q\rightarrow q^\ast$, so that
\begin{enumerate}[label=$(r.\roman*)$]
\item\label{fafromstarfirstrcond} $(M, I, Y)$ is $X^{\ast\ast\ast}$-iterable,
\item $(M;\in, I)\models``V=\Htwo\wedge I=\NS"$ and
\item\label{fafromstarlastrcond} $(M;\in, Y, R^{\psi}_A\mid\psi\in\Psi)^M\models\exists u\ \theta(u, \sigma(\bar\mu(\bar v))$
\end{enumerate}
where $Y=X^{\ast\ast\ast}\cap M$. As there is a proper class of Woodin cardinals,
$$(\LofR^V;\in, X)\equiv(\LofR^{V[G][H][g]};\in, X^{\ast\ast\ast})$$
and hence a density argument shows that there is $q=(N, J, b)\in g$, $q<_{\Vmax} p$, as witnessed by $\mu'\colon p\rightarrow p'$, such that 
\begin{enumerate}[label=$(q.\roman*)$]
\item\label{fafromstarxitcond} $(N, J, X\cap N)$ is $X$-iterable,
\item $(N;\in, J)\models``V=\Htwo\wedge J=\NS"$ and
\item\label{fafromstarexistitcond} for some $u\in N$, $(N;\in, X\cap N, R^{\psi}_A\mid\psi\in\Psi)^N\models \theta(u,\mu'(v))$.
\end{enumerate}
Let $\sigma\colon q\rightarrow q^\ast=(N^\ast, J^\ast, a^\ast)$ be the $g$-iteration of $q$.  By (the proof of) Lemma \ref{selfassemblingfilterlemm} \ref{selfassemblecond2}
$$(\Htwo,\NS, A)^V\models \varphi^{\Vmax}(q^\ast)$$
and hence
$$(N^\ast;\in, X\cap N^\ast, R^{\psi}_A \mid\psi\in\Psi)^{N^\ast}\prec_{\Sigma_0}(\Htwo;\in,X, R^{\psi}_A\mid\psi\in\Psi)^V.$$
Moreover, 
$$\sigma\colon (N, J, X\cap N)\rightarrow (N^\ast, J^\ast, X\cap N^\ast)$$
is fully elementary by \ref{fafromstarxitcond} so that
$$(N^\ast;\in, X\cap N^\ast, R^{\psi}_A\mid\psi\in\Psi)^{N^\ast}\models\theta(\sigma(u), \sigma(\mu'(v))).$$
By Lemma \ref{guidediterationslemm}, $\sigma\circ\mu'=\mu$, so we can conclude
$$(\Htwo;\in, X, R^{\psi}_A\mid\psi\in\Psi)^V\models\theta(\sigma(u), v)$$
which is what we had to show.
\end{proof}
 
In fact, we get an equivalence in case we can apply the $\PPdia$-method.

\begin{blueprinttwothm}\label{secondblueprintthm}
\index{Blueprint Theorem!Second}Suppose that
\begin{enumerate}[label=$(\roman*)$]
\item There are a proper class of Woodin cardinals,
\item $\Vmax$ is a self-assembling typical $\Pmaxvariant$,
\item $\Vmax$ has unique iterations and accepts $\diamondsuit$-iterations,
\item typicality of $\Vmax$ is witnessed by a set $\Psi$ of $(\Sigma_1\cup\Pi_1)$-formulae,
\item $\vec A=(A_0,\dots, A_{n^{\Vmax}})\in \Htwo$ and
\item $\Gamma^{\Psi}_{\vec A}=\Gamma^{\Vmax}_{\vec A}(\Psi)$.
\end{enumerate}
The following are equivalent:
\begin{enumerate}[label=$(\bigast.\roman*)$]
\item\label{starfaeqstarcond} There is a filter $g\subseteq\Vmax$ which witnesses $\stargen{\Vmax}$ and produces $\vec A$.
\item\label{starfaeqfacond}$\BFAPsigen{(\mathcal P(\mathbb R)\cap\LofR)}_{\vec A}(\Gamma^{\Vmax}_{\vec A}(\Psi))$.
\end{enumerate}
\end{blueprinttwothm}

\begin{proof}
``\ref{starfaeqstarcond}$\Rightarrow$\ref{starfaeqfacond}" follows from Theorem \ref{fafromstarlemm}. ``\ref{starfaeqfacond}$\Rightarrow$\ref{starfaeqstarcond}" can be proven similar to the First Blueprint Theorem \ref{blueprintthm}. We use the existence of a proper class of Woodin cardinals instead of $\SRP$ to justify $\AD^{\LofR}$, that all sets of reals in $\LofR$ are $\infty$-universally Baire and generic $\LofR$-absoluteness. It is not immediate that $\mathcal H_{\vec A}$ is almost a $\Vmax$-condition, nor did we assume that $\NS$ is saturated, however as $\Gamma^{\Vmax}_{\vec A}(\Psi)=\Gamma^{\Psi}_{\vec A}$, we can pass to a $(\Psi, \vec A)$-preserving forcing extension in which both of this is true. It follows that 
\begin{align*}
    g=\{p\in\Vmax\mid&\exists \mu\colon p\rightarrow p^\ast\text{ a generic iteration of } 
    \\&\text{length }\omega_1+1\text{ with }\mathcal H_{\vec A}\models\varphi^{\Vmax}(p^\ast)\}
\end{align*}
witnesses $\stargen{\Vmax}$ and produces $\vec A$.
\end{proof}

\subsection{The $\Qmax$-variation $\Qmaxm$}

We will have to do some work in order to find a forcing which freezes $\NS$ along a witness $f$ of $\diacol$. The main idea is to find the correct $\Pmaxvariant$ to throw into the $\diamondsuit$-$\Wstar$-forcing. Let us first introduce Woodin's $\Qmax$.

\begin{defn}
A condition $p\in\Qmax$\index{Qmax@$\Qmax$} is a generically iterable structure $p=(N, I, f)$ with
\begin{enumerate}[label=$(\Qmax.\roman*)$]
\item $N\models``f\text{ guesses }\Col(\omega,\omega_1)\text{-filters}"$ and
\item $N\models``\eta_f:\Col(\omega,\omega_1)\rightarrow(\Pomo/I)^+\text{ is a dense embedding}"$, where $\eta_f$ is the embedding associated to $f$.
\end{enumerate}
The order on $\Qmax$ is given by 
$$q=(M, J, h)<_{\Qmax}p$$
iff there is an iteration 
$$j\colon p\rightarrow p^\ast=(N^\ast, I^\ast, f^\ast)$$
in $q$ with $f^\ast=h$.
\end{defn}
We mention that it follows from Lemma \ref{diacomplemm} that if $(N, I, f)$ is a $\Qmax$-condition then $N\models``f\text{ witnesses }\diacolp$".\\
Forcing that $\mathcal H_f$ is almost a $\Qmax$-condition for some $f$ essentially amounts to forcing ``$\NS$ is $\omega_1$-dense". We replace $\Qmax$ by an equivalent forcing for which this is easier to achieve.

\begin{defn}
A condition $p\in\Qmaxm$\index{Qmaxminus@$\Qmaxm$} is a generically iterable structure of the form $p=(N, I, f)$ so that 
$$(N;\in, I)\models``f\text{ witnesses }\diacolip".$$
The order on $\Qmaxm$ is given by $q\coloneqq (M, J, h)<_{\Qmaxm}(N, I, f)=:p$ iff  there is an iteration 
$$j\colon p\rightarrow p^\ast=(N^\ast, I^\ast, f^\ast)$$
in $q$ so that
\begin{enumerate}[label=$(<_{\Qmaxm}.\roman*)$]
 \item $f^\ast = h$ and
 \item\label{Qmaxmorderkillcond} if $S\in J^+\cap p^\ast$ then there is $b\in\Col(\omega,\omega_1^q)$ with $S^h_b\subseteq S\mod J$.
 \end{enumerate} 
\end{defn}

We note that $\Qmaxm$ is essentially unchanged if condition \ref{Qmaxmorderkillcond} is dropped, but demanding it is convenient for us.

\begin{prop}[Woodin, {\cite[Definition 6.20]{woodinbook}}]\label{strongdiasharpsprop}
Suppose $\Pomo$ is closed under $A\mapsto A^\sharp$ and $I$ is a normal uniform ideal. Suppose $f$ guesses $\Col(\omega,\omega_1)$-filters. The following are equivalent:
\begin{enumerate}[label=\rn*]
\item\label{strongdiasharpcond1} $f$ witnesses $\diacolip$.
\item\label{strongdiasharpcond2} For any $A\subseteq\omega_1$, 
$$\{\alpha<\omega_1\mid f(\alpha)\text{ is not generic over }L[A\cap\alpha]\}\in I$$
and for all $b\in\BP$, $S^f_b\in I^+$.
\end{enumerate}
\end{prop}

The following is the key result about $\Qmaxm$.

\begin{lemm}\label{Qmaxmanalysislemm}
Suppose $J$ is a normal uniform ideal, $h$ witnesses $\diacolpgen{J}$, and $\Pomo$ is closed under $A\mapsto A^\sharp$. For any $p=(N, I, f)\in\Qmaxm$ there is an iteration 
$$j\colon p\rightarrow p^\ast=(N^\ast, I^\ast, f^\ast)$$
so that
\begin{enumerate}[label=\rn*]
\item $f^\ast=h\mod J$ (so in particular $f^\ast$ witnesses $\diacolpgen{J}$) and
\item if $S\in J^+\cap N^\ast$ then there is $b\in\Col(\omega,\omega_1)$ with $S^{f^\ast}_b\subseteq S\mod J$.
\end{enumerate}
\end{lemm}

\begin{proof}
Let $x$ be a real coding $p$ and let $D$ be the club of $x$-indiscernibles below $\omega_1$. By induction along $\omega_1$ we will define a filter $g\subseteq \Col(\omega, {<}\omega_1)$. Let 
$$\vec\alpha\coloneqq \langle \alpha_i\mid i<\omega_1\rangle$$
be the increasing enumeration of $D$. Assume that $g\res\alpha_i$ is already defined. First we define $g(\alpha_i)$:\\
\ul{Case 1:} $h(\alpha_i)$ is generic over $L[x, g\res\alpha_i]$. Then let $g(\alpha_i)=h(\alpha_i)$.\\
\ul{Case 2:} Case 1 fails. Then let $g(\alpha_i)$ be some generic for $\Col(\omega, \alpha_i)$ over $L[x, g\res\alpha_i]$.\\
Next, we choose $g\res(\alpha_i, \alpha_{i+1})$ to be any generic for $\Col(\omega, (\alpha_i, \alpha_{i+1}))$ over $L[x, g\res\alpha_i + 1]$. 
\begin{claim}
$g$ is generic over $L[x]$.
\end{claim}
\begin{proof}
$\vec\alpha$ enumerates a club of $L[x]$-regular ordinals. Thus for any $i<\omega_1$, $\Col(\omega,{<}\alpha_i)$ has the $\alpha_i$-c.c.~in $L[x]$. It follows by induction that $g\res\alpha_i$ is $\Col(\omega, {<}\alpha_i)$-generic over $L[x]$ and finally that $g$ is $\Col(\omega,{<}\omega_1)$-generic over $L[x]$. 
\end{proof}
By induction on $\alpha<\omega_1$, we now define a generic iteration 
$$\langle p_i, \sigma_{i, j}, U_i\mid i\leq j\leq\alpha\rangle$$
of $p_0=p$. Here, $U_i$ denotes the generic filter that produces the ultrapower $\sigma_{i, i+1}$.\\ 
Let $\eta_\alpha$ denote the map 
$$(\eta_{\sigma_{0,\alpha}(f)})^{p_\alpha}\colon\Col(\omega,\omega_1^{p_\alpha})\rightarrow ((\Pomo/\sigma_{0,\alpha}(I))^+)^{p_\alpha}.$$
Simply pick $U_\alpha$ least, according to the canonical global wellorder in $$L[x, g\res\omega_1^{p_\alpha}+1]$$ 
so that 
\begin{enumerate}[label=$(U.\roman*)$]
\item $U_\alpha$ is $((\Pomo/\sigma_{0,\alpha}(I))^+))^{p_\alpha}$-generic over $p_\alpha$ and
\item $\eta_\alpha[g(\omega_1^{p_\alpha})]\subseteq U_\alpha$.
\end{enumerate}
This is possible as $g(\omega_1^{p_\alpha})$ is $\Col(\omega, \omega_1^{p_\alpha})$-generic over $p_\alpha$, as
$$p_{\alpha}\models``\eta^{p_\alpha}\text{ is a regular embedding}"$$
and as $p_\alpha$ is countable in $L[x, g\res\omega_1^{p_\alpha}+1]$. $U_\alpha$ induces the generic ultrapower $\sigma_{\alpha,\alpha+1}\colon p_\alpha\rightarrow \Ult(p_\alpha, U_\alpha)=:p_{\alpha+1}$.\\

Finally we get a generic iteration map
$$\sigma\coloneqq \sigma_{0,\omega_1}\colon p\rightarrow p^\ast\coloneqq p_{\omega_1}=(N^\ast, I^\ast, f^\ast).$$

\begin{claim}\label{almostequalclaim}
$f^{\ast}=h\mod J$.
\end{claim}

\begin{proof}
$f^{\ast}$ and $g$ agree on the club of iteration points, i.e.~we have $f^{\ast}(\omega_1^{p_\alpha})=g(\omega_1^{p_\alpha})$ for any $\alpha<\omega_1$. Here we use that $U_\alpha$ extends $\pi^{p_\alpha}[g(\alpha)]$.\\
Moreover,
$$\{\alpha<\omega_1\mid h(\alpha)\text{ is not generic over }L[x, g\res\alpha]\}\in J$$
by Proposition \ref{strongdiasharpsprop} as $h$ witnesses $\diacolpgen{J}$. By construction of $g$, it follows that $\{\alpha<\omega_1\mid h(\alpha)\neq g(\alpha)\}\in J$. As $J$ is a normal uniform ideal, we can conclude
$$\{\alpha<\omega_1\mid f^\ast(\alpha)\neq h(\alpha)\}\in J.$$
\end{proof}

It follows that $f^\ast$ witnesses $\diacolpgen{J}$. Now let $S\in J^+\cap N^\ast$. We have to show the following.
\begin{claim}
$S^{f^\ast}_b\subseteq S\mod J$ for some $b\in\Col(\omega, \omega_1)$.
\end{claim}

\begin{proof}
We will prove that the intersection of $D$ with $S^{f^\ast}_b-S$ is bounded below $\omega_1$ for some $b$. Find $\alpha\in D$ so that
\begin{enumerate}[label=$(\alpha.\roman*)$]
 \item there is $\bar S\in p_\alpha$ with $\sigma_{\alpha,\omega_1}(\bar S)=S$ and
 \item\label{qmaxmalphacond2} $\alpha\in S$.
\end{enumerate}
By \ref{qmaxmalphacond2}, there must be some $b\in g(\alpha)$ with
$$b\forces^{L[x, g\res\alpha]}_{\Col(\omega, \alpha)}\bar S\in \dot U_\alpha$$
where $\dot U_\alpha$ is a name for the least filter $U$ that is generic over $p_\alpha$ and contains $\eta_\alpha[\dot g]$, where $\dot g$ is now the canonical name for the generic. Now suppose $\alpha<\beta\in S^{f^\ast}_b\cap D$. There is then an elementary embedding 
$$j\colon L[x]\rightarrow L[x]$$
with 
\begin{enumerate}[label=$(j.\roman*)$]
\item $j(\alpha)=\beta$ and
\item $\mathrm{crit}(j)=\alpha$.
\end{enumerate}
We have that $j$ lifts to an elementary embedding 
$$j^+\colon L[x, g\res\alpha]\rightarrow L[x, g\res\beta]$$
so that 
$$b=j(b)\forces^{L[x, g\res\beta]}_{\Col(\omega,\beta)}j^+\left(\bar S\right)\in j^+\left(\dot U_\alpha\right).$$
Clearly, $j^+\left(\dot U_\alpha\right)^{g(\beta)}=U_\beta$ and thus 
$$\beta\in \sigma_{\beta,\omega_1}\left(j^+\left(\bar S\right)\right)$$
as $b\in f^\ast(\beta)=g(\beta)$. Note that all points in $D$ are iteration points and recall that $f^\ast$ and $g$ agree on iteration points.
\begin{sub}
$j^+\left(\bar S\right)=\sigma_{\alpha,\beta}\left(\bar S\right)$.
\end{sub}
\begin{proof}
The reason is that, since $\alpha$ is a limit ordinal, $p_\alpha$ is the direct limit along $\langle p_i, \sigma_{i, k}\mid i\leq k<\alpha\rangle$ and thus there is some $\gamma<\alpha$ and $\bar{\bar S}\in p_\gamma$ with $\sigma_{\gamma\alpha}\left(\bar{\bar S}\right)=\bar S$. Hence
\begin{align*}
j^+\left(\bar S\right)&=j^+\left(\sigma_{\gamma, \alpha}\left(\bar{\bar S}\right)\right)=j^+(\sigma_{\gamma, \alpha})\left(j^+\left(\bar{\bar S}\right)\right)\\
&=\sigma_{\gamma, \beta}\left(\bar{\bar{S}}\right)=\sigma_{\alpha,\beta}\left(\sigma_{\gamma, \alpha}\left(\bar{\bar S}\right)\right)=\sigma_{\alpha, \beta}\left(\bar S\right).
\end{align*}
Here, we use $j^+(\sigma_{\gamma,\alpha})=\sigma_{\gamma,\beta}$ in the third equation. This holds as our lift $j^+$ satisfies $j^+(g\res\alpha)=g\res\beta$ and so it is easy to see that $j^+(\langle U_i\mid i<\alpha\rangle)=\langle U_i\mid i<\beta\rangle$ so that 
$$j^+(\langle p_i, \sigma_{i, k}\mid i\leq k <\alpha\rangle)=\langle p_i, \sigma_{i, k}\mid i\leq k <\beta\rangle.$$
\end{proof}
All in all, $\beta\in \sigma_{\beta,\omega_1}\left(\sigma_{\alpha,\beta}\left(\bar S\right)\right)=S$.  Thus 
$$\left(S_b^{f^\ast}-S\right)\cap D\subseteq\alpha$$
so that $S_b^{f^\ast}\subseteq S\mod J$.\\
\end{proof}
\end{proof}

\begin{prop}[Folklore?]\label{sharpsfromprecidealprop}
Suppose there is a precipitous ideal on $\omega_1$. Then $\mathcal P(\omega_1)$ is closed under $A\mapsto A^\sharp$.
\end{prop}
\begin{proof}
It is easy to see that $\mathbb R$ is closed under $x\mapsto x^\sharp$. Let $I$ be a precipitous ideal and let $j\colon V\rightarrow M=\Ult(V, g)$ be the generic ultrapower of $V$ in the extension $V[g]$, $g$ generic for $I^+$. Then $A=j(A)\cap\omega_1^V\in M$ and is coded by a real in $M$. By elementarity, $\mathbb R\cap M$ is closed under $x\mapsto x^\sharp$. Thus $A^\sharp$ exists in $M\subseteq V[g]$. As forcing cannot add a sharp, $A^\sharp\in V$.
\end{proof}

\begin{lemm}\label{QmaxdenseinQmaxmlemm}
Assume $\AD$ in $\LofR$. The inclusion $\Qmax\hookrightarrow\Qmaxm$ is a dense embedding.
\end{lemm}

\begin{proof}
It is easy to see that if $p, q\in\Qmax$ then 
$$q<_{\Qmax} p\Leftrightarrow q<_{\Qmaxm} p.$$
Now let $p\in\Qmaxm$ and find $x$ a real coding $p$. Our assumptions imply by Woodin's analysis of $\Qmax$ under $\AD^{\LofR}$ that there is $q=(M, J, h)\in\Qmax$ with $x^\sharp\in M$. By Proposition \ref{sharpsfromprecidealprop},
$$M\models``\Pomo\text{ is closed under }A\mapsto A^\sharp".$$
Thus we may apply Lemma \ref{strongdiasharpsprop} inside $M$ and find an iteration 
$$j\colon p\rightarrow p^\ast=(N^\ast, I^\ast, f^\ast)$$
so that 
$$q'\coloneqq (M, J, f^\ast)\in \Qmax$$
and $j$ witnesses $q'<_{\Qmaxm}p$.
\end{proof}

It is not obvious how to even prove construct a single $\Qmax$-condition assuming only $\AD^{\LofR}$. Woodin worked with a variant $\Qmaxast$ of $\Qmax$ instead to analyze the $\Qmax$-extension of $\LofR$. We remark that this can be done with $\Qmaxm$ as well. The arguments are, modulo Lemma \ref{Qmaxmanalysislemm}, quite similar to the arguments in the $\Qmaxast$ analysis.

\section{Consistency of $\QMM$ and forcing ``$\NS$ is $\omega_1$-dense"}

We are now in position to force $\QMM$ and force ``$\NS$ is $\omega_1$-dense".

We can now finally find a forcing which freezes $\NS$ along $f$ assuming large cardinals and that $f$ witnesses $\diacol$.

We will also reap what we have sown by replacing $\Qmax$ with $\Qmaxm$.

\begin{proof}[Proof of Lemma \ref{freezinglemm}]
Use the Woodin cardinal to make $\NS$ saturated while turning $f$ into a witness of $\diacolp$ by $f$-semiproper forcing in a generic extension $V[g]$ using the iteration theorem \ref{fsemiproperiterationtheoremotherpaper}. Shelah's construction to make $\NS$ saturated works just as well in this context. Observe that 
$$(\Htwo,\NS, f)^{V[g]}$$
is a almost a $\Qmaxm$-condition in $V[g]$. Work in $V[g]$. Next we want to apply Theorem \ref{forcingexistencelemm} with $\Vmax=\Qmaxm$ for the dense set $D=\Qmaxm$. Note that the universe is closed under $X\mapsto X^\sharp$ and as $D$ is $\Pi^1_2$, $D$ is $\infty$-universally Baire. We cannot guarantee full generic absoluteness for small forcings, however we actually only need that for any forcing $\PP$ of size $\leq 2^{\omega_2}$ we have that
\begin{enumerate}[label=$(\roman*)$]
    \item\label{freezingcond1} $(\Qmaxm)^{V[g]^\PP}\cap V[g]=(\Qmaxm)^{V[g]}$ and
    \item\label{freezingcond2} $(\Qmaxm)^{V[g]^{\PP}}$ is a $\Pmaxvariant$ in $V[g]^{\PP}$
\end{enumerate}
\ref{freezingcond1} is again guaranteed by the closure under $X\mapsto X^\sharp$.  The only nontrivial thing one has to verify for \ref{freezingcond2} is that $\Qmaxm$ has no minimal conditions in $V[g]^{\PP}$. This follows from the closure of $\mathbb R$ under $x\mapsto M_1^\sharp(x)$.\\
Thus $\PPdia=\PPdia(\Qmaxm, f, \Qmaxm)$ exists and in a further extension $V[g][h]$ by $\PPdia$ we have: 

\begin{tikzpicture}
\def\x{2.1};
\def\y{0.5};

\node (pT) at (0*\x,  0*\y){$\Qmaxm$};
\node (N) at (0*\x, -2*\y) {$q_0$};
\node (Nast) at (2*\x, -2*\y) {$q_{\omega_1}=(N_{\omega_1}, I_{\omega_1}, f)$};
\node (M0) at (-2*\x, -4*\y) {$p_0$};
\node (MN) at (0*\x, -4*\y) {$p_{\omega_1^{q_0}}$};
\node (M3) at (2*\x, -4*\y) {$p_{\omega_1}$};
\node (H) at (2*\x, -6*\y) {$((H_{\omega_2})^{V[g]}, \NS^{V[g]},f)$};
\node (Vmax) at (-2*\x, -6*\y) {$\Qmaxm$};

\path (N)--(pT) node[midway, rotate=90]{$\in$};

\draw[->] (N)--(Nast) node[midway, above]{$\sigma_{0,\omega_1}$};

\path (MN)--(N) node[midway, rotate=90]{$\in$};
\path (M3)--(Nast) node[midway, rotate=90]{$\in$};

\draw[->] (M0)--(MN) node[midway, above]{$\mu_{0, \omega_1^{q_0}}$};
\draw[->] (MN)--(M3) node[midway, above]{$\mu_{\omega_1^{q_0}, \omega_1}$};

\path (H)--(M3) node[midway, rotate=90]{$=$};

\path (Vmax)--(M0) node[midway, rotate=270]{$\in$};

\end{tikzpicture}

So that
\begin{enumerate}[label=$(\PPdia.\roman*)$]
\item $\mu_{0, \omega_1}, \sigma_{0,\omega_1}$ are generic iterations of $p_0$, $q_0$ respectively,
\item\label{qmaxmsmallerppdiacond} $\mu_{0,\omega_1^{q_0}}$ witnesses $q_0<_{\Qmaxm} p_0$,
\item\label{qmaxmstretchcond} $\mu_{0,\omega_1}=\sigma_{0,\omega_1}(\mu_{0,\omega_1^{q_0}})$ and
\item\label{qmaxmidealsmatchppdiacond} the generic iteration $\sigma_{0, \omega_1}\colon q_0\rightarrow q_{\omega_1}$ is a $\diamondsuit$-iteration.
\end{enumerate} 
\begin{claim}
$f$ witnesses $\diacol$ in $V[g][h]$.
\end{claim}
\begin{proof}
By Lemma \ref{diamondsuititerationdiablemm} and \ref{qmaxmidealsmatchppdiacond}, $I_{\omega_1}=\NSf^{V[g][h]}\cap N_{\omega_1}$, in particular $f$ witnesses $\diacol$ in $V[g][h]$.
\end{proof}
It remains to show that the extension $V\subseteq V[g][h]$ has ``frozen $\NS^V$ along $f$". Let $S\in\Pomo^V$. It follows from \ref{qmaxmsmallerppdiacond}, \ref{qmaxmstretchcond} and the definition of $<_{\Qmaxm}$ (especially \ref{Qmaxmorderkillcond}) that one of the following holds:
\begin{itemize}
    \item Either $S\in I_{\omega_1}$,
    \item or for some $p\in\Col(\omega,\omega_1)$ we have $S^f_p\subseteq S\mod I_{\omega_1}$.
\end{itemize}
As any $\diamondsuit$-iteration is correct, $I_{\omega_1}=\NS^{V[g][h]}\cap N_{\omega_1}$. It follows that
\begin{itemize}
    \item either $S\in\NS^{V[g][h]}$,
    \item or for some $p\in\Col(\omega,\omega_1)$ we have $S^f_p\subseteq S\mod \NS^{V[g][h]}$,
\end{itemize}
which is what we had to show.
\end{proof}
\begin{rem}
Instead of closure of $V$ under $X\mapsto M_1^{\sharp}$ we could just as well have assumed that there is a second Woodin cardinal with a measurable above. 
\end{rem}

\begin{thm}\label{getQMMthm}
Suppose $f$ witnesses $\diacol$ and there is a supercompact limit of supercompact cardinals. Then there is a $f$-preserving forcing extension in which $f$ witnesses $\QMM$. 
\end{thm}

\begin{proof}
Let $\kappa$ be a supercompact limit of supercompact cardinals and 
$$L\colon V_\kappa\rightarrow V_\kappa$$
an associated Laver function. We describe a $Q$-iteration w.r.t. $f$
$$\PP=\langle \PP_\alpha, \Q_\beta\mid\alpha\leq\kappa,\beta<\kappa\rangle$$
that forces $\QMM$. For any $\alpha<\kappa$, $\Q_\alpha$ is a two step-iteration of the form 
$$\Q_\alpha=\Q_\alpha^0\ast\ddot{\mathbb Q}^1_\alpha$$
with $\vert \Q_\alpha\vert<\kappa$. If $\alpha$ is a successor (or $0$) then
\begin{enumerate}[label=\rn*]
\item $\Q_\alpha^0$ is forced to be a $f$-preserving forcing that freezes $\NS$ along $f$ and
\item $\ddot{\mathbb Q}^1_\alpha$ is a name for a $f$-preserving partial order forcing $\SRP$.
\end{enumerate}

Note that $\Q_\alpha^0$ exists by Lemma \ref{freezinglemm} and $\ddot{\mathbb Q}^1_\alpha$ exists by Corollary \ref{getSRPcor}.\\

If $\alpha$ is a limit ordinal, then 
\begin{enumerate}[label=\rn*]
\item $\Q_\alpha^0$ is $L(\alpha)$ if that is a $\PP_\alpha$-name for a $f$-preserving forcing and the trivial forcing else,
\item $\ddot{\mathbb Q}^1_\alpha$ is as in the successor case.
\end{enumerate}

It is clear that this constitutes a $Q$-iteration and hence $\PP$ preserves $f$ and in particular $\omega_1$ is not collapsed. $\PP$ is $\kappa$-c.c..  As we use $f$-preserving forcings guessed by $L$ at limit steps, $\QMM$ holds in the extension as witnessed by $f$ by the usual argument.
\end{proof}

If one is only interested in forcing ``$\NS$ is $\omega_1$-dense", a slightly weaker large cardinal assumption is sufficient.

\begin{thm}\label{firstnsdenseforcingthm}
Suppose $f$ witnesses $\diacol$ and $\kappa$ is an inaccessible limit of ${<}\kappa$-supercompact cardinals. Then there is a $f$-preserving forcing extension in which $\NS$ is $\omega_1$-dense.
\end{thm}

\begin{proof}
Indeed any nice iteration 
$$\PP=\langle \PP_\alpha,\Q_\beta\mid\alpha\leq\kappa, \beta<\kappa\rangle$$
so that for all $\gamma<\kappa$
$$V_\kappa\models``\PP_\gamma\text{ is a }Q\text{-iteration w.r.t. }f"$$
preserves $f$ and forces ``$\NS$ is $\omega_1$-dense". To see this, first of all note that $\PP$ is $\kappa$-c.c.~by Fact \ref{niceccfact}. Now any $\PP_\gamma$ for $\gamma<\kappa$ preserves $f$ by Theorem \ref{Qiterationthm} applied in $V_\kappa$ and it follows immediately that $\PP$ preserves $f$. Suppose now that $G$ is $\PP$-generic and 
$$V[G]\models S\in\NS^+.$$
There must be some nonlimit $\gamma<\kappa$ with $S\in V[G_\gamma]$. As $\Q_\gamma^{G_\gamma}$ freezes $\NS$ along $f$ in $V[G_{\gamma}]$, there must be some $b\in\Col(\omega,\omega_1)$ with $S^f_b\subseteq S\mod\NS$ in $V[G_{\gamma+1}]$, hence in $V[G]$.
\end{proof}

Neither of these results answers the original question, as Woodin asks specifically for a semiproper forcing, but $Q$-iterations are not stationary set preserving if $\NS$ is not $\omega_1$-dense to begin with. However, we have one more trick up our sleeves: For once we will pick $f$ more carefully.

\begin{lemm}\label{nicediacolwitnesslemm}
Suppose $\vec S=\langle S_\alpha\mid\alpha<\omega_1\rangle$ is a sequence of pairwise disjoint stationary sets in $\omega_1$ and $\diamondsuit(S_\alpha)$ holds for all $\alpha<\omega_1$. Then there is $f$ witnessing $\diacol$ so that for all $\alpha<\omega_1$, there is $p\in\Col(\omega, \omega_1)$ with $S^f_p\subseteq S_\alpha$. 
\end{lemm}

\begin{proof}
From $\diamondsuit(S_\alpha)$, we get a witness $f_\alpha$ of $\diacol$ so that $f_\alpha(\beta)$ is the trivial filter if $\beta\notin S$. Let $\langle b_\alpha\mid\alpha<\omega_1\rangle$ be an enumeration of some maximal antichain in $\Col(\omega,\omega_1)$ of size $\aleph_1$. Now define $f\colon \omega_1\rightarrow H_{\omega_1}$ as follows: For $\beta\in S_\alpha$ we let
$$f(\beta)=\{p\in\Col(\omega,\beta)\mid\exists p'\leq p\exists q\in f_\alpha(\beta)\ p'\leq {b_\alpha}^\frown q\}.$$
Note that there is at most one $\alpha$ with $\beta\in S_\alpha$. If $\beta$ is not in any $S_\alpha$, let $f(\beta)$ be the trivial filter. It is now clear that $S^f_{b_\alpha}\subseteq S_\alpha$, but we still need to verify that $f$ indeed witnesses $\diacol$. So let $p\in\Col(\omega,\omega_1)$ and
$$\vec D=\langle D_\alpha\mid\alpha<\omega_1\rangle$$
be a sequence of dense subsets of $\Col(\omega,\omega_1)$. We have that show that 
$$\{\beta<\omega_1\mid p\in f(\beta)\wedge \forall \gamma<\beta\ f(\beta)\cap D_\gamma\neq\emptyset\}$$
is stationary. So let $C$ be a club in $\omega_1$. Find $\alpha$ so that $b_\alpha$ is compatible with $p$ and note that we may assume further that $p\leq b_\alpha$. Hence we can write $p$ as $p={b_\alpha}^\frown q$. For $\gamma<\omega_1$, let 
$$D'_\gamma=\{r\in\Col(\omega,\omega_1)\mid {b_\alpha}^\frown r\in D_\gamma\}$$
and note that $D'_\gamma$ is dense.  As $f_{\alpha}$ witnesses $\diacol$, we may find $\beta\in C$ large enough so that
\begin{enumerate}[label=$(\beta.\roman*)$]
\item $p\in\Col(\omega,\beta)$,
\item $q\in f_{\alpha}(\beta)$ and
\item $\forall \gamma<\beta\ f_{\alpha}(\beta)\cap D^\prime_\gamma\neq\emptyset$.
\end{enumerate}
It follows that $p\in f(\beta)$ and that 
$$\forall \gamma<\beta\ f(\beta)\cap D_\gamma\neq\emptyset.$$
\end{proof}

\begin{cor}
Assume there is a supercompact limit of supercompact cardinals. Then there is a semiproper forcing $\PP$ with $V^\PP\models\QMM$.
\end{cor}

\begin{proof}
By otherwise taking advantage of the least supercompact, we may assume all stationary-set preserving forcings are semiproper. Next, we force with 

$$\PP_0=\Col(\omega_1, 2^{\omega_1}).$$
Let $G$ be $\PP_0$-generic over $V$. There is then a partition $\langle T_\alpha\mid\alpha<\omega_1\rangle$ of $\omega_1$ into stationary sets so that whenever $S\in V$ is stationary in $\omega_1$, then $T_\alpha\cap S$ is stationary for all $\alpha<\omega_1$. Also, there is an enumeration 
$$\langle S_\alpha\mid\alpha<\omega_1\rangle$$
of all stationary sets in $V$. Now in $V[G]$,
$$\langle S_\alpha\cap T_\alpha\mid\alpha<\omega_1\rangle$$
is a sequence of pairwise disjoint stationary sets. Moreover, $\diamondsuit_T$ holds for any stationary $T\subseteq\omega_1$. By Lemma \ref{nicediacolwitnesslemm}, there is a witness $f$ of $\diacol$ so that for any $\alpha<\omega_1$ there is $p\in\Col(\omega,\omega_1)$ with $S^f_p\subseteq (S_\alpha\cap T_\alpha)$. Thus for any stationary $S\in V$, $S$ contains some $S^f_p$. Note that any further $f$-preserving forcing preserves the stationarity of any $S^f_p$ and hence does not kill any stationary $S\in V$. By Theorem \ref{getQMMthm}, there is an $f$-preserving $\PP_1$ that forces $\QMM$. It follows that back in $V$, the two-step forcing $\PP_0\ast\dot\PP_1$ preserves stationary sets, hence is semiproper, and forces $\QMM$.
\end{proof}

Similarly, can prove the following from Theorem \ref{firstnsdenseforcingthm}.

\begin{cor}
Assume there is an inaccessible $\kappa$ that is a limit of ${<}\kappa$-supercompact cardinals. Then there is a stationary set preserving forcing $\PP$ with 
$$V^\PP\models``\NS\text{ is }\omega_1\text{-dense}".$$
\end{cor}
Assuming one more (sufficiently past $\kappa$-) supercompact cardinal below $\kappa$, one can replace stationary set preserving forcing by semiproper forcing.\medskip

So the answer to Woodin's question is yes assuming sufficiently large cardinals.

\subsection{$\QMM$ implies $\stargen{\Qmax}$}
We apply the Blueprint Theorems to show that the relation between $\QM$ and $\stargen{\Qmax}$ is analogous to the one of $\MMpp$ and $\Wstar$.\\
Typicality of $\Qmax$ is witnessed by $\Psi^{\Qmax}$ consisting of the formulae
\begin{itemize}
    \item $\psi^{\Qmax}_0(x)=``x\in\dot I"$,
    \item $\psi^{\Qmax}_1(x)=``x=\dot f"$ and
    \item $\psi^{\Qmax}_2(x)=``x=\dot f\wedge x\text{ witnesses }\diacol"$.
\end{itemize}
Note that $\psi^{\Qmax}_2(x)$ is (in context equivalent to) a $\Pi_1$-formula.
\begin{thm}\label{QMMimpliesQmaxstarthm}
$\QM$ implies $\stargen{\Qmax}$.
\end{thm}

\begin{proof}
Suppose $f$ witnesses $\QM$. We already mentioned that forcing an instance of $\SRP$ is $f$-preserving and so $\SRP$ holds. $\mathcal H_f$ is almost a $\Qmax$-condition by Lemma \ref{NSdensefromQMlemm}. $\Qmax$ accepts $\diamondsuit$-iterations by Lemma \ref{diamondsuititerationdiablemm}. $\stargen{\Qmax}$ now follows from the First Blueprint Theorem \ref{blueprintthm}.
\end{proof}

\begin{defn}
For $\Delta\subseteq\mathcal P(\mathbb R)$,  $\Delta$-$\BQM$ states that there is $f$ witnessing $\diacol$ so that 
$$\Delta\text{-}\BFA(\{\PP\mid\PP\text{ preserves }f\})$$
holds.
\end{defn}
We mention that already $\BQM=\emptyset$-$\BQM$\index{BQM@$\BQM$} is enough to prove ``$\NS$ is $\omega_1$-dense".

\medskip

Finally, we remark that one can show that fragments of $\QMM$ hold in $\Qmax$-extensions of canonical models of determinacy. For example $\QMM(\mathfrak{c})$, i.e.~$\QMM$ for forcings of size at most continuum, holds in the $\Qmax$-extension of models of $\ADR+``\Theta\text{ is regular}"+V=L(\mathcal P(\mathbb R))$ and $\BQMM$ holds in the $\Qmax$-extension of suitable $\mathbb R$-mice.\\
Finally we want to mention that Woodin has formulated a forcing axiom $\mathrm{FA}(\diacol)[\mathfrak{c}]$ somewhat similar to $\QMM(\mathfrak{c})$ and has proven that it holds in the $\Qmax$-extension of a model of $\ADR+``\Theta\text{ is regular}"+V=L(\mathcal{P}(\mathbb R))$, see Theorem 9.54 in \cite{woodinbook}\footnote{We remark once again that Woodin has defined $\diacol$ slightly different than we have here.} The global version $\mathrm{FA}(\diacol)$ of Woodin's axiom does \textit{not} imply ``$\NS$ is $\omega_1$-dense". The reason is that if $f$ witnesses $\diacol$ and $\MMpp(f)$ holds then $\mathrm{FA}(\diacol)$ is true, however $\NS$ is not $\omega_1$-dense.

\bibliographystyle{alpha}
\bibliography{bib}

\end{document}